%% file: main.tex
\documentclass[11pt]{amsart}
\input{commands.tex}

\usepackage{dynkin-diagrams}

\usetikzlibrary{decorations.pathreplacing}

\usepackage{csvsimple}

\usepackage{dsfont}

\def\cE{\mathcal{E}}
\def\cX{\mathcal{X}}
\def\cM{\mathcal{M}}
\def\cS{\mathcal{S}}
\def\cL{\mathcal{L}}
\def\cSopen{\mathcal{S}^\circ}
\def\cB{\mathcal{B}}
\def\cO{\mathcal{O}}
\def\Aut{\mathrm{Aut}}
\def\Stab{\mathrm{Stab}}
\def\Mon{\mathrm{Mon}}

\def\C{\mathbb{C}}
\def\A{\mathbb{A}}
\def\P{\mathbb{P}}
\def\Spec{\mathrm{Spec}\ \!}
\def\GL{\mathrm{GL}}
\def\PGL{\mathrm{PGL}}
\def\cY{\mathcal{Y}}
\def\cZ{\mathcal{Z}}
\def\cA{\mathcal{A}}
\def\cW{\mathcal{W}}
\def\cD{\mathcal{D}}
\def\cG{\mathcal{G}}
\def\cI{\mathcal{I}}
\def\cF{\mathcal{F}}
\def\cU{\mathcal{U}}

\def\BPQ{\mathrm{Bitan}_{2,4}}
\def\Lines{\mathrm{Lines}_{3,3}}

\providecommand{\ASL}{\mathrm{ASL}}

\providecommand{\GO}{\mathrm{GO}}
\providecommand{\PGO}{\mathrm{PGO}}

\providecommand{\Mongeom}{\text{Mon}^{\text{geom}}}
\providecommand{\Monmarked}{\text{Mon}^{\text{mark}}}
\providecommand{\Monpar}{\text{Mon}^{\text{par}}}
\providecommand{\Monstack}{\text{Mon}^{\text{stack}}}

\newcommand{\gitquot}{/\!\!/}

\newcommand{\LinesProblem}{\mathfrak{L}}
\newcommand{\BitangentsProblem}{\mathfrak{B}}

\title{Symmetry and solvability: Galois groups of equivariant enumerative problems}
\author{Thomas Brazelton}
\author{Alberto Landi}
\author{Sidhanth Raman}
\date{\today}
\begin{document}

\begin{abstract}
At the genesis of Galois theory at the end of the 19th century, enumerative geometers were concerned with the \emph{solvability} of various problems: for instance whether one can express lines on a cubic surface in radicals in terms of the coefficients that define it. By work of Hermite and a celebrated theorem of Harris, this is equivalent to determining whether the monodromy group of a certain finite \'etale cover is solvable. This problem, as well as the related problem of computing bitangents to plane quartics, are both unsolvable. When one restricts to the locus of $G$-symmetric cubic surfaces, monodromy will drop and can become solvable. In this setting there are four flavors of monodromy one can consider: the one from restricting the classical cover, the one over the GIT quotient after modding out by projective transformations, and two related notions of ``monodromy'' coming from the stacky cover of $G$-symmetric lines on $G$-symmetric cubic surfaces --- unlike in the non-symmetric setting these are all different. In this paper we characterize the relationships between all four notions of monodromy in the general setting of quotient stacks, and use these connections to compute all these monodromy groups for all 11 possible automorphism groups of smooth cubic surfaces, as well as for all 12 possible automorphism groups for smooth planar quartics. We demonstrate that, in the presence of \emph{any symmetry} whatsoever, the problem of solving for lines on cubic surfaces or bitangents to plane quartics is solvable.
\end{abstract}

\maketitle

\section{Introduction}

In Chapter III of Jordan's \emph{Trait\'e des substitutions}, the first textbook account of Galois theory, he investigates the ``geometric applications'' of Galois theory \cite{jordanTraiteSubstitutions1870}. The leap to this idea is not so conceptually difficult---rather than considering the Galois group of one polynomial in one variable (as in field theory), we investigate the Galois group of a system of polynomials in many variables. When the system of equations arises in nature from geometry, the resulting Galois group is very rich and can encode beautiful structure about the underlying geometry.

As early as the 1850's, it was known that the Galois group of a geometric problem should be related to the monodromy group of the associated covering space obtained by taking complex points \cite{hermiteFonctionsAlgebriques1851}. This relationship, that ``Galois equals monodromy", is now a celebrated theorem of Harris \cite{Harris-Galois} and has seen many applications in enumerative geometry (see \cite{sottileGaloisGroupsEnumerative2021} for a lovely introduction to these ideas). In this paper we explore two related problems: solving for the 27 lines on smooth cubic surfaces (which has Galois group $W(E_6)$), and solving for the 28 bitangents to smooth planar quartic curves (which has Galois group $W^+(E_7)$). Both of these groups are unsolvable, which implies there exists no formula in radicals for the general lines to a smooth cubic surface in terms of the coefficients of its defining equation, and similarly for bitangents.

It is a theorem of the first named author that problems such as these admit an equivariant analogue of Schubert's conservation of number \cite{brazeltonEquivariantEnumerativeGeometry2025} --- for instance if $G$ is the automorphism group of a smooth cubic surface over the complex numbers, then any two $G$-symmetric cubic surfaces admit the same $G$-symmetry on their 27 lines. It is natural then to ask: can we solve for the lines on $G$-symmetric cubic surfaces in radicals? The first and third named author answered this question in the affirmative in the case of $G=S_4$, using techniques from Hodge theory and hyperbolic geometry \cite{brazeltonMonodromySpaceSymmetric2025}. This computation was numerically certified by Duff and Lee in \cite{duff2026certifyinggaloismonodromyactionshomotopy}. Using entirely different techniques coming from the theory of fundamental groups of stacks \cite{noohiFundamentalGroupsAlgebraic2004}, the second-named author computed the Galois groups of lines on cubic surfaces for the groups $G = C_2$, $S_3$, $S_3\times C_2$, and $S_4$ \cite{landiStacksMonodromySymmetric2025}. For $G$ a subgroup of $S_5$, Pichon-Pharabod and Telen numerically certified the monodromy groups of lines on $G$-symmetric cubic surfaces by certifying the monodromy action on homology in \cite{pichon-pharabodGaloisGroupsSymmetric2025}.

In this paper we introduce four related notions of monodromy that one could feasibly investigate for equivariant enumerative geometry problems: these are \emph{parameter-level monodromy, geometric monodromy, stacky monodromy}, and \emph{marked stacky monodromy}. The parameter-level monodromy is the classical one which would be recognizable to Hermite or Jordan, the geometric one takes place over the GIT quotient by projective symmetries, and the stacky ones are defined as in \cite{landiStacksMonodromySymmetric2025}. While these all agree in the absence of symmetry, they become quite different once we restrict our attention to $G$-symmetric objects --- in \Cref{sec:four-flavors}, we characterize the relationships between these different notions of monodromy. 
While the stacky monodromy group determines the others, in practice it is often necessary to compute the parameter-level or geometric monodromy and deduce the stacky monodromy from it.
We compute all of these monodromy groups for the 11 possible automorphism groups of smooth cubic surfaces and the 12 possible automorphism groups of smooth planar quartics.

\begin{introtheorem}[{As~\Cref{thm:main}}]
Let $G$ be an automorphism group of a smooth cubic surface over the complex numbers. Then the parameter-level, geometric, stacky, and marked stacky monodromy groups of the 27 lines on $G$-symmetric cubic surfaces are as follows:
\input{TABLE-cubics}
\end{introtheorem}

\begin{introtheorem}[{As~\Cref{thm:main2}}] Let $G$ be an automorphism group of a smooth planar quartic curve over the complex numbers. Then the parameter-level, geometric, stacky, and marked stacky monodromy groups of the 28 bitangents to $G$-symmetric planar quartics are as follows:\footnote{See \Cref{rmk:table-notation} for an explanation of the notation used here.}
\input{TABLE-quartics}
\end{introtheorem}

In problems such as these there is an expected answer that the parameter-level monodromy is the centralizer of the group $G$ in the ambient Galois group of the problem, and the stacky monodromy is its normalizer. Interestingly this fails in a handful of cases in the above theorems.

With these computations and a bit more work we can conclude that:
\begin{corollary}[{See~\Cref{sec:radicals}}]
Lines on symmetric cubic surfaces (with any symmetry) admit formulas in radicals in terms of the coefficients defining the cubic surface. Similarly bitangents to symmetric planar quartics (with any symmetry) admit formulas in radicals in terms of the coefficients defining the quartic.
\end{corollary}

\subsection*{Related work} We would like to thank Igor Dolgachev for pointing us to the work of Naruki \cite{Naruki}, from which one can solve in radicals for lines on $C_2$-symmetric cubic surfaces in terms of the coordinates of the marked moduli space. Also related in spirit is \cite{Minahan} where we can solve for all 27 lines given the input of three skew lines.

\subsection{Overview}
In \Cref{sec:preliminaries} we recap the definitions of monodromy and Galois groups of covers of varieties, before introducing stacky monodromy and some basic techniques used to compute it. In \Cref{sec:four-flavors} we define the four kinds of monodromy groups under consideration for a general ``equivariant enumerative problem'' and discuss how they relate to one another. In \Cref{sec:moduli-spaces} we define the moduli space and stack of smooth degree $d$ hypersurfaces in $\P^n$, we show that the moduli spaces of $G$-symmetric smooth cubic surfaces (resp. $G$-symmetric smooth planar quartics) are irreducible, and we compute the group of projective symmetries which centralize and normalize $G$ in each case. In \Cref{sec:galois-groups-of-lines} we compute the monodromy groups of lines on $G$-symmetric cubic surfaces. In \Cref{sec:correspondence-thm} we reframe the classical correspondence between the 27 lines on a cubic surface and the 28 bitangents to a plane quartic in the language of stacks, and investigate how symmetry interplays with this correspondence. This is used in \Cref{sec:galois-groups-of-bitangents} to compute the monodromy groups of bitangents to $G$-symmetric plane quartics. Finally we discuss solvability in radicals of the rational Galois groups in \Cref{sec:radicals}.

\subsection{Notation and conventions} We are always working over a field of characteristic zero, and when we need it to be algebraically closed we will state this explicitly. All of our stacks will be assumed to be algebraic.

\textbf{Covers:} A ``covering space'' or ``cover'' of algebraic stacks is in the sense of \cite{noohiFundamentalGroupsAlgebraic2004}, meaning a representable finite \'etale morphism. A cover of schemes means a finite \'etale morphism of schemes.

\textbf{Algebraic groups:} Let $P$ be an algebraic group over our base field $k$. We denote by $P^\circ$ the schematic connected component of the identity (\cite[Tag 0B7R]{Stacks}), and by $\pi_0(P)=P/P^\circ$ the quotient.

\textbf{Notation:} In this paper, $f\colon V \to U$ will generally denote a covering space of schemes, $\mathcal{F}\colon \mathcal{Y} \to \mathcal{X}$ will be a covering space of algebraic stacks, and $F \colon Y \to X$ the associated covering space of the coarse moduli. We denote by $\pi_{\mathcal{X}} \colon \mathcal{X} \to X$ the structure map to the coarse moduli space.

\subsection{Acknowledgements} The first-named author would like to thank Kisun Lee for help with certified tracking software in numerical algebraic geometry which was tremendously helpful in catching errors in parameter-level monodromy groups in an early draft. TB was partially supported by NSF Grant No. DMS-2303242. AL was partially supported by NSF Grant No. DMS-2401358. SR was partially supported by NSF Grant No. DMS-2503485.

\subsection{Code availability} The groups $\Monpar$, $\Monstack$, and $\Monmarked$ are all subgroups of the ambient Galois group $W(E_6)$ or $W^+(E_7)$ and are therefore accessible as finite groups, which have been implemented in GAP \cite{GAP4}. These are publicly available at the following url, where various group theory assertions in the paper are verified:
\begin{center}
    \url{https://github.com/tbrazel/symmetry_solvability_code}
\end{center}

\subsection{AI disclosure} The first-named author used AI to design the architecture for the repository linked above, to sanity check computations done by hand in \Cref{thm:irred-UG} and \Cref{thm:irred-VG}, and to locate small errors throughout a final draft of the paper.

\setcounter{tocdepth}{1}
\tableofcontents{}

\section{Covers of stacks}\label{sec:preliminaries}

\subsection{Galois groups of schemes and monodromy}

Attached to a suitable ``cover'' of schemes, we can associate a Galois group. In this work we are always in characteristic zero, however we remark that the general theory of Galois categories \cite[V.4]{grothendieckRevetementsEtalesGroupe1971}, \cite[Tag 0BQ6]{Stacks}, holds in much more generality. We refer the reader to \cite{szamuelyGaloisGroupsFundamental2009} for a lovely overview of these ideas, which we assert mostly without proof.

\begin{theorem} Let $\pi \colon Y \to X$ be a finite \'etale morphism of integral complex varieties. Then the Galois group of the cover (in the sense of \cite{grothendieckRevetementsEtalesGroupe1971}) is isomorphic to the Galois group of the normal closure of $\C(Y)$ over $\C(X)$, which furthermore is isomorphic to the monodromy group of the covering space $Y(\C) \to X(\C)$.
\end{theorem}

This can be extended to generically finite dominant morphisms by \cite{Harris-Galois}, however an extra assumption is needed. We can only define monodromy for such a cover by restricting to an affine open over which the cover is unbranched. In order to guarantee this is well-defined, we need to require $X$ to be normal, or to require both $X$ and $Y$ to be integral. We therefore have:

\begin{theorem}[{\cite{Harris-Galois}}] Let $\pi \colon Y \to X$ be a generically finite and dominant morphism between integral complex varieties. Then the Galois group and monodromy group are isomorphic as permutation groups of a generic fiber.
\end{theorem}

As a mild generalization, we do not need to require $Y$ to be connected.

\begin{definition} Let $Y \to X$ be a generically finite and dominant morphism, where $X$ is integral and normal. We define the \emph{normal core cover} $\widetilde{Y} \to X$ to be the smallest cover of $X$ which is Galois and refines each of the irreducible components of $Y$.
\end{definition}

We can see (either by the theory of Galois categories or by direct inspection), that even when $Y$ is disconnected, the Galois group of the normal core cover coincides with the monodromy group of $Y(\C) \to X(\C)$. This is an essential observation needed to work with equivariant enumerative geometry problems, since the covers under consideration are generally disconnected.

\subsection{Stacky monodromy groups}

Here we recap the key features of fundamental groups and monodromy groups of covers of pointed stacks, as defined in \cite{landiStacksMonodromySymmetric2025} following~\cite{noohiFundamentalGroupsAlgebraic2004}. Let us first say what we mean by a stack being pointed.

\begin{definition}[{\cite[Definition 3.3]{noohiFundamentalGroupsAlgebraic2004}}]\label{def:pointed-stack}
    A \emph{pointed algebraic stack} is a pair $(\cX,x)$, where $\cX$ is an algebraic stack and $x\colon\Spec K\to\cX$ is a \emph{geometric point}, i.e., $K$ is an algebraically closed field. We will often write simply $x\in\cX$, suppressing the corresponding morphism $\Spec K\to\cX$.

    A \emph{morphism of pointed stacks} (or \emph{pointed morphism}) is a pair $(f,\phi)\colon(\cY,y)\rightarrow(\cX,x)$ where $f\colon\cY\rightarrow\cX$ is a morphism and $\phi\colon x\rightarrow f(y)$ is a natural transformation of functors. We will always omit $\phi$ in the notation, as it will not play a direct role in the paper.
\end{definition}
The above definition is rather technical, but for practical purposes one can think of pointed morphisms simply as morphisms which preserve base points.

Associated to every pointed connected algebraic stack $(\mathcal{X},x)$ is a Galois category of finite \'etale covers $\FEt_{(\mathcal{X},x)}$. Mirroring the construction of the \'etale fundamental group, Noohi constructs the \emph{fundamental group} of the pointed stack $(\mathcal{X},x)$ \cite{noohiFundamentalGroupsAlgebraic2004}, which is profinite.

We will denote the fundamental group of a pointed stack $(\cX,x)$ by $\pi_1{(\cX,x)}$. Different choices of base points in the same connected component yield isomorphic fundamental groups, thus we will often suppress the notation of the base point, especially when $\cX$ is connected.

Since an extensive treatment of the theory of fundamental groups is beyond the scope of the paper, we limit ourselves to reminding the main feature of stacky fundamental groups compared to schemes, before moving on to monodromy groups. As usual, we refer to~\cite{noohiFundamentalGroupsAlgebraic2004,landiStacksMonodromySymmetric2025} for details.

\subsubsection{Stabilizers and hidden loops}

We start by recalling the definition of the stabilizer of a geometric point. Since we will need it in the future, we also introduce the inertia stack.

\begin{definition}\label{def:inertia-stack-geometric-stabilizer}
    Let $\cX$ be an algebraic stack. The \emph{inertia stack} $\cI_{\cX}$ of $\cX$ is the fiber product
    \[
    \begin{tikzcd}
        \cI_{\cX}\arrow[r]\arrow[d] & \cX\arrow[d,"\Delta"]\\
        \cX\arrow[r,"\Delta"] & \cX\times\cX
    \end{tikzcd}
    \]
    where $\Delta\colon\cX\rightarrow\cX\times\cX$ is the diagonal morphism.
    
    Given a scheme $S$ and a morphism $s\colon S\rightarrow\cX$, equivalently an object $s\in\cX(S)$, the \emph{stabilizer} or \emph{automorphism group space} of $\cX$ at $s$ is the fiber product
    \[
    \begin{tikzcd}
        \underline{\Aut}_{\cX}(s)\arrow[d]\arrow[r] & S\arrow[d,"{(s,s)}"]\\
        \cX\arrow[r,"\Delta"] & \cX\times\cX.
    \end{tikzcd}
    \]
    When the object is a geometric point $x\colon \Spec K\rightarrow\cX$, so with $K$ an algebraically closed field, we often drop the subscript $\cX$, and denote by $\Aut(x)$ the group of $K$-points of $\underline{\Aut}_{\cX}(x)$.
\end{definition}

\begin{remark}\label{rmk:inertia-stack}
    Note that $\underline{\Aut}_{\cX}(s)\cong\cI_{\cX}\times_{\cX,s}S$. In particular, $\cI_{\cX}\rightarrow\cX$ is the `universal' stabilizer group space of $\cX$, and the fiber over a point $x\in\cX$ is its stabilizer.
\end{remark}

Let $(\cX,x)$ be a pointed algebraic stack. The key feature of Noohi's theory of stacky fundamental groups is that every element $g\in\Aut(x)$ should be considered as an \emph{infinitesimal loop} centered at $x$, also called \emph{hidden loops} in the terminology of \cite{noohiFundamentalGroupsAlgebraic2004}. More precisely, Noohi constructs a functorial natural homomorphism
\[
    \omega_x\colon\Aut(x)\rightarrow\pi_1(\cX,x).
\]
Moreover, if $g$ lies in the connected component $\Aut(x)^{\circ}$ of $\Aut(x)$, by definition $g$ can be continuously deformed into the identity, thus the infinitesimal loop is homotopically trivial. This shows that $\omega_x$ further factors as a homomorphism
\[
\omega_x\colon\Aut(x)/\Aut(x)^{\circ}\rightarrow\pi_1(\cX,x);
\]
see~\cite[\S3]{noohiFundamentalGroupsAlgebraic2004} for the rigorous construction and \cite[Corollary 5.4]{noohiFundamentalGroupsAlgebraic2004} for the characterization of when $\omega_x$ is injective. For any other choice of geometric base point $y\in\cX$ in the same connected component as $x$, there is an isomorphism $\alpha_{x,y}:\pi_1(\cX,y)\xrightarrow{\cong}\pi_1(\cX,x)$ unique up to conjugation; in particular, we can compose $\omega_y$ with $\alpha_{x,y}$ to obtain maps $\Aut(y)/\Aut(y)^{\circ}\rightarrow\pi_1(\cX,x)$. This provides a huge source of interesting loops that would be otherwise hard to construct.

\begin{example}[{\cite[Example 4.2]{noohiFundamentalGroupsAlgebraic2004}}]\label{exm:classifying-stack}
    Let $G$ be a finite group, and let $BG=[\Spec\C/G]$ be the corresponding classifying stack, together with the natural geometric point $x\colon\Spec\C\rightarrow BG$. Then, $G\cong \Aut(x)\xrightarrow{\omega_x}\pi_1(BG,x)$ is an isomorphism, and the fundamental group is generated by infinitesimal loops only. Note that $x\colon\Spec\C\rightarrow BG$ also defines the universal cover of $BG$, as it is a $G$-torsor from a simply-connected scheme. This shows the necessity of including infinitesimal loops when working with stacks.
\end{example}

\subsubsection{Monodromy groups and Galois closures}

We are ready to introduce monodromy groups.

\begin{definition}[{\cite[p.~71]{noohiFundamentalGroupsAlgebraic2004}}]
\label{def:covering-map-of-stacks}
Let $f\colon \mathcal{Y} \to \mathcal{X}$ be a morphism of algebraic stacks. We say that $f$ is a \emph{covering map} if $f$ is representable, finite, and \'etale. It is a \emph{connected covering map} (resp. \emph{irreducible covering map}) if $\mathcal{Y}$ is further assumed to be connected (resp. irreducible).
\end{definition}

Attached to a cover of stacks, we can define a monodromy group, as in \cite[\S2]{landiStacksMonodromySymmetric2025}. Note that we do not require the covering map to be connected; this generality is very important for our scope, as the covers obtained by restricting to symmetric loci naturally split in several connected components.

\begin{definition}[{\cite[Definition 2.17]{landiStacksMonodromySymmetric2025}}]\label{def:stacky-monodromy}
Let $f\colon\cY\rightarrow\cX$ be a covering map of a connected stack $\cX$, with $\cY$ not necessarily connected. Let $\cY_1,\ldots,\cY_r$ be the connected components, with induced maps $f_i\colon\cY_i\rightarrow\cX$. Let $x$ be a geometric point of $\cX$, and for every $i$ fix a geometric point $y_i$ such that $f_i(y_i)=x_i$. We denote by $N_{f_i}$ the largest normal subgroup of $\pi_1(\cX,x)$ contained in the image of $f_{i*}\colon\pi_1(\cY_i,y_i)\rightarrow\pi_1(\cX,x)$. Let $N_f:=\cap_i N_{f_i}$. We define the \emph{stacky monodromy group} $\Mon(f)$ of $f$ to be
\[
\Mon(f)=\pi_1(\cX,x)/N_f.
\]
\end{definition}

\begin{remark}\label{rmk:map-on-pi1-injective}
    From general principles of Galois categories, for every morphism $f\colon(\cY,y)\rightarrow(\cX,x)$ pointed morphism of stacks, there is an induced pushforward $f_*\colon\pi_1(\cY,y)\rightarrow\pi_1(\cX,x)$ of fundamental groups \cite[Tag 0BN5, Tag 0BND]{Stacks}. When $f$ is a covering map, $f_*$ is injective \cite[Tag 0BN0]{Stacks}.
\end{remark}

Note that Definition~\ref{def:stacky-monodromy} is well-posed and $\Mon(f)$ is a finite group. Indeed, the image of $f_{i*}$ is of finite index in $\pi_1(\cX,x)$, hence so is $N_{f_i}$. Moreover, the normality of $N_{f_i}$ implies that of $N_f$, which is also of finite index as $\pi_1(\cX,x)$ is profinite.

As in Galois theory, it is technically easier to work with Galois covers, which we introduce now.

\begin{definition}
Let $f\colon (\mathcal{Y},y) \to (\mathcal{X},x)$ be a pointed connected covering map. We say that $f$ is \emph{Galois} if for every (equivalently, any) choice of geometric point $y\in \mathcal{Y}$, the image of the induced map $\pi_1(\mathcal{Y},y) \to \pi_1(\mathcal{X},f(y))$ is a normal subgroup of $\pi_1(\mathcal{X},f(y))$.

Let $f\colon\cY\rightarrow\cX$ be a covering map of a connected stack $\cX$, and $\cY=\cup\cY_i$ the decomposition in connected components, with induced connected covering maps $f_i\colon\cY_i\rightarrow\cX$. The \emph{Galois closure} of $f$ is the smallest Galois cover $\widetilde{f}\colon\widetilde{\cY}\rightarrow\cX$ that factors through $f_i$ for all $i$'s.
\end{definition}

\begin{remark}\label{rmk: Galois covers are torsors}
	A Galois cover $f$ with monodromy group $\Mon(f)$ is the same as a connected $\Mon(f)$-torsor, see~\cite[\href{https://stacks.math.columbia.edu/tag/03SF}{Tag 03SF}, \href{https://stacks.math.columbia.edu/tag/0BMQ}{Tag 0BMQ}]{Stacks}. In particular, the Galois closure always exists and is the cover corresponding to the open subgroup $N_f\leq\pi_1(\cX,x)$, with the notation of Definition~\ref{def:stacky-monodromy}. When $\cY$ is connected, the construction is rather explicit, see~\cite[Proposition 2.21]{landiStacksMonodromySymmetric2025}. When $\cY$ is disconnected, let $\widetilde{\cY}_i$ the Galois cover of each of the $\cY_i$; then $\widetilde{\cY}$ can be taken to be any connected component of $\widetilde{\cY}_1\times_{\cX}\ldots\times_{\cX}\widetilde{\cY}_r$. Note that all connected components are isomorphic to each other as they are permuted by the component-wise action of $\prod_i \Mon(f_i)$. Moreover, it is clear that any Galois cover of $\cX$ that factors through all $\cY_i$ factors through $\widetilde{\cY}$, and we conclude by noticing that $\widetilde{\cY}\rightarrow\cX$ is a torsor under the subgroup $H$ of $\prod_i \Mon(f_i)$ that fixes $\widetilde{\cY}$. Note that $H\cong\Mon(f)$.

    As for fundamental groups, we omit basepoints when the base $\cX$ is connected, since the monodromy group does not depend on it when viewed as an abstract finite group.
\end{remark}

\subsubsection{Basic properties of stacky monodromy}

We recall a few useful results about monodromy groups that we will use in the following. We start with a very well-known general fact of Galois theory, whose proof is therefore omitted.

\begin{lemma}\label{lem:classical-exact-sequence-Galois-covers}
    Let $f\colon\cX\rightarrow\cY$ and $h:\cY\rightarrow\cZ$ be Galois covers. Then, there exists a short exact sequence
    \[
        1\rightarrow\Mon(f)\rightarrow\Mon(h\circ f)\rightarrow\Mon(h)\rightarrow1.
    \]
\end{lemma}

The following result summarizes the main invariance properties of monodromy groups under pullback.

\begin{proposition}\label{prop:invariance-properties-monodromy}
    Let $f\colon\cY\rightarrow\cX$ be a covering map, with $\cX$ irreducible, and $g\colon\cW\rightarrow\cX$ a morphism from a connected stack. Form the pullback diagram
    \[
    \begin{tikzcd}
        g^*\cY\arrow[r]\arrow[d,"g^*f"'] & \cY\arrow[d,"f"]\\
        \cW\arrow[r,"g"] & \cX
    \end{tikzcd}
    \]
    Then,
    \begin{enumerate}
        \item\label{point:injectivity-pullback} $g^*f$ is a covering map and there is an inclusion $\Mon(g^*f)\hookrightarrow\Mon(f)$;
        \item\label{point:magic-square} if $g$ and $f$ are Galois covers, then the cokernel of the injective map $\Mon(g^*f)\hookrightarrow\Mon(f)$ is isomorphic to a common quotient of $\Mon(g)$ and $\Mon(f)$;
        \item\label{point:restriction-to-open} if $g$ is an open immersion and all connected components of $\cY$ are irreducible, then $\Mon(g^*f)=\Mon(f)$;
        \item\label{point:pullback-of-monodromy-along-flat-map} if $g$ is a faithfully flat morphism with geometrically irreducible fibers, then $\cW$ is irreducible and $\Mon(g^*f)=\Mon(f)$.
    \end{enumerate}
\end{proposition}
\begin{proof}
    To prove point~\eqref{point:injectivity-pullback}, eventually after taking a Galois closure of $\cY$, we may assume $f$ to be a Galois cover. Then, $g^*f$ is a $\Mon(f)$-torsor, in particular all connected components of $g^*\cY$ are permuted by $\Mon(f)$ and are thus isomorphic. Let $\cZ$ be any of the connected components, and $H\leq\Mon(f)$ the stabilizer of it. Then, $g^*f|_{\cZ}\colon\cZ\rightarrow\cW$ is Galois and
    \[
        \Mon(g^*f)=\Mon(g^*f|_{\cZ})=H\leq\Mon(f).
    \]
    Now, we prove point~\eqref{point:magic-square}. The Galois covers $f$ and $g$ correspond to normal subgroups $N_{f}$ and $N_{g}$ of $\pi_1(\mathcal{X})$. Set $N:=N_{f}\cap N_{g}$, and let $\psi\colon\cD\rightarrow\cX$ be the associated covering map. Note that this cover is Galois as $N$ is normal in $\pi_1(\cX)$, and $\psi$ factors through $g$ and $f$, by definition of $N$. In particular, $\psi$ factors as $\cD\xrightarrow{a}g^*\cY\xrightarrow{b}\cX$; let $\cA:=\mathrm{im}(a)\subset g^*\cY$ be the connected component that is the image of $a$. Then, $\cA\rightarrow\cX$ is a Galois cover that factors through both $\cY$ and $\cW$, hence the associated subgroup $N_{\cA}\leq\pi_1(\cX)$ is contained in $N_{f}\cap N_{g}=N$. On the other hand, the existence of $a$ implies the opposite containment, that is, $N_{\cA}=N$, so $\cD=\cA\subset g^*\cY$. We have proved that $\Mon(g\circ g^*f)=\pi_1(\cX)/N$. By Lemma~\ref{lem:classical-exact-sequence-Galois-covers}, we have then an exact sequence
    \[
        1\rightarrow\Mon(g^*f)\rightarrow\pi_1(\cX)/N\rightarrow\pi_1(\cX)/N_g\rightarrow1.
    \]
    Therefore,
    \[
        \Mon(g^*f)\cong N_{g}/N=N_g/(N_f\cap N_g)\cong N_gN_f/N_f.
    \]
    This identifies the injection $\Mon(g^*f)\hookrightarrow\Mon(f)$ with the first two terms of the exact sequence
    \begin{equation}\label{eq: exact sequence magic square}
    \begin{tikzcd}
        1\rightarrow N_g/(N_g\cap N_f)\rightarrow \pi_1(\cX)/N_f\rightarrow \pi_1(\cX)/N_gN_f \rightarrow 1.
    \end{tikzcd}
    \end{equation}
    The cokernel is then clearly a quotient of $\Mon(g)\cong\pi_1(\cX)/N_g$ and of $\Mon(f)=\pi_1(\cX)/N_f$, proving point~\eqref{point:magic-square}.
    
    Finally, points~\eqref{point:restriction-to-open} and~\eqref{point:pullback-of-monodromy-along-flat-map} are~\cite[Corollary 2.23]{landiStacksMonodromySymmetric2025} and~\cite[Corollary 2.27]{landiStacksMonodromySymmetric2025}, respectively.
\end{proof}

\begin{remark}\label{rmk: maximal intermediate cover}
    The roles of $f$ and $g$ in the setting of \Cref{prop:invariance-properties-monodromy}\eqref{point:magic-square} are completely symmetric, hence the same statement holds for $\Mon(f^*g\colon g^*\cY\rightarrow\cY)$. Moreover, the proof shows that the quotients $\Mon(f)/\Mon(g^*f)$ and $\Mon(g)/\Mon(f^*g)$ coincide and are isomorphic to $\pi_1(\cX)/N_gN_f$. This can be geometrically interpreted as follows. Form the cover $\pi:\mathcal{E}\rightarrow\mathcal{X}$ corresponding to $N_g N_f$, which is Galois. This cover is dominated by both $\mathcal{Y}$ and $\mathcal{W}$. The natural morphism $\mathcal{W}\times_{\mathcal{E}} \mathcal{Y}\to g^*\cY$ is both a closed and open immersion, as it is a morphism of finite étale covers with the target having smaller degree. Its image is the stack $\cA$ defined in the proof of Proposition~\ref{prop:invariance-properties-monodromy}\eqref{point:magic-square}. Therefore, the sequence~\eqref{eq: exact sequence magic square} can be rewritten as
    \[
    1\rightarrow\Mon(g^*f)\rightarrow\Mon(f)\rightarrow\Mon(\pi)\rightarrow 1,
    \]
    and similarly for $\Mon(f^*g)$.
\end{remark}

Proposition~\ref{prop:invariance-properties-monodromy}\eqref{point:restriction-to-open} allows us to extend the definition of monodromy groups to ramified finite covers of irreducible schemes.

\begin{definition}\label{def:monodromy-generically-covering-maps}
Let $f\colon\cY\rightarrow\cX$ be a representable, generically finite and generically étale cover between stacks, where $\cX$ and all irreducible components of $\cY$ are irreducible. Let $\cU\subset\cX$ be any non-empty open substack over which the restriction $f_{\cU}$ of $f$ is a covering map. We define the monodromy group of $f$ as $\Mon(f):=\Mon(f_{\cU})$.
\end{definition}

\begin{example}
It is extremely important that the connected components of $\cY$ are irreducible for Proposition~\ref{prop:invariance-properties-monodromy}\eqref{point:restriction-to-open} to hold, even for schemes. As an example, let $C$ be a rational curve with only one node $c\in C$. Let $\pi\colon\P^1\rightarrow C$ be the normalization map, and $p,q\in\P^1$ the (distinct) ramification points. Let $Y=\P^1\cup_{p,q}\P^1$ be the union of two copies of $\P^1$ along the corresponding pairs of points. Then, $\pi$ induces a Galois double cover $f\colon Y\rightarrow C$, whose restriction to $C\setminus\{c\}\cong\P^1\setminus\{p,q\}$ is trivial.
\end{example}

\begin{remark}\label{rmk:normality-base-implies-irreducibility}
If the base $\cX$ is normal, then $\cY$ is always automatically normal, hence all connected components are irreducible. In particular, every ramified finite cover of a connected normal stack has a well-defined monodromy group.
\end{remark}

Recall that the key computational feature of stacky monodromy is the ability to exhibit elements of monodromy via symmetries of points in the base stack. Explicitly, for a geometric point $x\in \mathcal{X}$, there is a natural group homomorphism
\[
\omega_x \colon \Aut(x) \to \pi_1(\mathcal{X},x).
\]
This need not be injective in general, but the following lemma gives a useful criterion for when this happens.

\begin{lemma}\label{lem:injectivity-hidden-loops}
Let $f\colon \mathcal{Y} \to \mathcal{X}$ be a covering map between connected locally-Noetherian algebraic stacks, and let $x\in \mathcal{X}$, $y\in\cY$ be geometric points with $f(y)=x$. If $\cY$ is an algebraic space (e.g. a scheme) at $y$, then the composite
\[
\begin{tikzcd}
    \Aut(x) \arrow[r,"\omega_x"] & \pi_1(\mathcal{X},x)\arrow[r] & \Mon(f)
\end{tikzcd}
\]
is injective.
\end{lemma}
\begin{proof}
    This is already implicit in~\cite{landiStacksMonodromySymmetric2025}, but we still provide a proof for completeness. Let $g\colon\mathcal{G}_x\rightarrow\cX$ be the residual gerbe of $\cX$ at $x$ (\cite[Tag 06MU]{Stacks}), which exists as $\cX$ is locally-Noetherian (\cite[Tag 0H22]{Stacks}), and let $g^*\cY\rightarrow\cG_x$ be the covering map obtained by pullback. Note that $y$ defines a schematic point of $g^*\cY$. By Proposition~\ref{prop:invariance-properties-monodromy}\eqref{point:injectivity-pullback}, it is enough to prove the result for $h\colon\{y\}\rightarrow B\mathcal{G}_x$, and after further base change, we may assume the existence of a trivialization $B\mathcal{G}_x\cong B\Aut(x)$ of the residual gerbe. Then, $\Mon(h)\cong\Aut(x)$ (Example~\ref{exm:classifying-stack}), and we are done.
\end{proof}

Thus when we can bound the monodromy above via geometric constraints (some of which has been well-understood since the latter half of the 19th century), we can show that stacky monodromy attains this upper bound by lifting hidden loops arising from symmetries of given objects. A key example is \cite[Propositions 2.33, 2.38]{landiStacksMonodromySymmetric2025}.

Suppose that the algebraic stack $\cX$ is a quotient stack $\cX=[U/P]$ for some scheme $U$ over a field $k$ and a $k$-algebraic group $P$. We want to compare the monodromy of a cover over $\cX$ with the monodromy of the corresponding $P$-equivariant cover over $U$.

\begin{remark}
    Note that $\pi_0(P)$ is still a $k$-algebraic group as $P^\circ$ is normal in $P$. Moreover, the construction of $P^{\circ}$ and $\pi_0(P)$ commute with extension of fields, as $P^{\circ}$ is geometrically irreducible by~\cite[Tag 0B7R]{Stacks}.
\end{remark}

\begin{proposition}\label{prop:gal-f-vs-quotient}
Let $f\colon V\rightarrow U$ be a covering map of schemes of finite type over a base field $k$, and let $P$ be a $k$-algebraic group acting on both $X$ and $Y$ so that $f$ is $P$-equivariant. Let $\cF\colon\cY=[V/P]\rightarrow[U/P]=\cX$ be the induced covering map, and assume $U$ to be connected. Then we have a short exact sequence of groups
\[
1 \to \Mon(f) \to \Mon(\cF) \to Q \to 1,
\]
where $Q$ is some quotient of $\pi_0(P)$. Moreover, if $V$ is connected, then $Q=\pi_0(P)$.
\end{proposition}
\begin{proof}
    First, suppose that $P$ is geometrically irreducible. Then, Proposition~\ref{prop:invariance-properties-monodromy}\eqref{point:pullback-of-monodromy-along-flat-map} implies that $\Mon(f)=\Mon(\cF)$, so $Q=\pi_0(P)=0$. Note that these assumptions are satisfied by $P^\circ$, hence it is enough to compare $\Mon(f)$ with $\Mon([V/P^\circ]\rightarrow[U/P^\circ])$. Note that the induced map $[U/P^\circ]\rightarrow\cX=[U/P^\circ]$ is a $\pi_0(P)$-Galois cover, hence the result follows from Proposition~\ref{prop:invariance-properties-monodromy}\eqref{point:magic-square}. When $V$ is connected, then both the proof of Proposition~\ref{prop:invariance-properties-monodromy}\eqref{point:magic-square} and Remark~\ref{rmk: maximal intermediate cover} imply that $Q=\pi_0(V)$.
\end{proof}

\section{Four flavors of monodromy}\label{sec:four-flavors}

Here we discuss the four types of monodromy which one can consider for an enumerative problem, and demonstrate that while they classically agree for the problems under consideration in this paper, they are quite different when one wishes to compute Galois groups of equivariant enumerative problems. We provide a full comparison between the four types of monodromy, and show in many cases that the stacky monodromy determines all of the others, improving the results in~\cite[\S3.4]{landiStacksMonodromySymmetric2025}.

\subsection{Marked, stacky, and geometric monodromy}

We start by introducing the symmetric loci of an algebraic stack with respect to a finite group, and define the setup we will work in.

\begin{definition}\label{def:XG-stacks}
Let $\mathcal{X}$ be a connected, separated, Deligne-Mumford stack which is smooth over an algebraically closed field of characteristic zero, and let $G$ be a finite group.
\begin{enumerate}
    \item Define $\mathcal{X}_G$ to be the stack whose objects over a scheme $S$ are pairs $(x,\rho)$ where $x\in\cX(S)$ and $\rho:G_S\hookrightarrow\underline{\Aut}_{\cX}(s)$ is a closed immersion over $S$.
    \item We let $\mathcal{X}_{=G} \subseteq \mathcal{X}_G$ be the open substack where $\rho$ is required to be an isomorphism.
\end{enumerate}
\end{definition}

In \Cref{def:XG-stacks}, the stack $\mathcal{X}_G$ is a smooth Deligne-Mumford stack with a representable, finite, unramified forgetful map
\begin{align*}
    \Phi \colon \mathcal{X}_G \to \mathcal{X}
\end{align*}
sending $(x,\rho)$ to $x$ \cite[Proposition 3.6]{landiStacksMonodromySymmetric2025}. This motivates the following definition:

\begin{definition}\label{def:X-upper-G}
Under the assumptions of Definition~\ref{def:XG-stacks}, we define $\mathcal{X}^G\subset\cX$ to be the image of $\Phi$, and let $\mathcal{X}^{=G} \subseteq \mathcal{X}^G$ the open substack where $\Aut(x) \cong G$.

We further denote by $\pi_{\cX}\colon\cX\rightarrow X$ the coarse moduli space of $\cX$, and by $X^{=G}\subset X^G\subset X$ the images of $\cX^{=G}$ and $\cX^G$ along $\pi_{\cX}$, respectively.
\end{definition}

\begin{proposition}\label{prop:X-lower-G-properties}
With the assumptions of Definition~\ref{def:XG-stacks}, we have that
\begin{enumerate}
    \item The natural map
    \begin{align*}
        \Phi_{=G} \colon \mathcal{X}_{=G} \to \mathcal{X}^{=G}
    \end{align*}
    is an $\Aut(G)$-torsor, and
    \[
        \pi_{\cX}|_{\cX^{=G}}\colon\cX^{=G}\rightarrow X^{=G}
    \]
    is a $G$-gerbe.
    \item $\mathcal{X}^{=G}$ and $X^{=G}$ are smooth.
    \item\label{lowerG-can-be-disconnected} When $\mathcal{X}^{=G}$ is connected, all connected components of $\mathcal{X}_{=G}$ are isomorphic, dominate $\mathcal{X}^{=G}$, and are permuted by the action of $\Aut(G)$.
    \item\label{point:stabilizer-X_G} The stabilizer of a point in $\cX_{=G}$ is isomorphic to the center $Z(G)$ of $G$.
\end{enumerate}
\end{proposition}
\begin{proof}
    The first point is~\cite[Proposition 3.15]{landiStacksMonodromySymmetric2025}, and it easily implies all the second and third point. Point~\eqref{point:stabilizer-X_G} is \cite[Lemma 3.12]{landiStacksMonodromySymmetric2025}.
\end{proof}

\begin{setup}\label{setup:stacky-problem} Let $\mathcal{X}$ be a connected, separated, Deligne-Mumford stack which is smooth over a field of characteristic zero. Let $\cF\colon \mathcal{Y} \to \mathcal{X}$ be a connected covering map in the sense of \Cref{def:covering-map-of-stacks}, with Galois group $\Gamma$, and denote by $F\colon Y\rightarrow X$ the induced ramified finite cover between coarse moduli spaces. Let $G$ be a finite group, and assume that $\mathcal{X}^{=G}$ is connected.
\end{setup}

We are ready to define the covers induced over the $G$-symmetric loci, and their respective monodromy groups.

\begin{definition}\label{def:XG-YG-square}
In \Cref{setup:stacky-problem}, we define $\cY^{=G}:=\cX^{=G}\times_{\cX}\cY$ and $\cY_{=G}:=\cX_{=G}\times_{\cX}\cY$, thus fitting into a cartesian diagram
\begin{equation}\label{diag:def-cYG}
\begin{tikzcd}
    \cY_{=G}\arrow[r]\arrow[d,"\cF_{G}"] & \cY^{=G}\arrow[r]\arrow[d,"\cF^{G}"] & \cY\arrow[d,"\cF"]\\
    \cX_{=G}\arrow[r,"\Phi_{=G}" below] & \cX^{=G}\arrow[r] & \cX
\end{tikzcd}
\end{equation}
Similarly, we define $Y^{=G}:=X^{=G}\times_XY$, which fits into a cartesian diagram of smooth stacks
\begin{equation}\label{diag:def-YG}
\begin{tikzcd}
    Y^{=G}\arrow[r]\arrow[d,"F^G"] & Y\arrow[d,"F"]\\
    X^{=G}\arrow[r] & X
\end{tikzcd}
\end{equation}
\end{definition}

\begin{remark}
    Note that the definitions of $\cY^{=G}$ and $\cY_{=G}$ (Definition~\ref{def:XG-YG-square}) differ from those of $\cX^{=G}$ and $\cX_{=G}$ (Definitions~\ref{def:XG-stacks}, \ref{def:X-upper-G}). We nevertheless retain this slight abuse of notation, since the distinct roles of $\cX$ and $\cY$ make the intended meaning clear, and we use this notation consistently throughout.
    Note also that the vertical morphisms in diagrams~\eqref{diag:def-cYG} and~\eqref{diag:def-YG} are $\Gamma$-torsors, and they are usually disconnected except for $\cF$ and $F$.
\end{remark}

\begin{definition}\label{def:stacky-and-marked-stacky-monodromy} Let $\cF\colon \mathcal{Y} \to \mathcal{X}$, and $G$ be as in \Cref{setup:stacky-problem}. We define:
\begin{enumerate}
    \item the \emph{stacky monodromy of $\cF$ relative to $G$} to be
    \begin{align*}
        \Monstack_G(\cF) = \Mon(\mathcal{Y}^{=G} \xto{\cF^G}\mathcal{X}^{=G})
    \end{align*}
    \item the \emph{marked stacky monodromy of $\cF$ relative to $G$} to be
    \begin{align*}
        \Monmarked_G(\cF) = \Mon(\mathcal{Y}_{=G} \xto{\cF_G} \mathcal{X}_{=G});
    \end{align*}
    If $\mathcal{X}_{=G}$ happens to be disconnected, we define marked monodromy by restriction to a connected component --- note that this is well-defined, as it is independent on the component chosen by \Cref{prop:X-lower-G-properties}(\ref{lowerG-can-be-disconnected});
    \item the \emph{geometric monodromy of $\cF$ relative to $G$} to be
    \begin{align*}
        \Mongeom_G(\cF)=\Mon(Y^{=G}\xto{F^G}X^{=G}).
    \end{align*}
\end{enumerate}
\end{definition}

\begin{remark}\label{rmk:connectedness-XG-not-restrictive}
    We show that the assumption of $\cX^{=G}$ being connected is not restrictive at all. First of all, one could of course define the monodromy and state the next results for each connected component $\cZ\subset\cX^{=G}$ separately. We avoid this not only to simplify the notation, but because we can always reduce to the case where $\cZ$ is the only connected component of $\cX^{=G}$. Indeed, if $\cX^{=G}=\bigsqcup_i\cZ_i$ is the decomposition in connected components with $\cZ_1=\cZ$, let $\cW:=\overline{\bigsqcup_{i\geq2}\cZ_i}\subset\cX$ be the closure, which does not intersect $\cZ_1$. In particular, $\cW$ is a proper closed substack of $\cX$, hence the restriction of the cover $\cF\colon\cY\rightarrow\cX$ to $\cX_1:=\cX\setminus\cW$ has the same monodromy group $\Gamma$ by Proposition~\ref{prop:invariance-properties-monodromy}\eqref{point:restriction-to-open}, and $\cZ$ is now the only connected component of $\cX_1^{=G}$.
\end{remark}

\subsubsection{Comparison: marked, stacky, and geometric monodromy}

Now, we compare the three flavors of monodromy defined so far, in Setup~\ref{setup:stacky-problem} under the following further assumption.

\begin{assumption}\label{assumption:comparison-stacky-problem}
    With the notation of Setup~\ref{setup:stacky-problem}, we suppose further that $\cY^{=G}$ is an algebraic space; equivalently, for all $x\in\cX^{=G}(k)$ the action of $\Aut(x)$ on the fiber of $\cF^G$ is faithful.
\end{assumption}

Under the above assumption, by Lemma~\ref{lem:injectivity-hidden-loops} there is an inclusion $G\hookrightarrow\Monstack_G(\cF)\leq\Gamma$, by choosing any base point $x\in\cX^{=G}$. It turns out that the inclusion is canonical, i.e. it is independent of the base point. We recall the upper bounds on $\Monstack_G(\cF)$ and $\Monmarked_G(\cF)$ which are proven in~\cite{landiStacksMonodromySymmetric2025}, and are often attained.

\begin{theorem}[{\cite[Corollary 3.27]{landiStacksMonodromySymmetric2025}}]\label{thm:stacky-basic-bound}
    With the notation of Setup~\ref{setup:stacky-problem} and Assumption~\ref{assumption:comparison-stacky-problem}, there are natural subgroup inclusions
    \begin{align}
        \label{eq:inclusion-stacky} G\leq\Monstack_G(\cF)\leq N_{\Gamma}(G),\\
        \label{eq:inclusion-marked}\Monmarked_G(\cF)\leq C_{\Gamma}(G)\cap\Monstack_G(\cF)=C_{\Monstack_G(\cF)}(G),
    \end{align}
    and an isomorphism
    \begin{equation}\label{eq:isom-geometric-monodromy}
        \Mongeom_G(\cF)\cong\Monstack_G(\cF)/G.
    \end{equation}
\end{theorem}

Trying to understand when the upper bound of~\eqref{eq:inclusion-stacky} is an equality is the core difficulty of computing symmetric monodromy groups. The following theorem relates stacky monodromy and marked stacky monodromy, and shows that~\eqref{eq:inclusion-marked} is actually an equality.

\begin{theorem}\label{thm: monodromy X_G vs X^G}
In \Cref{setup:stacky-problem} and under Assumption~\ref{assumption:comparison-stacky-problem}, let $\cZ$ be any connected (therefore irreducible, since $\mathcal{X}_{=G}$ is smooth) component of $\cX_{=G}$, and let $H\leq\Aut(G)$ be the stabilizer of $\cZ$. Then:
\begin{enumerate}
    \item\label{point 1} there exists a factorization
    \[
    \begin{tikzcd}[column sep=large]
        \cY^{=G}\arrow[r,"g"] & \cZ\arrow[r,"\Phi_{=G}|_{\cZ}"] & \cX^{=G}
    \end{tikzcd}
    \]
    of $\cF^G$ with $\Phi_{=G}|_{\cZ}$ Galois.
    \item\label{point 2} There is an isomorphism
    \[
        \Monmarked_G(\cF)\cong C_{\Gamma}(G)\cap\Monstack_G(\cF)=C_{\Monstack_G(\cF)}(G),
    \]
    and a short exact sequence
    \begin{equation}\label{eq: exact sequence M_G vs M^G}
        1\rightarrow \Monmarked_G(\cF)\rightarrow \Monstack_G(\cF)\rightarrow H\rightarrow 1.
    \end{equation}
    In particular, $H\cong\Monstack_G(\cF)/C_{\Monstack_G(\cF)}(G)$ under the natural inclusions
    \[
    \begin{tikzcd}
        \Monstack_G(\cF)/C_{\Monstack_G(\cF)}(G)\arrow[r,hookrightarrow] & N_{\Gamma}(G)/C_{\Gamma}(G) \arrow[r,hookrightarrow] & \Aut(G).
    \end{tikzcd}
    \]
\end{enumerate}
\end{theorem}
\begin{proof}
Note that $\cY^{=G}$ is an algebraic space. We start by proving~\eqref{point 1}.
For any algebraic stack $\cS$ and morphism $\alpha\colon\cS\rightarrow\cX^{=G}$, a lift of $\alpha$ to $\beta\colon\cS\rightarrow\cX_{=G}$ is equivalent to giving a trivialization $G_{\cS}\cong\cS\times_{\cX^{=G}}\cI_{\cX^{=G}}$ of the pullback of the inertia stack $\cI_{\cX^{=G}}$ of $\cX^{=G}$. Therefore, a factorization as in point~\eqref{point 1} is equivalent to a trivialization $G_{\cY^{=G}}\cong\cY^{=G}\times_{\cX^{=G}}\cI_{\cX^{=G}}$. Note that we have an isomorphism
\[
    \cY^{=G}\times_{\cX^{=G}}\cI_{\cX^{=G}}\cong\underline{\Stab}_{\cY^{=G}}
\]
over $\cY^{=G}$, where $\underline{\Stab}_{\cY^{=G}}$ is the stabilizer group scheme of the $\Gamma$-action on $\cY^{=G}$, defined as the fiber product
\[
\begin{tikzcd}
    \underline{\Stab}_{\cY^{=G}}\arrow[d]\arrow[r,"\varphi"] & \cY^{=G}\arrow[d,hookrightarrow,"\Delta"]\\
    \Gamma\times\cY^{=G}\arrow[r,"\sigma"] & \cY^{=G}\times\cY^{=G}
\end{tikzcd}
\]
where $\sigma(g,y)=(y,g\cdot y)$. We are left with showing that $\underline{\Stab}_{\cY^{=G}}\cong G\times\cY^{=G}$, which follows from Lemma~\ref{lem: trivialization stab} below. This proves point~\eqref{point 1}, as the induced map $\cY^{=G}\rightarrow\cX_{=G}$ can be chosen to factor through an irreducible component $\cZ$ by irreducibility of $\cX^{=G}$. Indeed, note that any automorphism of $G$ induces a different trivialization of $\underline{\Stab}_{\cY^{=G}}$, and this is compatible with the $\Aut(G)$ action on $\cX_{=G}$. In particular, different choices of trivializations yield different choices of the irreducible component $\cZ$, as one expects. Choosing the trivialization compatibly for all irreducible components of $\cY^{=G}$ shows that $\cY^{=G}\rightarrow\cX_{=G}$ can be chosen to factor through $\cZ$. Moreover, the morphism $g$ is finite and étale as $\cF^G$ and $\Phi_{=G}$ are, and the cover is Galois as $\cF^G$ is Galois.

Point~\eqref{point 2} follows immediately. Indeed, the map $\cY^{=G}\xrightarrow{g}\cZ\subset\cX_{=G}$ induces a section $\cY^{=G}\rightarrow\cY_{=G}$, hence an isomorphism $\cY_{=G}\cong\cY^{=G}\times\Aut(G)$. It follows that $\Mon(\cF_G)=\Mon(g)$ over $\cZ$. In particular, the exact sequence in~\eqref{eq: exact sequence M_G vs M^G} follows immediately by Lemma~\ref{lem:classical-exact-sequence-Galois-covers}. Moreover, recall that
\[
    \Mon(\cF_G)\leq C_{\Gamma}(G)\cap\Monstack_G(\cF)=C_{\Monstack_G(\cF)}(G),
\]
by Theorem~\ref{thm:stacky-basic-bound}, hence it is enough to show the opposite containment. In turn, this is equivalent to saying that $g\colon\cY^{=G}\rightarrow\cZ$ factors through the projection $[\cY^{=G}/C_{\Gamma^G}(G)]$. On the other hand, the action of $C_{\Monstack_G(\cF)}(G)$ on $\cY^{=G}\times_{\cX^{=G}}\cI_{\cX^{=G}}$ is the one induced by $\Gamma$ on $\cY^{=G}$ and trivial on $\cI_{\cX^{=G}}$ (as $\Gamma$ acts by conjugation there). It follows that the isomorphism $\cY^{=G}\times G\cong\cY^{=G}\times_{\cX^{=G}}\cI_{\cX=G}$ induces an isomorphism
\[
    [\cY^{=G}/C_{\Monstack_G(\cF)}(G)]\times G\cong[\cY^{=G}/C_{\Monstack_G(\cF)}(G)]\times_{\cX^{=G}}\cI_{\cX^{=G}}.
\]
Since this is a trivialization of the pullback of the inertia of $\cI_{\cX^{=G}}$ to $[\cY^{=G}/C_{\Monstack_G(\cF)}(G)]$, this induces the required morphism $[\cY^{=G}/C_{\Monstack_G(\cF)}(G)]\rightarrow\cZ\subset\cY_{=G}$.
\end{proof}

The following simple lemma is used in the proof of Theorem~\ref{thm: monodromy X_G vs X^G}\eqref{point 1}. Not all assumptions are necessary.

\begin{lemma}\label{lem: trivialization stab}
    Let $\Gamma$ be a finite group acting on an irreducible, separated algebraic space $Y$ that is smooth over an algebraically closed field $k$, and suppose that all $k$-valued points of $Y$ have stabilizer group of the same cardinality $c$. Then, there exists a subgroup $G\leq\Gamma$ of cardinality $c$ and an isomorphism $\underline{\Stab}_Y\cong G\times Y$ over $Y$.
\end{lemma}
\begin{proof}
    Note that $\varphi\colon\underline{\Stab}_{Y}\rightarrow Y$ is surjective and finite of constant degree $|G|$. Moreover, $\underline{\Stab}_Y$ splits as the disjoint union $\bigsqcup_{\gamma\in\Gamma}Y^{\gamma}\times\{\gamma\}$, where $Y^{\gamma}$ is the closed locus of points fixed by $\gamma$. Suppose by contradiction that for all $y\in Y(k)$ there exists some $g_y\in\Gamma$ for which $y\in Y^{g_y}\not= Y$. Then, we have a finite decomposition $Y=\bigcup Y^{g_y}$ in proper closed subschemes, contradicting the irreducibility of $Y$. It follows that there exists some $y\in Y(k)$ for which $Y^{\Stab(y)}=Y$. Then, $G:=\Stab(y)$ is contained in the stabilizer of all points of $Y$, hence equal by cardinality. The result follows.
\end{proof}

\subsection{Parameter-level and geometric monodromy}
Here we are interested in the setting where $\mathcal{X}$ and $\mathcal{Y}$ are quotient stacks arising from schemes. We are interested in the following general setting:

\begin{setup}\label{setup:all-four-monodromies}
Let $f\colon V \to U$ be generically finite dominant morphism of schemes over an algebraically closed field of characteristic zero, where $U$ is smooth. Let $\Gamma$ denote the Galois group of $f$. Let $P$ be a (geometrically) connected algebraic group acting on $V$ and $U$ properly and with finite stabilizers, and suppose that $f$ is $P$-equivariant. Assume further that for any $x\in U$ over which $f$ is unramified, the normalizer $N_P(\Stab_P(x))$ acts transitively on the set of connected components of the subvariety $U^{=\Stab_P(x)}\subseteq U$ consisting of those points whose stabilizer in $P$ is exactly equal to that of $x$. We denote by $\mathcal{Y} = [V/P]$ and $\mathcal{X}= [X/P]$, and by $Y$ and $X$ their coarse moduli spaces, and we denote by $\mathcal{F}$ and $F$ the induced maps
\begin{align*}
    \mathcal{F} &\colon \mathcal{Y} \to \mathcal{X} \\
    F &\colon Y \to X.
\end{align*}
We represent this data by the tuple $\Psi = (f\colon V \to U, P)$, and note that all the other data can be extracted from this.
\end{setup}

\begin{remark} Regarding \Cref{setup:all-four-monodromies}:
\begin{enumerate}
    \item Often in \Cref{setup:all-four-monodromies} we think of $U$ as a parameter space for objects and we think of $P$ as a group which identifies objects in $U$. In this paper we will always be working with a space of hypersurfaces in projective space, and $P$ will be the automorphism group $\PGL_n$ of that projective space (hence the notation $P$).
    \item The assumptions on \Cref{setup:all-four-monodromies} are not very restrictive, by~\Cref{rmk:connectedness-XG-not-restrictive}.
    \item For any point $x\in U$, we have that the stabilizer $G = \Stab_P(x)$ acts naturally on the fibers $f^{-1}(x)$, and this action lands inside the monodromy group $\Gamma$.
    \item The assumptions are sufficient to identify $X$ and $Y$ with the GIT quotients $U\gitquot P$ and $Y\gitquot P$. In particular $F$ is generically finite and dominant, and unramified over the points in $X$ with trivial stabilizer.
\end{enumerate}
\end{remark}

\begin{definition} For $\Psi$ as in \Cref{setup:all-four-monodromies}, and $G = \Stab_P(x)$ the stabilizer of any $x\in U$, we have that the conditions in \Cref{setup:stacky-problem} and \Cref{assumption:comparison-stacky-problem} are satisfied. We can therefore define stacky, geometric, and marked stacky monodromy for $\Psi$ as
\begin{align*}
    \Monstack_G(\Psi) &:= \Monstack_G(\mathcal{F}) \\
    \Mongeom_G(\Psi) &:= \Mongeom_G(\mathcal{F}) \\
    \Monmarked_G(\Psi) &:= \Monmarked_G(\mathcal{F}).
\end{align*}
\end{definition}

We can define a fourth monodromy group in terms of $\Psi$, which we call parameter-level monodromy:

\begin{definition}\label{def:parameter-level-monodromy}
In \Cref{setup:all-four-monodromies}, let  $G = \Stab_P(x)$ be the stabilizer of a point $x\in U$. Denote by $U^G\subseteq U$ be the subscheme of points which are $G$-fixed, and let $U^{G\circ} \subseteq U^G$ be the connected component containing $x$. We define the $G$-\emph{parameter-level monodromy} to be
\begin{align*}
    \Monpar_G(\Psi) := \Mon(V|_{U^{G\circ}} \xto{f|_{U^{G\circ}}} U^{G\circ}).
\end{align*}
\end{definition}

\begin{proposition}\label{prop:stacky-MG-normalization}
In Setup~\ref{setup:all-four-monodromies}, let $G$ be the stabilizer of a point in $U$, and suppose further that the closure of $\cX^{=G}$ in $\cX$ coincides with $\cX^G$. Then,
\begin{enumerate}
    \item the natural morphism of stacks
        \begin{align*}
            [U^G / N_{P}(G)] \to \mathcal{\cX}^G,
        \end{align*}
        is a normalization. In particular, if $\cX^G$ is normal then $\Stab_P(u)\leq N_P(G)$ for all $u\in U^G$.
    \item the natural morphism of schemes
    \begin{align*}
        U^G\gitquot N_P(G) \to X^G,
    \end{align*}
    is a normalization.
\end{enumerate}
\end{proposition}
\begin{proof}
Since we are working in characteristic zero in this paper, we have that $G$ is linearly reductive, and therefore $U^G$ is smooth. Thus the quotient stack $[U^G/N_{P}(G)]$ is also smooth and in particular normal. Moreover, $\cX^{=G}\cong[U^{=(G)}/P]\cong[U^{=G}/N_P(G)]$, and its closure in $\cX$ is $\cX^G$ by assumption. Therefore, $[U^G/N_P(G)]\rightarrow\cX^G$ is a representable birational morphism from a normal stack. Therefore, it coincides with the normalization.
Moreover, if $\cX^G$ is normal then the normalization map is an isomorphism, hence $\Stab(u)\cong\Aut(x)\leq N_P(G)$ for all $u\in U^G$ with image $x\in\cX^G$.

The second claim for coarse moduli spaces is analogous.
\end{proof}

\begin{corollary} In \Cref{setup:all-four-monodromies}, assuming that the closure of $\mathcal{X}^{=G}$ is $\mathcal{X}^G$, we have an isomorphism
\[
    \Mongeom_G(\Psi) \cong \Mon((V\gitquot P)|_{U^G \gitquot N_P(G)} \to U^{G} \gitquot N_P(G)).
\]
\end{corollary} 
\begin{proof}
    Since the map $U^G \gitquot N_P(G) \to X^G$ is a normalization, by \Cref{prop:stacky-MG-normalization} they are isomorphic on a dense open subset, and therefore the monodromy groups agree.
\end{proof}

\subsection{Comparing the four kinds of monodromy} The following observation is immediate:

\begin{proposition}\label{prop:G-is-e-all-four-agree}
In \Cref{setup:all-four-monodromies} if $G=e$ is the trivial group, and the generic stabilizer of $\mathcal{X}$ is trivial, then
\[
\Monpar_e(\Psi) \cong \Mongeom_e(\Psi) \cong \Monstack_e(\Psi) \cong \Monmarked_e(\Psi) \cong \Gamma.
\]
\end{proposition}

When $G$ is non-trivial, these can all be different from one another. Recall that we have already compared the marked, stacky, and geometric monodromy in Theorems~\ref{thm:stacky-basic-bound} and~\ref{thm: monodromy X_G vs X^G}. Therefore, we are left with studying the parameter monodromy and compare it to the other types.

In order to compare parameter-level monodromy with marked stacky monodromy, we first show that the projection morphism $U^G\rightarrow\cX^G$ factors through $\cX_G$.

\begin{lemma}\label{lem:factorization-UG-to-XG}
The natural $N_P(G)$-invariant map $U^G \to \mathcal{X}^{G}$ factors as
\begin{align*}
    U^G \xto{\varphi} \mathcal{X}_G \to \mathcal{X}^G,
\end{align*}
and $\varphi$ is $C_P(G)$-invariant. In particular, we get an induced morphism
\[
    \overline{\varphi}\colon{[U^{G}/C_P(G)]}\rightarrow \mathcal{X}_{G}.
\]
\end{lemma}
\begin{proof}
For every scheme $T$, and object $u\in U^G(T)$ mapping to $x\in\cX^G(T)$, the equalities
\[
    \underline{\Aut}(x)\cong\underline{\mathrm{Stab}}(u)=G_T\subset P_T
\]
yield a canonical closed immersion $\rho:G_T\hookrightarrow\underline{\Aut}(x)$, hence an object $(x,\rho)\in\cX_G(T)$. Moreover, $C_P(G)$ preserves the $G$-action, thus making $\varphi$ invariant.    
\end{proof}

\begin{theorem}\label{thm:comparison-UG/centralizer-MG}
    The restriction $\overline{\varphi}^{\circ}\colon[U^{=G}/C_P(G)]\rightarrow\cX_{=G}$ of $\overline{\varphi}$ to $\cX_{=G}$ sits in a commutative diagram
    \[
    \begin{tikzcd}
        {[X^{=G}/C_P(G)]}\arrow[rr,"\overline{\varphi}^{\circ}"]\arrow[rd,"\alpha"'] & & \cX_{=G}\arrow[ld,"\beta"]\\
        & \cX^{=G}
    \end{tikzcd}
    \]
    where $\alpha$ is a $N_P(G)/C_P(G)$-torsor, and $\beta$ is a $\Aut(G)$-torsor. In particular, $\overline{\varphi}^{\circ}$ is an open immersion.
\end{theorem}
\begin{proof}
    The existence of the diagram and that $\alpha$ is a $N_P(G)/C_P(G)$-torsor follow immediately from Lemma~\ref{lem:factorization-UG-to-XG} and the isomorphism $\mathcal{X}^{=G}\cong[U^{=G}/N_P(G)]$. Moreover, $\beta$ is an $\Aut(G)$-torsor by Proposition~\ref{prop:X-lower-G-properties}. Note that there is a natural inclusion $N_P(G)/C_P(G)\hookrightarrow\Aut(G)$, and $\overline{\varphi}^{\circ}$ is $N_P(G)/C_P(G)$-equivariant. By Galois theory, it follows that $\overline{\varphi}^{\circ}$ is an open immersion.
\end{proof}

This yields an immediate comparison between the monodromy of a cover over $\cX_G$ and that over $U^G$.

\begin{corollary}\label{cor:comparison-monodromy-UG-MG}
In \Cref{setup:all-four-monodromies}, the following hold.
\begin{enumerate}
    \item There is an isomorphism
    \begin{align*}
        \Mon([V^{=G}/C_P(G)] \to [U^{=G}/C_P(G)]) \cong \Monmarked_G(\Psi)
    \end{align*}
    \item\label{point:exact-seq-parameter-marked} Let $U^{=G\circ}$ be any connected component of $U^{=G}$, and let $H\leq \pi_0(C_P(G))$ be the stabilizer of $U^{=G\circ}$ in the set of connected components of $U^{=G}$. Then, there exists a quotient of $Q$ over $H$ and an exact sequence
    \[
    1\rightarrow\Monpar_G(\Psi)\rightarrow\Monmarked_G(\Psi)\rightarrow Q\rightarrow 1.
    \]
    \item\label{point:isom-parameter-marked} If $C_P(G)$ is connected, there is an isomorphism
    \begin{align*}
        \Monpar_G(\Psi) \cong \Monmarked_G(\Psi).
    \end{align*}
\end{enumerate}
\end{corollary}
\begin{proof}
    Recall that it makes sense to consider the monodromy over $\cX_{=G}$ even though $\cX_{=G}$ is not necessarily connected, as all its connected components are permuted by the action of $\Aut(G)$, hence isomorphic. A similar discussion holds for $[U^{=G}/C_P(G)]$ and $U^{=G}$.
    
    The first point follows from Theorem~\ref{thm:comparison-UG/centralizer-MG}, as it says that $[U^{=G}/C_P(G)]\rightarrow\cX_{=G}$ induces an isomorphism between a pair of connected components.

    For the second point, consider the factorization $U^{=G}\rightarrow[U^{=G}/C_P^\circ(G)]\rightarrow[U^{=G}/C_P(G)]$, and note that the first morphism induces an isomorphism between monodromy groups by Proposition~\ref{prop:invariance-properties-monodromy}\eqref{point:pullback-of-monodromy-along-flat-map}. The second morphism is a $\pi_0(C_P(G))=C_P(G)/C_P^{\circ}(G)$-torsor, and $[U^{=G\circ}/C_P^\circ(G)]$ is a connected component of $[U^{=G}/C_P^\circ(G)]$. Then, the result follows from Proposition~\ref{prop:invariance-properties-monodromy}\eqref{point:magic-square}.
    
    Point~\eqref{point:isom-parameter-marked} follows from the second point as $H=0$. Alternatively, apply directly Proposition~\ref{prop:invariance-properties-monodromy}\eqref{point:pullback-of-monodromy-along-flat-map} and the first point of this corollary.
\end{proof}

\begin{corollary}\label{cor:comparison-monodromy-UG-MG-specialcase}
In \Cref{setup:all-four-monodromies}, assume that $\mathcal{X}^{=G}$ is irreducible, that $G$ acts on the fibers of $\mathcal{Y} \to \mathcal{X}$ faithfully (equivalently $\mathcal{Y}^{=G}$ is an algebraic space). Then, there is an inclusion
\[
    \Monpar_G(\Psi) \leq C_{\Monstack_G(\Psi)}(G).
\]
Moreover, if $C_P(G)$ is (geometrically) connected, then the inclusion is an equality.

\end{corollary}
\begin{proof}
    The inclusion follows from Corollary~\ref{cor:comparison-monodromy-UG-MG}\eqref{point:exact-seq-parameter-marked} and Theorem~\ref{thm: monodromy X_G vs X^G}\eqref{point 2} (or Theorem~\ref{thm:stacky-basic-bound} \eqref{eq:inclusion-marked}). When $C_P(G)$ is connected, the result follows from Corollary~\ref{cor:comparison-monodromy-UG-MG}\eqref{point:isom-parameter-marked} and Theorem~\ref{thm: monodromy X_G vs X^G}\eqref{point 2}.
\end{proof}

\begin{remark}\label{rmk:classical-proof}
    The inclusion in Corollary~\ref{cor:comparison-monodromy-UG-MG-specialcase} can also be proved via more classical methods, as follows. As $G$ fixes every point in $U^G$, when we consider the covering space $V|_{U^{G\circ}} \to U^{G\circ}$, we can think about $G$ as acting via the deck group. It then suffices to argue the deck group action commutes with the monodromy action on the fiber. If $\gamma \in \pi_1(U^{G\circ},x)$, and $y \in V_x$, we  let $\gamma_{y}$ denote the lift of $\gamma$ from $U^{G\circ}$ to $V$ which starts at $y$. Acting via $g$, we get that $g\cdot \gamma_{y}$ is a lift of $\gamma$ which starts at $g\cdot y$. However by uniqueness of lifts, we necessarily have that $g\cdot \gamma_{y} = \gamma_{g\cdot y}$. Thus the deck group and monodromy group actions on the fiber commute.
\end{remark}

\begin{proposition}\label{prop:SES-par-stack-monodromy}
In Setup~\ref{setup:all-four-monodromies}, let $(U^{=G})^\circ$ be any connected component of $U^{=G}$, and let $H\le \pi_0 N_P(G)$ be its stabilizer in the set of connected components. Then we have a short exact sequence of groups
\[
    1 \to \Monpar_G(\Psi) \to \Monstack_G(\Psi) \to Q\to 1,
\]
and the cokernel $Q$ is a quotient of $H$.
\end{proposition}
\begin{proof} Recall that $\cX^{=G} = [U^{=G}/N_P(G)]$. We can factor the quotient map $U^{=G} \to \mathcal{X}^{=G}$ first through the $N_P^{\circ}(G)$-torsor $U^{=G}\rightarrow[U^{=G}/N_P(G)]$, and then the $\pi_0 N_P(G)$-torsor $[U^{=G}/N_P(G)]\rightarrow\cX^{=G}$. Then, we can conclude as in the proof of Corollary~\ref{cor:comparison-monodromy-UG-MG}\eqref{point:exact-seq-parameter-marked}, applying Proposition~\ref{prop:invariance-properties-monodromy}\eqref{point:pullback-of-monodromy-along-flat-map} to the first torsor, and Proposition~\ref{prop:invariance-properties-monodromy}\eqref{point:magic-square} to the second.
\end{proof}

\section{Moduli spaces of symmetric hypersurfaces}\label{sec:moduli-spaces}

Recall that the parameter space of degree $d$ hypersurfaces in $\mathbb{P}^n$ is given by $\P^{\binom{n+d}{d}-1}$. Inside this locus is a discriminant hypersurface $\Delta_{n,d} \subset \P^{\binom{n+d}{d}-1}$ parametrizing singular hypersurfaces. We make the following notational convention:

\begin{notation}\label{nota:Und-Mnd} For $n,d\ge1$ we denote by $U_{n,d} := \mathbb{P}^{\binom{n+d}{d}-1} \setminus \Delta_{n,d}$ the (affine) space of smooth degree $d$ hypersurfaces in $\mathbb{P}^n$. The convention in this paper is to let $\PGL_n$ act on $U_{n,d}$ on the right: a matrix $M$ acts on a form $f$ by $M\cdot f= f\circ M^{-1}$. Thus we denote by
\[
\mathcal{M}_{n,d} = \left[ U_{n,d} / \PGL_{n+1}\right]
\]
the associated quotient stack, and $M_{n,d}$ its coarse moduli space.
\end{notation}

In this paper we will be primarily interested in $U_{3,3}$, parametrizing smooth cubic surfaces, and $U_{2,4}$, parametrizing smooth planar quartics. Nevertheless some general facts are true of these quotient stacks, which we record here:

\begin{proposition}[{\cite{matsumuraAutomorphismsHypersurfaces1963}}]
Over a field of characteristic zero, we have that $\mathcal{M}_{n,d}$ is a smooth and irreducible algebraic stack of dimension $\binom{n+d}{d} - (n+1)^2$, and it is Deligne-Mumford for $d\ge 3$. The coarse moduli space $M_{n,d}$ is isomorphic to the GIT quotient $U_{n,d}\gitquot\PGL_{n+1}$. For $d\ge 3$ in characteristic zero, it is irreducible and quasi-projective of dimension $\binom{n+d}{d} - (n+1)^2$. It is smooth at points parametrizing hypersurfaces with trivial automorphism group.
\end{proposition}

We will be interested in studying \emph{symmetric} hypersurfaces, i.e. those that are fixed under some finite group of projective symmetries $G \le \PGL_{n+1}$. We define the \emph{isotropy group} of projective symmetries as follows:

\begin{notation}
For $X\in U_{n,d}$, we denote by
\[
\PGL_{n+1,X} = \{g\in \PGL_{n+1} \mid gX = X\}.
\]
\end{notation}

\begin{remark}
For $d\ge 3$, and for all $(n,d)$ except $(n,d) = (2,3)$ (plane cubics) and $(n,d) = (3,4)$ (quartic surfaces), we have that $\PGL_{n+1,X}$ is equal to the automorphism group of the variety $\Aut(X)$ \cite{matsumuraAutomorphismsHypersurfaces1963}.
\end{remark}

To that end, we make the following definitions:

\begin{notation} Let $G\le \PGL_{n+1}$ be a finite subgroup. Then we denote by
\begin{align*}
    U_{n,d}^G &= \{ X\in U_{n,d} \mid G \le \PGL_{n+1,X} \} \\
    U_{n,d}^{=G} &= \{ X\in U_{n,d} \mid G =\PGL_{n+1,X} \}
\end{align*}
In words, $U_{n,d}^G$ is the space of hypersurfaces fixed by $G$, and $U_{n,d}^{=G}$ is the space of hypersurfaces fixed by $G$ \emph{and nothing more}.
\end{notation}

We may also want to consider hypersurfaces which are $G$-fixed after a change of coordinates. That is, they may not be fixed literally by $G$, but by some conjugate copy of $G$ in $\PGL_{n+1}$. To that end we define the notation
\begin{align*}
    U_{n,d}^{(G)} &= \{X\in U_{n,d} \mid gGg^{-1} \leq \PGL_{n+1,X} \text{ for some } g\in \PGL_{n+1}\} \\
    U_{n,d}^{=(G)} &= \{X\in U_{n,d} \mid gGg^{-1} = \PGL_{n+1,X} \text{ for some } g\in \PGL_{n+1}\}.
\end{align*}
It is clear that
\begin{align*}
    U_{n,d}^{(G)} = \bigcup_{g\in \PGL_{n+1}}gU_{n,d}^G.
\end{align*}

Since $G$ is linearly reductive (by assumption, we are in characteristic zero), we have that $U_{n,d}^G$ is smooth, however note that it can fail to be irreducible. This is because the $G$-fixed locus may not be a single projective space, but a union of projective spaces coming from different 1-dimensional characters of the group.

\begin{example} Consider the space of cubic surfaces which are fixed by the cyclic group of order two acting via $[x:y:z:w] \mapsto [x:y:-z:-w]$. We have two character eigenspaces, both containing smooth cubic surfaces. Thus for \emph{this cyclic group} of order two, we have that $U_{3,3}^{C_2}$ is not irreducible.\footnote{This is not an issue for us because this $C_2$ is not itself the automorphism group of a cubic surface. When a cubic surface admits an involution of this form, it is necessarily of Type X, and is therefore fixed under the involution $w\mapsto -w$ as well. Under this $C_2\times C_2$ action, we only see one character eigenspace with smooth cubic forms in it, as we compute in \Cref{prop:Utype10}.}
\end{example}

\begin{remark} While the specific $U_{n,d}^G$ under consideration in this paper are in fact irreducible (this requires proof), when we start to take orbits of them under $\PGL_{n+1}$, we will see overlapping irreducible components. In particular this tells us that the $U_{n,d}^{(G)}$ varieties are not, in general, normal.
\end{remark}

\subsection{The moduli spaces of symmetric cubic surfaces}

We now explore the spaces $U_{3,3}^G$ as $G$ ranges over the possible automorphism groups of smooth cubic surfaces over the complex numbers. Recall the following classification, which can be found for instance in \cite{Dolgachev}\footnote{There are errata in the table there that have been corrected in the forthcoming second edition of Dolgachev's \emph{Classical Algebraic Geometry}.} (see also \cite{brundu2023construction}).
\begin{theorem}\label{thm:possible-aut-gps}
For $X \in U_{3,3}$, there are exactly 11 possible nontrivial values for $\Aut(X) = \PGL_4(\C)_X$, up to conjugacy in $\PGL_4(\C)$. These groups, laid out in a poset corresponding to subconjugacy in $W(E_6)$ or $\PGL_4$\footnote{These two notions can be verified to agree since we are working over a closed field of characteristic zero.}, and decorated with their ``Type,'' is as follows:
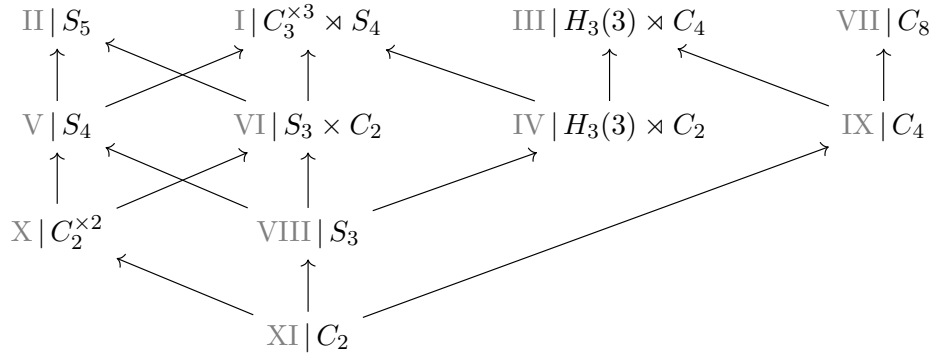
\[
\begin{tikzcd}[column sep=large]
     \gpdiagram{II}{S_5} & \gpdiagram{I}{C_3^{\times 3}\rtimes S_4} & \gpdiagram{III}{H_3(3)\rtimes C_4} & \gpdiagram{VII}{C_8}\\
    \gpdiagram{V}{S_4}\ar[u]\ar[ur] & \gpdiagram{VI}{S_3\times C_2}\ar[ul]\ar[u]   & \gpdiagram{IV}{H_3(3)\rtimes C_2}\uar\ar[ul] & \gpdiagram{IX}{C_4}\uar\ar[ul]\\
    \gpdiagram{X}{C_2^{\times 2}}\ar[u]\ar[ur]  & \gpdiagram{VIII}{S_3}\ar[ul]\ar[u]\ar[ur] &  & \\
      & \gpdiagram{XI}{C_2}\uar\ar[ul]\ar[uurr] &  & 
\end{tikzcd}
\]
\end{theorem}

\begin{theorem}\label{thm:irred-UG}
For each group $G$ which appears in \Cref{thm:possible-aut-gps}, we have that $U_{3,3}^G$ and $U_{3,3}^{(G)}$ are closed irreducible subvarieties of $U_{3,3}$. Their dimensions, together with the normalizer of $G$ in $\PGL_4$ is listed in the following table:
\begin{center}
    \footnotesize
    \begin{tabular}{r c | c c | c c | c }%
    type & $G$ & 
    $\dim U_{3,3}^{G}$ & $\dim U_{3,3}^{(G)}$ & 
    $N_{\PGL_4}(G)$ & $C_{\PGL_4}(G)$ & 
    ref \\
    \hline
    \begin{filecontents*}{symmcubics.csv}
grptype,grpname,Atlas,Carter,Eckardt,dimUG,dimUGG,Tu,NPGL,pizeroNPGL,centralizer,propref
I,$C_3^{\times 3} \rtimes S_4$,3C,$A_2$,18,0,15,$H_{10}^{(18)}$,$C_3^{\times 3}\rtimes S_4$,$C_3^{\times 3}\rtimes S_4$,0,\ref{prop:Utype1}
II,$S_5$,5A,$A_4$,10,0,15,$H_{10}^{(10)}$,$S_5$,$S_5$,0,\ref{prop:Utype2}
III,$H_3(3)\rtimes C_4$,12A,$E_6$,9,1,15,?,\ref{eqn:normalizer-type-3-cubics},$C_3^{\times 2}\rtimes Q_8$,$\mathbb{C}^\times$,\ref{prop:Utype3}
IV,$H_3(3)\rtimes C_2$,3A,$3A_2$,9,2,16,$H_4^{(9)}$,\ref{eqn:normalizer-type-4-cubics},$C_3^{\times 2}\rtimes \SL_2(\mathbb{F}_3)$,$\mathbb{C}^\times$,\ref{prop:Utype4}
V,$S_4$,4B,$A_1+A_3$,6,2,16,$H_4^{(6)}$,$S_4\times \C^\times$,$S_4$,$\mathbb{C}^\times$,\ref{prop:Utype5}
VI,$S_3\times C_2$,6E,$A_1A_5$,4,3,16,$H_4^{(4)}$,$S_3\times (\C^\times)^2$,$S_3$,$(\mathbb{C}^\times)^{2}$,\ref{prop:Utype6}
VII,$C_8$,8A,$D_5$,1,3,15,tuhere,$(\C^\times)^3$,0,$(\C^\times)^3$,\ref{prop:Utype7}
VIII,$S_3$,3D,$2A_2$,3,6,17,$H_3$,$S_3\times \GL_2(\C)$,$S_3$,$\GL_2(\C)$,\ref{prop:Utype8}
IX,$C_4$,4A,$D_4(a_1)$,1,6,16,,$\GL_2(\C)\times \C^\times$,0,$\GL_2(\C)\times \mathbb{C}^\times$,\ref{prop:Utype9}
X,$C_2^{\times 2}$,2B,$2A_1$,2,7,17,$H_2^{(2)}$,$(\GL_2(\C) \times \C^\times) \rtimes C_2$,$C_2$,$\GL_2(\C)\times \C^\times$,\ref{prop:Utype10}
XI,$C_2$,2A,$4A_1$,1,12,18,$H_1$,$\GL_3(\C)$,0,$\GL_3(\C)$,\ref{prop:Utype11}
\end{filecontents*}
\csvreader[head to column names]{symmcubics.csv}{}% 
    {\grptype & \grpname & \dimUG & \dimUGG & \NPGL & \centralizer & \propref \\}% 
    \end{tabular}
\end{center}
\end{theorem}
\begin{proof} Since cubic surfaces are anticanonically embedded, the action of $G = \Aut(X)$ on $\mathbb{P}^3$ is projectivized from its action on $H^0(X,-K_X)$, hence each $G$ lifts to an isomorphic copy of itself in $\GL_4$. We can use this to explicitly compute each of the varieties $U_{3,3}^G$, which we include in \Cref{sec:U33}. Since each $G$ is linearly reductive and $U_{3,3}$ is smooth, we have that $U_{3,3}^G$ is smooth as well.
\end{proof}

\subsection{The moduli spaces of symmetric planar quartics}

Recall that a planar quartic curve is cut out by a homogeneous degree four form in three variables -- the moduli space of such forms is $\mathbb{P}H^0(\mathbb{P}^2, \mathcal{O}(4)) \cong \mathbb{P}^{14}$. We again have a classification for the kinds of symmetries which can appear.

\begin{theorem} For $X\in U_{2,4}$, there are exactly 12 possible nontrivial values for $\Aut(X) = \PGL_3(\C)_X$, up to conjugacy in $\PGL_3(\C)$ (or $W^+(E_7)$). These groups laid out according to subconjugacy are:
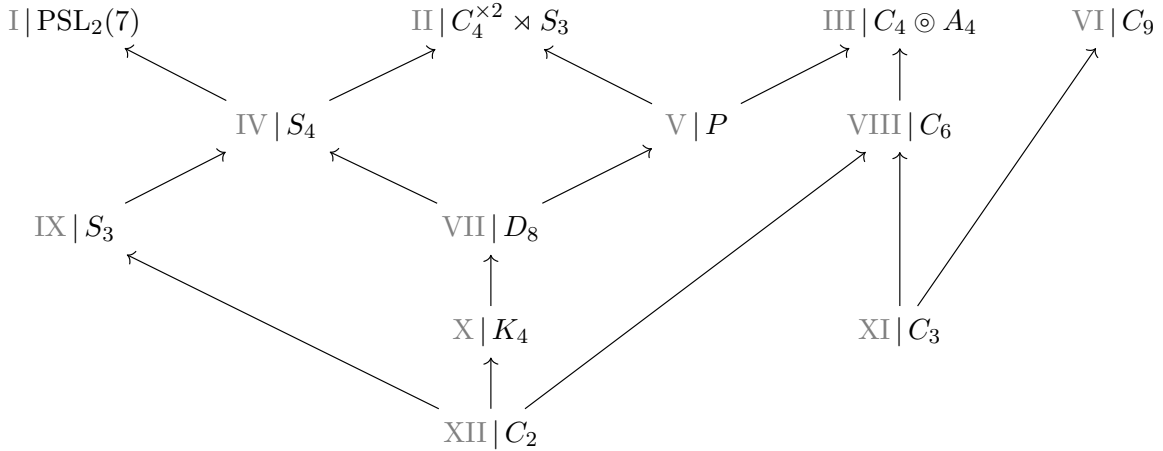
\[ \begin{tikzcd}
    \gpdiagram{I}{\PSL_2(7)} &  & \gpdiagram{II}{C_4^{\times2}\rtimes S_3} &  & \gpdiagram{III}{C_4\circledcirc A_4} & \gpdiagram{VI}{C_9}\\
     & \gpdiagram{IV}{S_4}\ar[ul]\ar[ur] &  &\gpdiagram{V}{P}\ar[ul]\ar[ur]  & \gpdiagram{VIII}{C_6}\uar & \\
    \gpdiagram{IX}{S_3}\ar[ur] &  & \gpdiagram{VII}{D_8}\ar[ul]\ar[ur] &  &  & \\
     &  &  \gpdiagram{X}{K_4}\uar &  & \gpdiagram{XI}{C_3}\ar[uu]\ar[uuur]  & \\
     &  &  \gpdiagram{XII}{C_2}\uar\ar[uull]\ar[uuurr] &  & & 
\end{tikzcd} \]
\end{theorem}

For each of the 12 possible automorphism groups of smooth planar quartics, we compute the dimensions of their moduli, and we compute their normalizers in $\PGL_3$.

\begin{theorem}\label{thm:irred-VG}
Let $G$ be one of the 12 possible automorphism groups for smooth quartic curves over $\mathbb{C}$ (see \cite[6.5.2]{Dolgachev} or \cite[2.4]{betheabrazelton_bitangentssymmetricquartics}). Then $U_{2,4}^G$ and $U_{2,4}^{(G)}$ are smooth closed irreducible subvarieties of $V$, and their dimensions and normalizers are given as follows:
\begin{center}
    \footnotesize
    \begin{tabular}{r c | c c | c c | c}%
    type & $G$ & $\dim U_{2,4}^G$ & $\dim U_{2,4}^{(G)}$ & $N_{\PGL_3}(G)$ & $C_{\PGL_3}(G)$ & ref \\
    \hline
    \begin{filecontents*}{symmquartics.csv}
grptype,grpname,dimVG,dimVGG,NPGL,pizeroNPGL,centralizer,propref
I,$\PSL_2(7)$,0,8,$\PSL_2(7)$,$\PSL_2(7)$,0,\ref{prop:Vtype1}
II,$C_4^{\times 2}\rtimes S_3$,0,8,$C_4^{\times 2} \rtimes S_3$,$C_4^{\times 2} \rtimes S_3$,0,\ref{prop:Vtype2}
III,$C_4\circledcirc A_4$,1,8,Eqn. (\ref{eqn:normalizer-P}),$S_4$,$\mathbb{C}^\times$,\ref{prop:Vtype3}
IV,$S_4$,1,9,$S_4$,$S_4$,0,\ref{prop:Vtype4}
V,$P$,2,9,Eqn. (\ref{eqn:normalizer-P}),$S_4$,$\mathbb{C}^\times$,\ref{prop:Vtype5}
VI,$C_9$,2,8,$(\mathbb{C}^\times)^2$,0,$(\mathbb{C}^\times)^2$,\ref{prop:Vtype6}
VII,$D_8$,3,10,Eqn. (\ref{eqn:normalizer-quartictype7}),$D_8$,$\mathbb{C}^\times$,\ref{prop:Vtype7}
VIII,$C_6$,3,9,$(\mathbb{C}^\times)^2$,0,$(\mathbb{C}^\times)^2$,\ref{prop:Vtype8}
IX,$S_3$,3,10,$S_3\times \mathbb{C}^\times$,$S_3$,$\mathbb{C}^\times$,\ref{prop:Vtype9}
X,$K_4$,5,11,$(\mathbb{C}^{\times})^2\rtimes S_3$,$S_3$,$(\mathbb{C}^\times)^2$,\ref{prop:Vtype10}
XI,$C_3$,6,10,$\GL_2$,0,$\GL_2$,\ref{prop:Vtype11}
XII,$C_2$,8,12,$\GL_2$,0,$\GL_2$,\ref{prop:Vtype12}
\end{filecontents*}
\csvreader[head to column names]{symmquartics.csv}{}% 
    {\grptype & \grpname & \dimVG & \dimVGG & \NPGL & \centralizer & \propref \\}% 
    \end{tabular}
\end{center}
\end{theorem}
\begin{proof} This is directly analogous to \Cref{thm:irred-UG}. The explicit computations of the $U_{2,4}^G$ varieties can be found in \Cref{sec:U24}.
\end{proof}

\section{Galois groups of lines on symmetric cubic surfaces}\label{sec:galois-groups-of-lines}

Given $G$ some automorphism group of smooth cubic surfaces, we can ask what the parameter-level, stacky, and geometric monodromy groups are. For Type V, the geometric and parameter-level mondoromy groups were computed in \cite{brazeltonMonodromySpaceSymmetric2025}. The geometric and stacky monodromy groups for Types V, VI, VIII, and XI were computed in \cite{landiStacksMonodromySymmetric2025}. The parameter-level monodromy for Types II, V, VI, VIII, X, XI were computed in \cite{pichon-pharabodGaloisGroupsSymmetric2025}. Here we complete all the remaining cases.

\begin{notation} Consider the incidence variety of lines on cubic surfaces
\[
\Lines = \{(X,\ell) \in U_{3,3} \times \mathbb{G}(1,3) \mid \ell \subseteq X\}.
\]
This is finite and \'etale over $U_{3,3}$, and is $\PGL_4$-equivariant. We denote by 
$$
\LinesProblem = (\Lines \to U_{3,3}, \PGL_4)
$$
the problem of lines on cubic surfaces, and we remark that it satisfies all the conditions in \Cref{setup:all-four-monodromies}.
\end{notation}

\begin{theorem}\label{thm:main}
Let $G$ be an automorphism group of a smooth complex cubic surface. Then the parameter-level, marked, stacky, and geometric monodromy groups of solving for lines on $G$-symmetric cubic surfaces are as follows:
\input{TABLE-cubics}
\end{theorem}

We compute these groups row-by-row. The first few are rather boring, but we will see that the computations get more interesting and much more geometrically rich as we move down the subgroup poset.

\subsection{Type I} The Type I cubic surface is the Fermat, which is unique (up to projective change of coordinates). Hence there is no parameter-level monodromy to investigate. Similarly, the geometric monodromy is trivial, since the Fermat yields just a single point in the coarse moduli space. By Theorem~\ref{thm:stacky-basic-bound}\eqref{eq:inclusion-stacky}, the stacky monodromy is just equal to the group itself. The marked monodromy group is also trivial, for instance because the center of the automorphism group $C_3^{\times3}\rtimes S_4$ is trivial.

\subsection{Type II} The Type II cubic surface is also unique, given by the Clebsch. Thus the marked, parameter and geometric monodromy vanish as in Type I, and the stacky monodromy is the group itself.

\subsection{Type III}\label{subsec:type3-cubic-monodromy}
Unlike the Type I and II cubic surfaces, once we fix a projective representation for the Type III automorphism group, we do \emph{not} obtain a uniquely determined cubic surface, instead we get a 1-dimensional family. These are all projectively equivalent, hence in the coarse moduli space there is a unique point corresponding to the Type III cubic surface, and we see that the geometric monodromy vanishes and the stacky monodromy is equal to the Type III group, as in the previous two cases. By \Cref{cor:comparison-monodromy-UG-MG-specialcase}, the parameter-level monodromy is the center $Z(H_3(3)\rtimes C_4) \cong C_3$. This can be witnessed geometrically by a small loop around a singular element in the 1-dimensional family $U_{2,4}^{H_3(3)\rtimes C_4}$ and can be verified computationally if desired.

\subsection{Type IV}\label{subsec: Type IV cubic surfaces} Recall that Type IV cubic surfaces have automorphism group $H_3(3)\rtimes C_2$.

\begin{theorem}
\label{thm:Type-IV-monodromy}
Letting $W(L_3)$ denote the Shephard-Todd group of the Dynkin diagram of type $L_3$, we have
\begin{align*}
\Monstack_{H_3(3)\rtimes C_2}(\LinesProblem) &\cong W(L_3) \\
    \Mongeom_{H_3(3)\rtimes C_2}(\LinesProblem) &\cong A_4 \\
    \Monmarked_{H_3(3)\rtimes C_2}(\LinesProblem)\cong\Monpar_{H_3(3)\rtimes C_2}(\LinesProblem) &\cong C_{W(E_6)}(H_3(3)\rtimes C_2) \cong C_3.
\end{align*}
In particular the stacky and geometric monodromy are the unique index two subgroups of the expected monodromy.
\end{theorem}

We prove this by exploring the geometry of Type IV cubic surfaces -- they are intimately related to the geometry of elliptic curves. Recall by~\cite[Table 9.6]{Dolgachev}, Type IV complex cubic surfaces are all projectively isomorphic to one with equation
\begin{equation}\label{eq:TypeIVprojectiveequation}
\begin{aligned}
    S(a) = V(x^3+y^3+z^3+w^3-3a xyz)
\end{aligned}
\end{equation}
for some parameter $a\in\mathbb{C}$. In particular, they form a 1-dimensional family. The singular members correspond to a parameter $a$ such that $a^3=1$ or at infinity.

Every Type IV cubic surface has exactly 9 Eckardt points, which are all coplanar. In the equation \eqref{eq:TypeIVprojectiveequation}, the Eckardt points lie on the plane $w=0$, and are of the form $[1:0:-\zeta_3^j:0]$, $[1:-\zeta_3^j:0:0]$, and $[0:1:-\zeta_3^j:0]$, for $j\in\{0,1,2\}$.

Restricting our attention to the plane $w=0$, the cubic surface intersects at a cubic curve, which is immediately seen to be in Hesse form:
\begin{align*}
    E(a) = V(x^3+y^3+z^3 - 3axyz).
\end{align*}
The 9 Eckardt points become the 9 flexes of the cubic. We will see this defines a morphism of stacks.

\begin{definition}\label{def:stack-of-elliptic-curves} Let $\mathcal{E} = [(\P^9 \setminus \Delta_{\text{sing}}) / \PGL_3]$ be the moduli stack of embedded smooth plane cubic curves.
\end{definition}

\begin{remark}\label{rmk: facts on plane cubics}
Recall that the moduli space of $\cE$ is $\A^1$, via the $j$-invariant. Explicitly, every plane cubic can be put in Hesse form
\begin{equation}\label{eq: Hesse pencil}
    E(a):=\mathrm{V}(x^3+y^3+z^3-3axyz),
\end{equation}
which is smooth away from $a^3=1$, and has $j$-invariant
of the form
\begin{align*}
    j(E(a)) = \frac{27a^3(a^3+8)^3}{(a^3-1)^3}.
\end{align*}
The pencil parametrizes elliptic curves with full level-3 structure, that is $Y(3)=\A^1\setminus\mathrm{V}(a^3-1)$. In particular, $\cE\cong[(\A^1\setminus\mathrm{V}(a^3-1))/\mathrm{ASL}_2(\mathbb{F}_3)]$, see for instance~\cite[Proposition, page 694]{Harris-Galois}. Recall that the Hesse group $\mathrm{ASL}_2(\mathbb{F}_3)$ is of order $216=18\cdot12$, and it is a non-split extension
\[
\begin{tikzcd}
    1\arrow[r] & ((C_3\times C_3)\rtimes C_2)\arrow[r] & \mathrm{ASL}_2(\mathbb{F}_3)\arrow[r] & A_4\arrow[r] & 1
\end{tikzcd}
\]
It is non-split because $\ASL_2(\mathbb{F}_3)$ does not contain any subgroup isomorphic to $A_4$. The normal subgroup $(C_3\times C_3)\rtimes C_2$ is the automorphism group of a generic plane cubic, hence it acts trivially on $\A^1$, while the action of $A_4$ is generated by the two automorphisms
\[
    a\mapsto \zeta_3 a,\qquad a\mapsto \frac{a+2}{a-1}.
\]
See for instance~\cite[\S2]{planecubicsAM17} or other standard texts on elliptic curves.
\end{remark}

Intersecting with the hyperplane containing the 9 Eckardt points then gives a well-defined morphism of stacks
\begin{align*}
    \mathcal{M}_{3,3}^{H_3(3)\rtimes C_2} \to \mathcal{E}.
\end{align*}

We now want to compute the automorphism group of the cubic surfaces in this moduli space with the automorphism groups of the resulting cubic curves. Generically, the automorphism group $(C_3\times C_3)\rtimes C_2$ of a smooth planar cubic is the quotient of $H_3(3)\rtimes C_2$ by $C_3$, where $C_3$ acts via $w\mapsto \zeta_3 w$. Non-generically, the automorphism group of the planar curve grows if and only if $j=0$ or $j=1728$.
\begin{proposition} Let $S(a)$ be as above. Then
\begin{enumerate}
    \item $j(E_S(a)) = 0$ if and only if $S(a)$ is projectively equivalent to the Fermat cubic surface
    \item $j(E_S(a)) = 1728$ if and only if $S(a)$ is projectively equivalent to a Type III cubic surface.
\end{enumerate}
\end{proposition}
\begin{proof}
If $j(E(a)) = 0$, then either $a=0$ in which case we obtain the Fermat cubic surface, or $a^3 + 8=0$. In this latter case, we obtain a cubic surface defined by the equation
\[
x^3 + y^3 + z^3 + w^3 + 6\zeta_3^\ell xyz
\]
for $\ell=0,1,2$. Such a cubic surface is projectively equivalent to the Fermat as well,\footnote{There is a typo in \cite{Dolgachev} which says such surfaces are of Type II instead of Type I.} via the change of coordinates
\begin{align*}
    x &= X + Y + Z \\
    y &= X + \zeta_3 Y + \zeta_3^2 Z \\
    z &= X + \zeta_3^2 Y + \zeta_3 Z.
\end{align*}
If $j(E(a)) = 1728$, then $a$ is a solution to the equation
\begin{align*}
    27a^3(a^3 + 8)^3 - 1728(a^3-1)^3 = 27(a^6 - 20a^3 - 8)^2.
\end{align*}
Thus we obtain a degree six equation whose solutions yield Type III cubic surfaces, as remarked in \cite[9.5.8]{Dolgachev}.
\end{proof}
Let $\mathcal{E}^\circ = j^{-1}(\A^1\setminus\{0,1728\})$. Then we obtain the following result; we refer to~\cite[\S3]{EHKV01} and~\cite{Gir65} for the definition of properties of (banded) gerbes.
\begin{proposition}\label{prop:Phi-banded-C3-gerbe}
The natural map
\begin{align*}
    \Phi\colon\mathcal{M}_{3,3}^{=H_3(3)\rtimes C_2} \to \mathcal{E}^\circ
\end{align*}
is a banded $C_3$-gerbe.
\end{proposition}

We are now in a position to complete the computation of $\Monstack_{H_3(3)\rtimes C_2}(\LinesProblem)$ and its non-stacky companion.

\begin{proof}[Proof of Theorem~{\ref{thm:Type-IV-monodromy}}]
Letting $\A^1$ be the moduli space of $\mathcal{E}$, we can take $W\subset\A^1$ to be the complement of $j^{-1}(\{0,1728\})$ and $\mathrm{V}(a^3-1)$. Recall that this is the space that gives Dolgachev's pencil restricted to Type IV cubic surface, which can also be thought of as $Y(3)$. The above discussion yields a commutative diagram
\begin{equation}
\begin{tikzcd}[column sep=large]
    \widetilde{\cL}^{=H_3(3)\rtimes C_2}\arrow[rd,"\psi"] & W\arrow[d]\arrow[rd,"="]\\
    & \cM_{3,3}^{=H_3(3)\rtimes C_2}\times_{\cE^{\circ}}W\arrow[r]\arrow[d,"\varphi'"] & W\arrow[d,"\varphi"]\\
    & \cM_{3,3}^{=H_3(3)\rtimes C_2}\arrow[r,"\Phi"] & \cE^{\circ}
\end{tikzcd}
\end{equation}
where the square is cartesian by definition.
The morphism $\psi\colon\widetilde{\cL}^{=H_3(3)\rtimes C_2}\rightarrow \cM^{=H_3(3)\rtimes C_2}\times_{\cE^{\circ}}W$ is obtained by the fact that by having all lines ordered we can also determine an ordered tuple of Eckardt points.
By~\cite[Proposition, page 694]{Harris-Galois} or~\cite[Proposition 2.35, Remark 2.36]{landiStacksMonodromySymmetric2025}, the morphism $\varphi$ is a connected $\mathrm{ASL}_2(\mathbb{F}_3)$-Galois cover over $\mathbb{C}$, and a connected $\mathrm{AGL}_2(\mathbb{F}_3)$-cover over $\mathbb{R}$. Since $\Phi$ is a gerbe by~\Cref{prop:Phi-banded-C3-gerbe}, also $\varphi'$ is a connected Galois cover under the same group. Finally, the fact that over $W$ all lines are defined over $\mathbb{C}(\sqrt[3]{a^3+1})$ shows that $\psi$ is a finite étale morphisms of degree divisible by 3. Thus the stacky monodromy group of $\widetilde{\mathcal{L}}^{=H_3(3)\rtimes C_2} \to \cM_{3,3}^{=H_3(3)\rtimes C_2}$ has order at least $648$ over $\mathbb{C}$, and order $1296$ over $\mathbb{R}$. As we already know the bound $\Mon(\cL^{=H_3(3)\rtimes C_2}/\cM_{3,3}^{=H_3(3)\rtimes C_2})\leq N_{W(E_6)}(H_3(3)\rtimes C_2)=W(L_3)\rtimes 2$, which has order $1296$, this concludes. Geometric monodromy follows immediately, and also depends upon the field over which we are working.

For the marked and parameter-level monodromy computation, we note that $C_{\PGL_4}(H_3(3)\rtimes C_2)$ is connected, hence we conclude by Theorem~\ref{thm: monodromy X_G vs X^G}\eqref{point 2} and Corollary~\ref{cor:comparison-monodromy-UG-MG}\eqref{point:isom-parameter-marked}.
\end{proof}

\begin{remark} Over the real numbers, the proof shows that the stacky and geometric monodromy groups are the expected ones (the normalizer and the normalizer modulo the group, respectively). This is really the same phenomenon appearing in~\cite[\S II.2]{Harris-Galois}.
\end{remark}

\begin{remark}\label{rmk: brauer group}
    The $C_3$-banded gerbe $\Phi$ studied in Proposition~\ref{prop:Phi-banded-C3-gerbe} is non-trivial, and this holds true even after passing to Zariski-open subsets of $\cE^{\circ}$. Indeed, the Zariski-local triviality of the $C_3$-gerbe would imply that the exact sequence (for general parameter $a$)
    \[
        1\rightarrow C_3\rightarrow \Aut(S(a))\cong H_3(3)\rtimes C_2\rightarrow\Aut(E(a))\cong(C_3\times C_3)\rtimes C_2\rightarrow 1
    \]
    splits, which it does not. Since the $C_3$-gerbe is banded (or by direct inspection) the extension above is central, hence a splitting would induce an isomorphism $H_3(3)\rtimes C_2\cong C_3\times(C_3\times C_3)\rtimes C_2$, which is false. In particular, the class in the Brauer group of $\cE^{\circ}$ that is associated to the $C_3$-gerbe $\Phi$ is non-trivial (and of degree 3); see also~\cite[Remark 4.25]{landiStacksMonodromySymmetric2025}.
\end{remark}

\subsection{Type V} These computations are in \cite{brazeltonMonodromySpaceSymmetric2025,landiStacksMonodromySymmetric2025}. The parameter-level monodromy is in \cite{brazeltonMonodromySpaceSymmetric2025}, and has been certified numerically in \cite{duff2026certifyinggaloismonodromyactionshomotopy}. The marked monodromy follows as usual.

\subsection{Type VI} The automorphism group of Type VI surfaces is $S_3\times C_2$.

\begin{proposition}\label{prop:type6-stack-monodromy}
The stacky and geometric monodromy groups for lines on $S_3\times C_2$-equivariant cubic surfaces are
\begin{align*}
    \Monpar_{S_3\times C_2}(\LinesProblem) &\cong S_3\times C_2\\
    \Monstack_{S_3\times C_2}(\LinesProblem) &\cong S_3^2 \times C_2 \\
    \Mongeom_{S_3\times C_2}(\LinesProblem) &\cong S_3.
\end{align*}
\end{proposition}
\begin{proof} The stacky and geometric computations are \cite[Theorem 4.19]{landiStacksMonodromySymmetric2025}. Since $C_{\PGL_4}(S_3)$ is connected, the marked and parameter-level monodromy agree, and are equal to the centralizer of $S_3$ in $S_3^{2}\times C_2$, which is $S_3\times C_2$.
\end{proof}

\subsection{Type VII} In this setting, the automorphism group is $C_8$. There is a unique $C_8$-equivariant cubic surface in the geometric moduli space, so the geometric monodromy is trivial. This tells us that the stacky monodromy group is the group itself, $C_8$. Finally, since $C_{\PGL_4}(C_8)$ is connected by \Cref{prop:Utype7} and $C_8$ is abelian, we have that the stacky and parameter-level monodromy coincide.

\subsection{Type VIII} In this setting, the automorphism group is $S_3$. We have the following:

\begin{proposition}
The stacky and geometric monodromy for lines on $S_3$-cubic surfaces are
\begin{align*}
    \Monstack_{S_3}(\LinesProblem) &\cong S_3^3\\
    \Mongeom_{S_3}(\LinesProblem) &\cong S_3^2 \\
    \Monmarked_{S_3}(\LinesProblem)&\cong\Monpar_{S_3}(\mathcal{F}) \cong S_3^2.
\end{align*}
\end{proposition}
\begin{proof}
The stacky and geometric monodromy are computed in \cite[Theorem 4.13]{landiStacksMonodromySymmetric2025}. The marked monodromy equals the parameter-level monodromy, which can be computed as the centralizer of $S_3$ in the stacky monodromy since $C_{\PGL_4}(S_3)$ is connected.
\end{proof}

\subsection{Type IX}

In this case the group is $C_4$, and every smooth cubic surface $S$ with $\Aut(S)=C_4$ is isomorphic to one in the 1-dimensional family of cubic surfaces given by
\begin{equation}\label{eqn:dolg-family-c4}
\begin{aligned}
    x^3 + xy^2 + ay^3 + yz^2 + zw^2,
\end{aligned}
\end{equation}

by~\cite[Table 9.6]{Dolgachev}. We denote by $S_a$ the corresponding cubic surface. There are two singularities in this pencil, when $a = \pm \frac{2i}{3\sqrt{3}}$. The resulting singular cubic surface has exactly one singular point
\[
    \left[ \mp \frac{i}{\sqrt{3}} : 1 : 0 : 0\right],
\]
where $\mp$ is respective to the sign on $a$. For $a=0$, the corresponding surface $S_0$ is smooth, but the automorphism group grows to $C_8$. Let
\[
W = \A^1\setminus \left\{ \pm \frac{2i}{3\sqrt{3}}\right\}
\]
be the variety parametrizing the family $S_a$. Observe that, while each point is fixed by the action of $C_4$, there is an action of the Type VII group $C_8$ via the diagonal matrix with entries $[1,-1,\zeta_8^2,\zeta_8^3]$ which sends $S_a$ to $S_{-a}$.

\begin{notation}[Local notation] We will denote by $\mathcal{X} = [W/C_8]$ the quotient stack of the family $W$ by the action of $C_8$.
\end{notation}

\begin{proposition} \,
\begin{enumerate}
    \item There is no special value of $a\in \A^1$ for which the cubic surface in \Cref{eqn:dolg-family-c4} is a Type III cubic surface.
    \item There is a natural open immersion
    \[
    \mathcal{X} \to \mathcal{M}_{3,3}^{C_4}
    \]
    exhibiting $\mathcal{X}$ as the open substack of cubic surfaces in $\mathcal{M}_{3,3}^{C_4}$ whose automorphism group is isomorphic to $C_4$ or $C_8$.
\end{enumerate}
\end{proposition}
\begin{proof} Observe $\mathcal{M}_{3,3}^{C_4}$ is not smooth but it is irreducible. Since $W$ is smooth, the natural map $W \to \mathcal{M}_{3,3}^{C_4}$ factors through the normalization, inducing a map by Proposition~\ref{prop:stacky-MG-normalization}:
\begin{align*}
    W &\to [U_{3,3}^{C_4}/N_{\PGL_4}(C_4)] \\
    a &\mapsto S_a.
\end{align*}
This map is $C_4$-equivariant and we can see it factors through the quotient by $C_8$ since the generator for $C_8$ normalizes $C_4$ in $\PGL_4$. Modding out by this action, it is clear we get an open immersion
\[
\mathcal{X} \to [U^{C_4}/N_{\PGL_4}(C_4)].
\]
To prove the statements in the proposition, we want to see that the Type III cubic surface is not in the image of the morphism above. To see this, we regard the Type III cubic as a morphism of stacks
\[
BC_{12} = BN_{H_3(3)\rtimes C_4}(C_4) \to [U^{C_4}/N_{\PGL_4}(C_4)].
\]
In particular this cannot lie in the image of $[W/C_8]$ since $C_{12}$ cannot be a subgroup of $C_8$.
\end{proof}

Let $f\colon\widetilde{\cL}|_{\cX}=:\cY\rightarrow\cX$ be the Galois cover of lines pulled back to $\cX$. Note that $C_8=\Aut(S_0)$ injects into the monodromy group of $\Mon(f)$ by Lemma~\ref{lem:injectivity-hidden-loops}. Let $\cE\rightarrow\cX$ be the maximal finite étale cover which both $f:\widetilde{\cL}|_{\cX}=:\cY\rightarrow\cX$ and $W\rightarrow\cX$ factor through (Remark~\ref{rmk: maximal intermediate cover}). Let $\cZ:=\cY\times_{\cE}W\rightarrow W$ the Galois cover obtained by pullback, which by construction is a connected component of the cover of lines over Dolgachev's family. We get a commutative diagram
\begin{equation}\label{eq: diagram C4}
\begin{tikzcd}
    \cZ\arrow[d]\arrow[r] & W\arrow[d]\\
    \cY\arrow[r] & \cX
\end{tikzcd}
\end{equation}
To compute $\Mon(\cZ\rightarrow W)\leq\Mon(\cY\rightarrow\cX)$ we identify the field extension of $\C(a)$ obtained by equation of lines on $S_a$.

\begin{lemma}\label{lem: Eckardt lines C4 dield of definition}
    The field $\C(a)\subset K$ of definition of Eckardt  lines on Dolgachev's family $W$ is the splitting field of the polynomial $p(t)=t^3+t-a$. In particular, $\Mon(\cZ\rightarrow W)$ has $S_3$ as a quotient.
\end{lemma}
\begin{proof}
    We compute the equation of the three Eckardt lines, using the notation of equation~\eqref{eqn:dolg-family-c4}. First, note that the Eckardt involution is $[1:1:1:-1]$ and the Eckardt point is $p=[0:0:0:1]$. The Eckardt plane has equation $z=0$, as its intersection with the cubic surface has equation
    \[
        z=x^3+xy^2+ay^3=0,
    \]
    hence the three lines intersect at $p$. Write $x^3+xy^2+ay^3=(x+\alpha_{1}y)(x+\alpha_{2}y)(x+\alpha_{3}y)$. By equating the coefficients, we get that $\alpha_1+\alpha_2=-\alpha_3$, and $\alpha_1\alpha_2=1-\alpha_3(\alpha_1+\alpha_2)=1+\alpha_3^2$. Moreover, $a=\alpha_1\alpha_2\alpha_3=\alpha_3+\alpha_3^3$. It follows that the field of definition of the Eckardt lines is $K:=\C(a)(\alpha_1,\alpha_2,\alpha_3)$. Note that the discriminant of the polynomial $t^3+t-a=0$ is $-4-27a^2$, which is not a square in $\C(a)$, hence $\Gal(K/\C(a))\cong S_3$.
\end{proof}

Now, to compute the extension field of lines it is enough to compute the field of definition of the non-Eckardt lines intersecting one chosen Eckardt line $L$. Let $\alpha$ be one between one of the $-\alpha_i$ appearing in the proof of Lemma~\ref{lem: Eckardt lines C4 dield of definition}, thus $\alpha^3+\alpha+a=0$. A tritangent plane containing the line $z=x-\alpha y=0$ has equation $x=\alpha y+cz$, for a parameter $c$ we want to determine. Substitute this into our equation to get the product of $z$ (corresponding to the line $L$) with the conic
\[
    w^2+c(3\alpha^2+1)y^2+(3\alpha c^2+1)yz+c^3z^2=0.
\]
The condition for the conic to split into two lines is provided by the discriminant being zero:
\[
    (3\alpha^2+4)c^4-6\alpha c^2-1=0.
\]
Writing $u=c^2$ and solving the above equation, we get 4 distinct solutions, corresponding to the 4 tritangent planes. The solutions are

\[
    \pm\left(\sqrt{\frac{3\alpha_i\pm 2\sqrt{3\alpha_i^2+1}}{3\alpha_i^2+4}}\right).
\]

Now, let $c$ be such a solution. Then, $c(3\alpha^2+1)y^2+(3\alpha c^2+1)yz+c^3z^2$ can be completed to a square, which needs to be of the form
\[
    c(3\alpha^2+1)y^2+(3\alpha c^2+1)yz+c^3z^2=\frac{1}{c^3}\left(\frac{3\alpha c^2+1}{2}y+c^3 z\right)^2.
\]

The equation of the conic becomes the sum of $w^2$ with the above expression, which splits if we add $\sqrt{c}$.
We have proved the following.

\begin{lemma}\label{lem: computation monodromy C4 Dolgachev family}
    The field extension $\C(a)\subset F$ induced by the cover of lines is
    \[
        F=\C(a)\left(\sqrt[4]{\frac{3\alpha_i\pm 2\sqrt{3\alpha_i^2+1}}{3\alpha_i^2+4}}\right).
    \]
    In particular, the Galois group of the Galois cover $\cZ\rightarrow W$ is
    \[
        \Mon(\cZ\rightarrow W)\cong U_2(\mathbb{F}_3).
    \]
    Moreover, its image along the inclusions $\Mon(\cZ\rightarrow W)\hookrightarrow\Mon(\cY\rightarrow\cX)\hookrightarrow W(E_6)$ identifies $\Mon(\cZ\rightarrow W)$ with the centralizer $C_{W(E_6)}(C_4)$ of $C_4$ in $W(E_6)$.
\end{lemma}
\begin{proof}
    The computation of the extension $\C(a)\subset F$ has been carried out above. The isomorphism with $U_2(\mathbb{F}_3)$ follows from an explicit Galois computation. The last part follows either by checking that $C_4$ is in the center of $\Mon(\cZ\rightarrow W)$ or by analyzing the subgroups of the normalizer $N_{W(E_6)}(C_4)$ of $C_4$ in $W(E_6)$, which we know to contain $\Mon(\cY\rightarrow\cX)$ (and hence $\Mon(\cZ\rightarrow W)$).
\end{proof}

We are ready to compute the $C_4$-equivariant monodromy of the cover of lines.

\begin{theorem}\label{thm: C4 monodromy}
    We have
    \begin{align*}
        \Monstack_{C_4}(\LinesProblem) &\cong C_{W(E_6)}(C_4)\cong U_2(\mathbb{F}_3)\\
        \Mongeom_{C_4}(\LinesProblem)& \cong S_4.
    \end{align*}
\end{theorem}
\begin{proof}
    Since $\cX$ is smooth and $\cM_{3,3}^{=C_4}$ is open and dense in $\cX$, by Proposition~\ref{prop:invariance-properties-monodromy}\eqref{point:restriction-to-open} we have that $\Monstack_{C_4}(\LinesProblem)\cong\Mon(\cY\rightarrow\cX)$, where $\cY=\widetilde{\cL}|_{\cX}$. By~\Cref{lem: computation monodromy C4 Dolgachev family} and Proposition~\ref{prop:invariance-properties-monodromy}\eqref{point:injectivity-pullback}, we know that $C_{W(E_6)}(C_4)\leq\Mon(\cY\rightarrow\cX)$ as subgroups of $W(E_6)$, hence it is enough to show that the image of $\pi_1(\cX)\rightarrow W(E_6)$ commutes with $C_4\leq W(E_6)$. Diagram~\eqref{eq: diagram C4} yields a short exact sequence
    \[
        1\rightarrow \pi_1(W)\rightarrow \pi_1(\cX)\rightarrow C_8\rightarrow 1,
    \]
    and the inclusion $C_8\cong\Aut(S_0)\hookrightarrow\pi_1(\cX)$ induces a section to the last map. In particular, $\pi_1(\cX)$ is generated by $\pi_1(W)$ and $C_8$, and the canonical image of the generic automorphism group $C_4$ in $\pi_1(\cX)$ is contained in $C_8$ by construction. Since $C_8$ is abelian, and by~\Cref{lem: computation monodromy C4 Dolgachev family}, the images in $W(E_6)$ of both $C_8$ and $\pi_1(W)$ commute with $C_4$. Therefore, the same holds for the image of $\pi_1(\cX)\rightarrow W(E_6)$, that is $\Mon(\cY\rightarrow\cX)\leq C_{W(E_6)}(C_4)$, thus completing the proof. 
\end{proof}

The marked and parameter-level monodromy follow as usual. 

\subsection{Type X}
In this setting, we are considering cubic surfaces fixed by the Klein four-group $G:=C_2\times C_2$. We prove the following:

\begin{theorem}\label{thm:Type-X-monodromy} The monodromy groups for lines on $C_2\times C_2$-symmetric cubic surfaces are
\begin{align*}
    \Monstack_{C_2\times C_2}(\LinesProblem) &\cong ((C_2 \times C_2) \rtimes C_2) \times S_4 \\
    \Mongeom_{C_2 \times C_2}(\LinesProblem) &\cong C_2\times S_4 \\
    \Monmarked_{C_2\times C_2}(\LinesProblem)\cong\Monpar_{C_2\times C_2}(\LinesProblem) &\cong C_2 \times C_2\times S_4.
\end{align*}
\end{theorem}

As $N_{W(E_6)}(C_2\times C_2)\cong((C_2\times C_2)\rtimes C_2)\times S_4$, the inclusions from left to right always hold, so we have to argue equality.

Recall that a cubic surface $S$ of Type X has exactly two Eckardt points $p_1$ and $p_2$ that lie on a line $L\subset S$. We denote by $L_1$, $M_1$ the other two lines on $S$ passing through $p_1$, and  by $L_2$, $M_2$ the lines on $S$ passing through $p_2$. In particular, the cover $\widetilde{\cL}^{=G}\rightarrow\cM^{=G}$ splits as the union of three covers:
\begin{enumerate}
    \item one of degree 1, thus an isomorphism, corresponding to $L$;
    \item one of degree 4, corresponding to the other 4 Eckardt lines;
    \item one of degree 22, corresponding to the remaining non-Eckardt lines.
\end{enumerate}

Moreover, the action of $C_2\times C_2$ on $S$ is generated by the two Eckardt involutions $\sigma_1$, $\sigma_2$ centered at $p_1$ and $p_2$ respectively, see~\cite[\S9.1.4]{Dolgachev} and~\cite[\S6.2]{brundu2023construction} for a discussion. By~\cite[\S6.2]{brundu2023construction}, we also know that $\sigma_1:L_2\leftrightarrow M_2$, $\sigma_1$ fixes $L_1$ and $M_1$, and similarly for $\sigma_2$.

To prove Theorem~\ref{thm:Type-X-monodromy} we need to construct elements in the monodromy group. There are a few strategies to do so, for instance using the monodromy groups of deeper strata of $\cM_{3,3}$ that we have already computed. To do so, we need to resort to $\cM_{3,3,G}$, as we know it to be smooth thus allowing us to apply~\ref{prop:invariance-properties-monodromy}\eqref{point:restriction-to-open}.

Recall that $\cM_G:=\cM_{3,3,G}$ parametrizes pairs $(S,\rho)$ where $S$ is a smooth cubic surface and $\rho:G\hookrightarrow\Aut(S)$ is an injective homomorphism from $G$ to the automorphism group of $S$, and $\cM_{=G}\subset\cM_G$ is the open substack where $\rho$ is required to be an isomorphism (Definition~\ref{def:XG-stacks}). Note that for any surface $S$ with automorphism group isomorphic to $C_2\times C_2$, there are two possible types of identifications $\Aut(S)\cong C_2\times C_2$: one in which the two generators $(-1,1)$ and $(1,-1)$ correspond to Eckardt involutions, and those where this is not true. There are 2 of the first kind, and 4 of the second, compatibly with the fact that $\Aut(C_2\times C_2)\cong S_3$.

\begin{definition}\label{def:Eckardt-representation-C2C2}
    Let $G=C_2\times C_2$. We denote by $\cM_{=G}^E$ the full subcategory of $\cM_{=G}$ of pairs $(S,\rho)$ such that $\rho:C_2\times C_2\xrightarrow{\cong}\Aut(S)$ sends $(-1,1)$ and $(1,-1)$ to the two Eckardt involutions, in some order.
\end{definition}

\begin{remark}\label{rmk:Eckardt-representation-C2C2}
    Note that $\cM_{=G}^E$ is an open and closed substack of $\cM_{=G}$. Moreover, by Proposition~\ref{prop:X-lower-G-properties} the composite of
    \[
    \begin{tikzcd}
        \cM_{=G}^E\subset\cM_{=G}\arrow[r,"\Phi_{=G}"] & \cM^{=G}
    \end{tikzcd}
    \]
    is a $C_2$-torsor, where $C_2\leq\Aut(C_2\times C_2)\cong S_3$ is the generated by the involution swapping the two factors.
\end{remark}

In order to apply Proposition~\ref{prop:invariance-properties-monodromy}\eqref{point:restriction-to-open}, we need to understand the closure of $\cM_{=G}^E$ in $\cM_G$, for which we need the following definition (\cite[\S3.3]{landiStacksMonodromySymmetric2025}). We will give the definition in general for any $G$ as it will be used in the proof of Theorem~\ref{thm:Type-X-monodromy} also for groups different from $C_2\times C_2$.

\begin{definition}\label{def:MGPV}
    Let $G$ be any group that is the automorphism group of a smooth cubic surface, and let $V$ be a $G$-representation $V$, with associated projective representation $\P V$. We denote by $\cM_G^{\P V}$ the full subcategory of $\cM_G$ of pairs $(S,\rho)$ such that action of $G$ on $\P H^0(\omega_S^{\vee})\cong\P^3$ yields a projective representation isomorphic to $\P V$.
\end{definition}

\begin{remark}\label{rmk:MGPV}
    Since we work over a field of characteristic 0, $G$ is linearly reductive, hence $\cM_G^{\P V}$ is an open and closed substack of $\cM_G$ (\cite[Lemma 3.23]{landiStacksMonodromySymmetric2025}). See the discussion in~\cite[\S3.3]{landiStacksMonodromySymmetric2025} for more details.
\end{remark}

\begin{lemma}\label{lem: Type X projective representation}
    Let $G=C_2\times C_2$. For $i=1,2$, let $V_{1,i}$ be the 1-dimensional representation of $C_2\times C_2$ where the $i$-th factor acts non-trivially, while the other acts trivially. Let $V:=V_{1,1}\oplus V_{1,2}\oplus\mathds{1}^{\oplus2}$. Then, the closure of $\cM_{=G}^E$ in $\cM_G$ is contained in $\cM_G^{\P V}$.
\end{lemma}
\begin{proof}
    This follows from the fact that the action of $C_2\times C_2$ is generated by the Eckardt involutions, see the discussion above. Alternatively, we can choose $x_2=x_3=0$ to be the equation of the line $L$ passing through both Eckardt points, uniformly for all cubic surfaces of Type X. Let $H_1$ and $H_2$ be the fixed point hyperplanes of $\sigma_1$ and $\sigma_2$, respectively. Then, the line $H:=H_1\cap H_2$ can be set to be defined by $x_0=x_1=0$. Then, both $\sigma_1$ and $\sigma_2$ acts trivially on $H$, and restrict to different non-trivial involutions on $L$. The statement follows.
\end{proof}

\begin{lemma}\label{lem: Type X irreducibility}
    For $G=C_2\times C_2$, the stack $\cM_G^{\P V}$ is smooth and irreducible. In particular, $\cM_{=G}^E$ is irreducible and $\cM_G^{\P V}$ coincides with the closure of $\cM_{=G}^E$ in $\cM_G$.
\end{lemma}
\begin{proof}
    The smoothness is known to hold in general, see~\cite[Proposition 3.6]{landiStacksMonodromySymmetric2025} and Remark~\ref{rmk:MGPV}. For the irreducibility, we describe $\cM_G^{\P V}$ better. Let $(S,\rho)$ be an object over a scheme $T$. Eventually after passing to a Zariski-cover, we can assume that $S\subset\P^3_T$ and the $G$-action on $\P^3_T$ is generated by the matrices
    \[
        \sigma_1=\begin{bmatrix}
            1 & 0 & 0 & 0\\
            0 & 1 & 0 & 0\\
            0 & 0 & -1 & 0\\
            0 & 0 & 0 & 1
        \end{bmatrix}
        \qquad
        \sigma_2=\begin{bmatrix}
            1 & 0 & 0 & 0\\
            0 & 1 & 0 & 0\\
            0 & 0 & 1 & 0\\
            0 & 0 & 0 & -1
        \end{bmatrix}
    \]
    The polynomials that are invariant under this action are of the form
    \[
        f(x,y,z,w)=f_3(x,y)+z^2g_1(x,y)+w^2h_1(x,y)
    \]
    with $f_3$ is homogeneous of degree 3, while $g_1$ and $h_1$ are homogeneous of degree 1. In particular, there is a surjection from the open subset of $\A^8$ to $\cM_G^{\P V}$, corresponding to smooth polynomials $f$ as above. Since $\A^8$ is irreducible, also $\cM_G^{\P V}$ is.

    The last part of the statement follows from the fact that we have an open immersion $\cM_{=G}^E\subset\cM_G^{\P V}$ by Lemma~\ref{lem: Type X projective representation}, $\cM_G^{\P V}$ is (open and) closed in $\cM_G$ by~\cite[Lemma 3.23]{landiStacksMonodromySymmetric2025}, and $\cM_G^{\P V}$ is irreducible.
\end{proof}

We are ready to prove the monodromy groups of Type X cubic surfaces.

\begin{proof}[Proof of Theorem~\ref{thm:Type-X-monodromy}]
    Note that the geometric monodromy will follow from the stacky one by Theorem~\ref{thm:stacky-basic-bound}, and $\Monpar_{C_2\times C_2}(\LinesProblem)\cong\Monmarked_{C_2\times C_2}(\LinesProblem)$ by Corollary~\ref{cor:comparison-monodromy-UG-MG} since $C_{\PGL_4}(C_2\times C_2)$ is connected by~\ref{thm:irred-UG}. Moreover, by Remark~\ref{rmk:Eckardt-representation-C2C2} the morphism $\cM_{=C_2\times C_2}^E\rightarrow\cM_{3,3}^{=C_2\times C_2}$ is a $C_2$-torsor, and by Lemma~\ref{lem: Type X irreducibility} the total space is irreducible. It follows by Theorem~\ref{thm: monodromy X_G vs X^G}\eqref{point 2} that $\Monstack_{C_2\times C_2}(\LinesProblem)$ is a $C_2$-extension of $\Monmarked_{C_2\times C_2}(\LinesProblem)$, hence it is enough to prove Theorem~\ref{thm:Type-X-monodromy} for the marked monodromy. In turn, this was already proved in~\cite[Theorem 2, nK4]{pichon-pharabodGaloisGroupsSymmetric2025} with certified monodromy tracking. Nevertheless, we give an alternative more conceptual proof.

    By~\ref{prop:invariance-properties-monodromy}\eqref{point:restriction-to-open} and Lemma~\ref{lem: Type X irreducibility}, we know that
    \[
        \Monmarked_{C_2\times C_2}(\LinesProblem)\cong\Mon(\cL_{C_2\times C_2}^{\P V}\rightarrow\cM_{C_2\times C_2}^{\P V})\leq C_{W(E_6)}(G)\cong C_2^2\times S_4,
    \]
    which has order 96. Let $V_2$ be the standard 2-dimensional representation of $S_3$, and let $W_2=\mathds{1}\oplus-\mathds{1}\oplus V_2$, where $-\mathds{1}$ is the sign representation. By~\cite[Lemma 4.16, Theorem 4.17]{landiStacksMonodromySymmetric2025}, the closure of one of the two connected components $\cM_{=S_3\times C_2}$ in $\cM_{S_3\times C_2}$ is $\cM_{S_3\times C_2}^{\P W_2}$, which is irreducible. Moreover, $W_2$ restricted to a specific subgroup $C_2\times C_2\leq S_3\times C_2$ is isomorphic to the representation $V$ of \Cref{lem: Type X irreducibility}. It follows that there is a finite morphism $\cM_{S_3\times C_2}^{\P W_2}\rightarrow\cM_{C_2\times C_2}^{\P V}$. By~\cite[Theorem 4.19]{landiStacksMonodromySymmetric2025} or~\cite[Theorem 2]{pichon-pharabodGaloisGroupsSymmetric2025}, we have that $\Monmarked_{S_3\times C_2}(\LinesProblem)\cong S_3\times C_2$, which then injects into $\Monmarked_{C_2\times C_2}(\LinesProblem)$. By Theorem~\ref{thm:stacky-basic-bound}, we also have that the automorphism group itself $C_2\times C_2$ injects into the monodromy (note that it is abelian). Together, they generate a subgroup $C_2\times C_2\times S_3\leq C_2\times C_2\times S_4$, where the product of the two copies of $C_2$ on the left and right correspond to each other.
    
    Now, consider the cover $\cM^{=S_4}_{E}$ of $\cM_{3,3}^{=S_4}$ that parametrizes $S_4$-symmetric cubic surfaces together with a marked Eckardt point. As every Type V cubic surface has exactly 6 Eckardt points, the cover is finite étale of degree 6. It follows that the index of the monodromy group of the cover of lines over $\cM^{=S_4}_{E}$ in $\Monstack_{S_4}(\LinesProblem)$ divides $6$. Moreover, for each Eckardt point $e$ on an $S_4$-symmetric surface $S$, there is a unique distinct Eckardt point $e'\in S$, that we call sibling, that lies on a same line on the surface; see \cite[Figure 2]{brundu2023construction}. This yields a morphism $\cM_{E}^{=S_4}\rightarrow\cM_{C_2\times C_2}^{\P V}$ by sending a pair $(S,e)$, with $e\in S$ Eckardt point with sibling $e'$, to the surface $S$ with the $C_2\times C_2$ action induced by the two Eckardt involutions associated to $e$ and $e'$, in this order. This shows that the monodromy contains a subgroup of the $\Monstack_{S_4}(\LinesProblem)\cong C_2\times C_2\times S_4$ of index dividing 6. Putting everything together we get the statement.
\end{proof}

\subsection{Type XI} Finally, we consider the most general case, where the automorphism group is $C_2$.

\begin{proposition} The monodromy groups of lines on $C_2$-equivariant cubic surfaces are
\begin{align*}
    \Monstack_{C_2}(\LinesProblem)\cong\Monmarked_{C_2}(\LinesProblem)\cong\Monpar_{C_2}(\LinesProblem) &\cong\GO_4^+(3)\\
    \Mongeom_{C_2}(\LinesProblem) &\cong \PGO_4^+(3).
\end{align*}
\end{proposition}
\begin{proof}
The computation of the stacky and geometric monodromy groups are in \cite[Theorem 4.24]{landiStacksMonodromySymmetric2025}. Since the centralizer of $C_2$ in $\PGL_4$ is connected, the marked and parameter-level monodromy agree, and are the centralizer of $C_2$ in the stacky monodromy.
\end{proof}

\section{The moduli stacks of quartics and bitangents}\label{sec:correspondence-thm}

In this subsection we discuss the moduli spaces and stacks parametrizing embedded plane quartics and non-hyperelliptic genus 3 curves, together with their bitangents. Recall following \Cref{nota:Und-Mnd} that $U_{2,4}$ parametrizes embedded plane quartics, and has a natural action of $\PGL_3$ with quotient stack $\mathcal{M}_{2,4}$. This is isomorphic to the substack of genus 3 curves which are not hyperelliptic.

We are interested in constructing a natural cover of the stack $\mathcal{M}_{2,4}$ which will parametrize bitangents to plane quartics. Recall that a family of bitangents to a family $q\colon Q\rightarrow T$ of plane quartics is simply a family of lines $B\subset\P(q_*\omega_{Q/T})$ such that for every $t\in T$ the restriction $B_t\subset\P(q_*\omega_{Q/T})_t\cong\P^2$ is a bitangent to the quartic $Q_t$.

Let $\P^{2\vee}$ denote the dual space of the projective plane, whose points parametrize lines on the plane.
\begin{definition}\label{def:bitan-space}
We denote by $\BPQ \subseteq U_{2,4} \times \P^{2\vee}$ the incidence variety
\begin{align*}
    \BPQ = \{(F,\ell) \in U_{2,4}\times \P^{2\vee} : \ell \text{ is a bitangent to }V(F) \}.
\end{align*}
\end{definition}

Since $\PGL_3$ acts diagonally on both $U_{2,4}$ and $\P^{2\vee}$, we can make the following definition:

\begin{definition}\label{def:bitan-stack} We denote by $\mathcal{B}$ the stack of bitangents to plane quartics, defined as the quotient stack
\begin{align*}
    \mathcal{B}:= [\BPQ / \PGL_3].
\end{align*}
\end{definition}

Since the map $\BPQ \to U_{2,4}$ is $\PGL_3$-equivariant, it descends to a finite \'etale morphism of quotient stacks $\mathcal{G} \colon \mathcal{B} \to \mathcal{M}_{2,4}$.

\begin{notation} We denote by
\[
\BitangentsProblem = (\BPQ \to U_{2,4}, \PGL_3)
\]
the problem of bitangents on smooth plane quartics. We observe that it satisfies \Cref{setup:all-four-monodromies}.
\end{notation}

\begin{theorem} We have that
\[
\Monpar_e(\BitangentsProblem) \cong \Monstack_e(\BitangentsProblem) \cong \Mongeom_e(\BitangentsProblem) \cong \Monmarked_e(\BitangentsProblem) \cong W^+(E_7),
\]
where $W^+(E_7)$ is the index two subgroup of $W(E_7)$ obtained by splitting off the Geiser involution.
\end{theorem}
\begin{proof}
Exhibiting $W^+(E_7)$ as an upper bound on the parameter-level Galois group is classical, and dates back to \cite[III.VI]{jordanTraiteSubstitutions1870}. A rigorous proof that this is the parameter-level monodromy group can be found in \cite[\S4]{Harris-Galois}. The correspondence of the three monodromy groups follows by \Cref{prop:G-is-e-all-four-agree}.
\end{proof}

Similarly to the stack of lines on cubic surfaces, the action of $\PGL_3$ on the Galois closure of the bitangents cover $\widetilde{\BPQ}$ is free, and thus the quotient is a scheme, as shown in the next lemma.
First, recall that the automorphism group of any plane quartic $Q\subset\P^3$ is naturally a subgroup of $\PGL_3$, as it is canonically embedded. For an abstract non-hyperelliptic curve of genus 3, such subgroup is unique only up to conjugation. In both cases, there is a natural action of $\Aut(Q)$ on the set of bitangents.

\begin{lemma}\label{lem: automorphism act faithfully on bitangents}
    Let $Q$ be a smooth plane quartic. Then, $\Aut(Q)$ acts faithfully on the set of bitangents.
\end{lemma}
\begin{proof}
    Non-hyperelliptic curves $Q$ are canonically embedded, so automorphisms of $Q$ extend to projective linear automorphisms. We see that a projective linear automorphism fixing the $28$ bitangents must be the identity. To observe this, we pass to the dual plane $(\P^2)^\vee$ where the 28 bitangents become 28 points. The fixed locus on the projective plane of any element in $\PGL_3$ is at most a line with a point off the line. This implies that 27 of the 28 bitangents are coincident at a single point, which violates the Pl\"ucker formula for the dual curve of the quartic.
\end{proof}

\subsection{Correspondence between Cubic Surfaces and Quartics}

In this subsection we show that the stack parameterizing cubic surfaces equipped with a basepoint not lying on any of their $27$ lines is naturally isomorphic to the stack parametrizing quartics with a chosen bitangent. The construction over a field is classical, see for instance~\cite{Harris-Galois}.

Recall that we denote by $\cM_{3,3}$ the 4-dimensional DM-stack parametrizing smooth cubic surfaces.

\begin{definition}\label{def:punctured-universal-family-cubic-surfaces}
    We denote by $\pi\colon\cS\rightarrow\cM_{3,3}$ the universal family of cubic surfaces, whose $T$-points correspond to pairs $(S,p)$ with $S\rightarrow T$ a family of cubic surfaces and $p\colon T\rightarrow S$ a section.
    
    We denote by $\cSopen$ the open substack of $\cS$ given by the complement of the lines lying on the cubic surfaces, that is, for all $t\in T$ the marked point $p(t)$ does not intersect any line in $S_t$.
\end{definition}

The main result of this subsection is proving that there is a natural isomorphism $\Phi\colon\cSopen\rightarrow\cB$, such that the pullback of the Galois cover of bitangents over $\cB$ to $\cSopen$ is the pullback of the Galois cover of lines along $\cSopen\rightarrow\cM_{3,3}$.
We start by constructing the morphism $\Phi$.

\begin{construction}
    Let $T$ be a $\mathbb{C}$-scheme and $(S,p)\in\cSopen(T)$, that is, a flat family $f\colon S\rightarrow T$ of cubic surfaces and a section $p\colon T\rightarrow S$ not intersecting any line. As usual, the canonical surjection $f^*f_*\omega_{S/T}^{\vee}\rightarrow\omega_{S/T}^{\vee}$ defines a closed immersion $S\hookrightarrow\P(f_*\omega_{S/T}^{\vee})$, and the section $p$ then corresponds to the data of a line bundle $L$ on $T$ and a surjection $f_*\omega_{S/T}^{\vee}\twoheadrightarrow L$. Set $E:=f_*\omega_{S/T}^{\vee}\otimes L^{\vee}$, and let $F$ be the kernel of the induced surjection $E\rightarrow\mathcal{O}_T$, which is a rank-2 vector bundle. Projecting from $p$ defines a rational map
    \[
    \begin{tikzcd}
        S\subset\P(E)\arrow[r] & \P(F)
    \end{tikzcd}
    \]
    yielding a morphism $\phi\colon\widetilde{S}\rightarrow\P(F)$, where $\widetilde{S}:=\mathrm{Bl}_{p}S$ is the blowup of $S$ along $p$. Then, $\phi$ is a double cover ramified over a smooth family $Q\subset\P(F)$ of plane quartics over $T$. Moreover, the image of the exceptional divisor of $\widetilde{S}$ is a family of bitangents $B\subset\P(F)$ over $T$. The association $(S,p)\mapsto(Q,B)$ defines $\Phi\colon\cSopen\rightarrow\cB$ at the level of objects.

    Now we deal with morphisms; for simplicity, we only consider morphisms over the same scheme $T$. Suppose given two pointed families $(S_1,p_1)$, $(S_2,p_2)$ of cubic surfaces over $T$, and a morphism $\psi\colon S_1\rightarrow S_2$ over $T$ preserving the marking. For every $i$, we denote by $(Q_i,B_i)$ the pairs associated to $(S_i,p_i)$, and with $\phi_i\colon\widetilde{S}_i\rightarrow\P(F_i)$ the morphism from the blowup of $S_i$ along $p_i$ constructed above. Since $\psi$ preserves the marking, it lifts to a morphism $\widetilde{\psi}\colon\widetilde{S}_1\rightarrow\widetilde{S}_2$ preserving the exceptional divisors. In particular, $\psi$ commutes with the involutions defining the covers $\phi_i$'s, hence it descends to a morphism between the quotients $\overline{\psi}\colon\P(F_1)\rightarrow\P(F_2)$. By construction, $\overline{\psi}$ needs to respect the ramification locus of the $\phi_i$'s, hence it restricts to a morphism $Q_1\rightarrow Q_2$. Moreover, as $\widetilde{\psi}$ preserves the exceptional divisors and $B_i$ is the image of it along $\phi_i$, $\overline{\psi}$ sends the bitangent $B_1$ of $Q_1$ to the bitangent $B_2$ of $Q_2$. We have then defined the desired arrow $\overline{\psi}\colon(Q_1,B_1)\rightarrow(Q_2,B_2)$.
\end{construction}

\begin{theorem}\label{thm: isom stacks cubic surfaces and plane quartics}
    The morphism $\Phi\colon\cSopen\rightarrow\cB$ is an isomorphism.
\end{theorem}
\begin{proof}
    We start by showing how to locally recover the pair $(S,p)$ from a pair $(f\colon Q\rightarrow T,B)$ of a family over $T$ of bitangents to plane quartics. First, we can embed $Q\hookrightarrow\P(f_*\omega_{Q/T})$ via the surjection $f^*f_*\omega_{Q/T}\rightarrow\omega_{Q/T}$. After passing to a Zariski-cover of $T$, we can assume $f_*\omega_{Q/T}$ to be trivial, and the ideal sheaf of $Q$ in $\P(f_*\omega_{Q/S})\cong\P^2$ to be $\cO_{\P(F)}(-4)$. To construct the double cover $\phi\colon\widetilde{S}\rightarrow\P(F)$ of $\P(F)$ ramified over $Q$, we set $\widetilde{S}=\underline{\Spec}_{\cO_{\P^2}}(\cO_{\P^2}\oplus\cO_{\P^2}(-2))$, with the $\cO_{\P^2}$-algebra structure induced by the multiplication $\cO_{\P^2}(-2)\otimes\cO_{\P^2}(-2)\cong\cO_{\P^2}(-4)\hookrightarrow\cO_{\P^2}$, where the last inclusion is the one of the ideal sheaf of $Q$. The preimage of the bitangent $B$ splits as the union of two $(-1)$-curves $L_1$, $L_2$, and the two pairs $(\widetilde{S},L_1)$ and $(\widetilde{S},L_2)$ are isomorphic via the involution induced by the double cover $\phi$. Choose one of the two $(-1)$-curves, and contract it via Castelnuovo's criterion to a point $p$ on a cubic surface $S$; this is the desired pair $(S,p)\in\cSopen$.

    Now, we are ready to show that $\Phi$ is an isomorphism by showing that it is fully faithful and essentially surjective. First of all, it is clearly faithful, as a morphism $\psi\colon(S_1,p_1)\rightarrow(S_2,p_2)$ over $T$ is trivial if and only if $\widetilde{\psi}\colon\widetilde{S}_1\rightarrow\widetilde{S}_2$ is trivial, and in turn this happens if and only if the induced $\overline{\psi}\colon\P(F_1)\rightarrow\P(F_2)$ is trivial, as $\widetilde{\psi}$ preserves the exceptional divisor. Finally, since $Q_i$ is canonically embedded in $\P(F_i)$, to check the triviality of $\overline{\psi}$ it is enough to consider its restriction to the plane quartics.
    
    To show that $\Phi$ is full, suppose given a morphism $Q_1\rightarrow Q_2$ with induced linear morphism $\overline{\psi}\colon\P(F_1)\rightarrow\P(F_2)$ sending $B_1$ to $B_2$. Here, $(Q_i,p_i)$ is the $T$-object of $\cB$ associated to $(S_i,p_i)\in\cSopen(T)$. Locally, $\overline{\psi}$ induces a morphism $\widetilde{\psi}\colon\widetilde{S}_1\rightarrow\widetilde{S}_2$ via the root construction, which is unique up to the involution induced by the double covers $\phi_i$. However, there is only such a morphism if we ask it to preserve the exceptional divisors, hence they glue to give a uniquely defined global morphism $\widetilde{\psi}$. As it preserves the exceptional divisors by construction, this in turn induces a unique morphism $\psi\colon S_1\rightarrow S_2$ sending $p_1$ to $p_2$, as wanted.

    Since we have showed $\Phi$ to be fully faithful, to show that the essential image of $\Phi$ contains a family $(Q,B)\rightarrow T$ in $\cB$, we can work locally on the base $T$. Then, the result follows immediately from the construction at the beginning of the proof.
\end{proof}

Our second objective is to compare the cover of bitangents over $\cB$ to the cover of lines over $\cSopen$. The relation is actually classical, so we will limit ourself in stating the result and write only a sketch of the proof.

\begin{definition}
    We denote by $\widetilde{\cB}\rightarrow\cM_{2,4}$ the Galois closure of the (étale finite) cover of the 28 bitangents to plane quartics.
\end{definition}

By definition, $\widetilde{\cB}\rightarrow\cM_{2,4}$ factors through $\cB$.
Recall that we denote by $\widetilde{\cL}\rightarrow\cM_{3,3}$ the Galois closure of the cover of 27 lines on cubic surfaces.

\begin{theorem}\label{thm: lines bitangents correspondence}
    The isomorphism $\Phi:\cSopen\xrightarrow{\cong}\cB$ induces a cartesian diagram
    \[
    \begin{tikzcd}
        \widetilde{\cL}\times_{\cM_{3,3}}\cS^{\circ}\arrow[r]\arrow[d] & \widetilde{\cB}\arrow[d]\\
        \cSopen\arrow[r,"\Phi","\cong"'] & \cB
    \end{tikzcd}
    \]
\end{theorem}
\begin{proof}
    This follows from \Cref{thm: isom stacks cubic surfaces and plane quartics} and the fact that the image of the 27 lines under the projection $S\rightarrow\P(F)$ are the remaining 27 bitangents distinct from $B$. See for instance~\cite{Harris-Galois} for more details.
\end{proof}

\subsection{Equivariant Correspondence}

In this subsection we try to refine the correspondence proved in Theorem~\ref{thm: lines bitangents correspondence}, by fixing a group $G$ and finding what cubic surfaces map to $G$-equivariant quartics via the isomorphism of Theorem~\ref{thm: isom stacks cubic surfaces and plane quartics}.

As done for $\cM_{3,3}^{=G}$ in~\cite{landiStacksMonodromySymmetric2025}, it is easy to show that $\cM_{2,4}^{=G}$ is a smooth DM-stack.
Recall that $\cM_{2,4}^{=G}$ is a smooth stack by Proposition~\ref{prop:X-lower-G-properties}. Moreover, it is irreducible by~\cite[Theorem 6.5.2]{Dolgachev} or~\cite[Theorem 16]{barsAutomorphismsGroupsGenus}.

The table in~\cite[Theorem 1.1]{betheabrazelton_bitangentssymmetricquartics} also summarizes the automorphism groups of smooth plane quartics. Moreover, for every such $G$, it computes the decomposition of the set of 28 bitangents in $G$-orbits, in particular computing the corresponding stabilizers. This decomposition does not depend on the $G$-symmetric cubic surface chosen, hence it induces a decomposition of $\cB^{=G}\rightarrow\cM_{2,4}^{=G}$ as the disjoint union of sub-covers, according to the stabilizer of the bitangent.

\begin{definition}\label{def: decomposition B =G}
    Given an inclusion $i:H\hookrightarrow G$ we denote by $\cB^{(G,H,i)}$ the open and closed substack of $\cB^{=G}$ of pairs $(Q,B)$ such that the inclusion $\Aut(Q,B)\hookrightarrow\Aut(Q)$ is isomorphic to $i:H\hookrightarrow G$.

    Similarly, we denote by $\cS^{(G,H,i)}$ the substack of $\cS^\circ$ of pairs $(S,p)$ such that the inclusion $\Aut(S,p)\hookrightarrow\Aut(S)$ is isomorphic to $i:H\hookrightarrow G$.
    
    When $i$ is clear or not relevant, we omit it in the above definition notations, and simply write $\cB^{(G,H)}$ and $\cS^{(G,H)}$, respectively.
\end{definition}

We want to refine the correspondence isomorphism in Theorem~\ref{thm: isom stacks cubic surfaces and plane quartics} to give a correspondence of strata above. To obtain a full correspondence is very complicated, but it is possible to understand the correspondence for a few strata. The first step is to study the covers $\cS^{(G,H)}\rightarrow\cM_{3,3}^{=G}$ for some interesting $G$ and $H$.

\subsection{Case $H=C_2$}\label{subsec: H=C2 case}

Since every type of cubic surfaces contain a subgroup isomorphic to $C_2$, completely describing all $\cS^{(G,C_2)}$ would be complicated. We will need much less; the following is the main result we will use.

\begin{lemma}\label{lem: C2 locus}
    Let $S$ be a smooth cubic surface, and let $C_2$ act faithfully on $S$ via automorphisms. Then, $S^{C_2}$ is either:
    \begin{enumerate}
        \item a union of a smooth plane cubic and an Eckardt point. In this case, the involution is generated by an Eckardt involution, and for such involutions $\cS^{(G,C_2)}\rightarrow\cM_{3,3}^{=G}$ is a fibration in punctured plane cubics.
        \item the union of a line and three points, each lying on some line. In this case, the involution is the composite of two commuting Eckardt involution, and $\cS^{(G,C_2)}$ is empty.
    \end{enumerate}
    Moreover, all symmetric surfaces admit an involution of the first kind, while the second kind of involution is present only for cubic surfaces that are specialization of $C_2\times C_2$-symmetric ones, that is, Type I, II, V, VI, and X.
\end{lemma}
\begin{proof}
    This follows from~\cite[\S9.5]{Dolgachev}, and the fact that the fixed point locus of an action of a linearly reductive group on a smooth scheme is again smooth. Note that in~\cite[Table 9.5]{Dolgachev} it is inaccurately reported that the fixed locus of an involution of the second kind consists of a line and one isolated point, while there are 3 isolated points; this is actually proved in the same section~\cite[\S9.5]{Dolgachev}. We show that each of the 3 isolated points lies on some line. Recall that, up to projective transformation, we can assume the cubic surface to have equation
    \[
        xz(z+aw)+yw(w+bz)+x^3+y^3
    \]
    for some $a$ and $b$, and the involution to act as $[1:1:-1:-1]$. Then, the fixed locus consists of the line $x=y=0$ and the isolated points $[1:-\zeta_3^i:0:0]$ for $i=0,1,2$. Intersect the surface with the plane $y=-\zeta_3^ix$ to obtain the homogeneous equation
    \[
        x[z(z+aw)-\zeta_3^iw(w+bz)],
    \]
    which is the product of three linear factors.
\end{proof}

The above lemma shows that we can restrict our attention to Eckardt involution, and for such we have that the morphism $\cS^{(G,C_2)}\rightarrow\cM^{=G}$ is a fibration in punctured plane cubics.

\subsection{Case $H=C_3$}\label{subsec: H=C3 case}

Many types of cubic surfaces contain subgroups isomorphic to $C_3$, hence a full description of all $\cS^{(G,C_3)}$ might seem complicated to obtain. However, by Theorem~\ref{thm: isom stacks cubic surfaces and plane quartics}, we have a correspondence between strata $\cS^{(G,C_3)}$ as $\cB^{(G',C_3)}$. By~\cite[Theorem 1.1]{betheabrazelton_bitangentssymmetricquartics}, there is only one such irreducible stratum $\cB^{(C_3,C_3)}$, that has dimension 2, making the analysis simple.

\begin{lemma}\label{lem: C3 locus}
    Let $S$ be a smooth cubic surface, and let $C_3$ act faithfully on $S$ via automorphisms. Then, $(S^\circ)^{C_3}$ is either empty, finite, or a smooth cubic curve. Moreover, the last case is possible only for Type I, III, or IV. In particular, for all inclusions $i:C_3\hookrightarrow G$ such that $\cS^{(G,C_3,i)}\not=\emptyset$, we have $\dim\cS^{(G,C_3,i)}\leq\dim\cM_{3,3}^{=G}+1$, with equality possibly only for Type I, III, or IV.
\end{lemma}
\begin{proof}
    It follows immediately from~\cite[Table 9.5]{Dolgachev}.
\end{proof}

\begin{lemma}\label{lem: Type IV unique cubic surface with C3-points}
    The only group $G$ for which $\cS^{(G,C_3)}$ is non-empty is $G=H_3(3)\rtimes C_2$, corresponding to cubic surfaces of Type IV. Moreover, the only subgroup $i:C_3\hookrightarrow H_3(3)\rtimes C_2$ for which $\cS^{(H_3(3)\rtimes C_2,C_3,i)}\not=\emptyset$ is the inclusion of the center. For that inclusion, $\cS^{(H_3(3)\rtimes C_2,C_3)}\rightarrow\cM_{3,3}^{=H_3(3)\rtimes C_2}$ is a fibration in punctured irreducible smooth plane cubics.
\end{lemma}
\begin{proof}
    By~\cite[Theorem 1.1]{betheabrazelton_bitangentssymmetricquartics}, there is only one non-empty stratum $\cB^{(G',C_3)}$, with $G'=C_3$; it is irreducible and isomorphic to $\cM_{2,4}^{=C_3}$, that has dimension 2. The only group $G$ containing $C_3$ for which $\cM_{3,3}^{=G}$ is 2-dimensional is $G=S_3$, corresponding to Type VIII cubic surfaces. However, one can verify using~\cite[Theorem 4.11]{landiStacksMonodromySymmetric2025} that all $C_3$-invariants points are also $S_3$-invariant; in particular, $\cS^{(S_3,C_3)}=\emptyset$. By Lemma~\ref{lem: C3 locus}, the only possibility left are type IV cubic surfaces, that is $G=H_3(3)\rtimes C_2$, whose $C_3$-fixed point locus consisting of a punctured smooth cubic curves. Explicitly, by~\S\ref{subsec: Type IV cubic surfaces}, every type IV cubic surface $S_a$ is isomorphic to one with equation $x^3+y^3+z^3+w^3-3ayzw$, and the center of the automorphism group acts by $x\mapsto\zeta_3 x$. Then, the $C_3$-fixed locus is the smooth plane cubic $E_a$ with equation $x=y^3+z^3+w^3-3ayzw=0$. This concludes, after deleting from $E_a$ the finite number of points that lie on some line.
\end{proof}

\subsubsection{Case $H=C_6$}\label{subsec: H=C6 case}

The case $H=C_6$ is very special.

\begin{lemma}\label{lem: C6 locus}
    Let $S$ be a smooth cubic surface, and let $C_6$ act faithfully on $S$ via automorphisms. Then, $(S^\circ)^{C_6}$ is either finite or empty. In particular, when $\cS^{(G,C_6)}\not=\emptyset$, we have $\dim\cS^{(G,C_6)}=\dim\cM_{3,3}^{=G}$.
\end{lemma}
\begin{proof}
    This follows from Lemma~\ref{lem: C3 locus}, except possibly for smooth cubic surfaces $S$ isomorphic to one with equation $x^3+y^3+z^3+w^3-3ayzw$, in which case the $C_3$-invariant locus is the cubic $E_S$ with equation $x=y^3+z^3+w^3-3ayzw=0$. Then, the statement follows from noting that all faithful $C_2$-actions on $S$ that preserve $E_S$ restrict to a non-trivial action on $E_S$.
\end{proof}

The above lemma shows in particular that $\cS^{(G,C_6)}$ can be positive dimensional only for Type IV or VI.

\begin{lemma}\label{lem: Type IV C6-pointed cubics}
    There is only one conjugacy class of subgroups isomorphic to $C_6$ in $H_3(3)\rtimes C_2$. Moreover, $\cS^{(H_3(3)\rtimes C_2,C_6)}\rightarrow\cM_{3,3}^{=H_3(3)\rtimes C_2}$ is an irreducible cover of degree 27, while the induced cover between coarse moduli spaces has degree 3 and Galois group isomorphic to $S_3$.
\end{lemma}
\begin{proof}
    The first part is a straightforward verification.
    For the second part, recall that every smooth cubic surface of Type IV is isomorphic to a surface $S_a$ with equation $x^3+y^3+w^3+z^2-3ayzw$; see~\S\ref{subsec: Type IV cubic surfaces}. We can choose the subgroup $C_6$ generated by the automorphism $x\mapsto\zeta_3x$, and $y\leftrightarrow z$. Then, the fixed locus consists of the Eckardt point $[0:1:-1:0]$ and three points of the form $[0:1:1:w]$ with $w^3-3aw+2=0$. These last points do not lie on any line. The computations of the degree and Galois group of the resulting covers follows immediately.
\end{proof}

\begin{lemma}\label{lem: Type VI C6-pointed cubics}
    There is only one subgroup isomorphic to $C_6$ in $S_3\times C_2$. The $C_6$-fixed point locus consists of the Eckardt point whose Eckardt involution generates the central $C_2$ in $S_3\times C_2$. In particular, $\cS^{(S_3\times C_2,C_6)}=\emptyset$.
\end{lemma}
\begin{proof}
    This is an immediate verification using~\cite[Theorem 4.17]{landiStacksMonodromySymmetric2025}.
\end{proof}

\subsubsection{Case $H=S_3$}\label{subsec: H=S3 case}

We focus on the case where $H=S_3$. We start with a general lemma.

\begin{lemma}\label{lem: S3 locus}
    Let $S$ be a smooth cubic surface, and let $S_3$ act on $S$ via automorphisms. Then, $(S^\circ)^{S_3}$ is either finite or empty. In particular, when $\cS^{(G,S_3)}\not=\emptyset$, we have $\dim\cS^{(G,S_3)}=\dim\cM_{3,3}^{=G}$.
\end{lemma}
\begin{proof}
    The fixed locus of any faithful $S_3$-action on $\P^3$ is at most 1-dimensional, and each component is linear. Therefore, $S^{S_3}$ is the union of a finite number of points and lines. Alternatively, see~\cite[Table 9.5]{Dolgachev}.
\end{proof}

By the above lemma, the space $\cS^{(G,S_3)}$ is either empty or 0-dimensional for surfaces of Type I, II, or III, while it is at most 1-dimensional for Type IV, V, and VI.
We only study the 2-dimensional case of Type VIII cubics, as this is the only one we will need.

\begin{lemma}\label{lem: Type VIII S3-pointed cubics}
    The cover $\cS^{(S_3,S_3)}\rightarrow\cM_{3,3}^{=S_3}$ is of degree 3 and Galois group $S_3$; the same holds for the induced cover between moduli spaces. In particular, $\dim\cS^{(S_3,S_3)}=2$.
\end{lemma}
\begin{proof}
    Using the notation of~\cite[Theorem 4.11]{landiStacksMonodromySymmetric2025}, the only $S_3$-invariants points are $[-\zeta_3^i:1:0:0]$, which do not lie on any line. The action of $H\cong S_3\times C_3$ transitively permutes the ordering of those points, and the conclusion for stacks follows. Since the étale cover induces an isomorphism between the stabilizers, it descends to the coarse moduli space by Luna's fundamental lemma.
\end{proof}

\subsubsection{Case $H=C_4$}\label{subsec: H=C4 case}

We focus on the case where $H=C_4$.

\begin{lemma}\label{lem: C4 locus}
    Let $S$ be a smooth cubic surface, and let $C_4$ act faithfully on $S$ via automorphisms. Then, $S^{C_4}$ is either finite or empty. In particular, when $\cS^{(G,C_4)}\not=\emptyset$, we have $\dim\cS^{(G,C_4)}=\dim\cM_{3,3}^{=G}$.
\end{lemma}
\begin{proof}
    It follows immediately from~\cite[Table 9.5]{Dolgachev}.
\end{proof}

By the above lemma, the space $\cS^{(G,C_4)}$ is either empty or $0$-dimensional for surfaces of Type I, II, III, and VII. Therefore, the only interesting cases are Type V and IX.

\begin{lemma}\label{lem: Type V C4-pointed cubics}
    There is only one conjugacy class of subgroups isomorphic to $C_4$ in $S_4$. Moreover, for every smooth cubic surface of type V and any choice of $C_4\leq S_4$, there are exactly three points fixed by $C_4$, and all lie on some line. In particular, $\cS^{(S_4,C_4)}=\emptyset$.
\end{lemma}
\begin{proof}
    The first part is clear. Using the notation of~\cite[Theorem 4.5]{landiStacksMonodromySymmetric2025}, if we take $C_4=\langle(1,2,3,4)\rangle$, then the points fixed by $C_4$ are $[0:1:1:2]$, $[0:i:1:0]$, and $[0:-i:1:0]$. They do lie on lines, hence $\cS^{(S_4,C_4)}=\emptyset$.
\end{proof}

\begin{lemma}\label{lem: Type IX C4-pointed cubics}
    Every smooth cubic surface of Type IX has exactly five $C_4$-invariant points: one is the Ecakrdt point, three lie on some Eckardt line, and only one does not lie on any line. In particular, $\cS^{(C_4,C_4)}\cong\cM_{3,3}^{=C_4}$.
\end{lemma}
\begin{proof}
    Recall that every smooth cubic surface $S$ with $\Aut(S)=C_4$ is isomorphic to one in the 1-dimensional family of cubic surfaces given by $x^3+xy^2+ay^3+yz^2+zw^2$, and the automorphism group is generated by the diagonal matrix with elements $[1:1:-1:i]$; see~\cite[\S9.5.3]{Dolgachev}. Then, the $C_4$-invariant points are the Eckardt point $[0:0:0:1]$, the point $[0:0:1:0]$, and the three points $[x:1:0:0]$ with $x$ satisfying $x^3+x+a=0$. In Lemma~\ref{lem: Eckardt lines C4 dield of definition}, we showed that the Eckardt plane has equation $z=0$, and its intersection with the cubic surface has equation $x^3+xy^2+ay^3$. Therefore, the three $C_4$-invariant points of the form $[x:1:0:0]$ lie in the union of the three Eckardt lines. One can verify that $[0:0:1:0]$ does not lie on any line.
\end{proof}

\section{Galois groups of bitangents to symmetric quartics}\label{sec:galois-groups-of-bitangents}

In this section we compute the stacky, parameter-level, and geometric monodromy groups for the problem of 28 bitangents to a smooth planar quartic curve.

\begin{theorem}\label{thm:main2} Let $G$ be an automorphism group of a smooth planar quartic curve over the complex numbers. Then the parameter-level, stacky, and geometric monodromy groups of solving for the 28 bitangents to the $G$-symmetric planar quartics are as follows:
\input{TABLE-quartics}
\end{theorem}

\begin{remark}\label{rmk:table-notation}
When the group is too large to admit a clean structure description, we've opted to simply include the small group id of the group. When the group is not in the SmallGroups library, we describe it instead via $W^+(E_7)$. Other notation: $W(-)$ is the Weyl group of a Lie group, $\circledcirc$ is a classical notation for central products, $H_3(3)$ is the Heisenberg group, $P$ denotes the Pauli group of order 16.
\end{remark}

We prove \Cref{thm:main2} throughout this section.

\subsection{The Klein quartic (Type I)}

Since the Klein quartic is a unique point in the moduli space, the stacky monodromy is equal to the group itself, the marked-stacky monodromy coincides with the center, while the parameter-level and geometric monodromy are trivial.

\subsection{The Dyck quartic (Type II)}

Similarly to the Fermat, uniqueness of this quartic governs the various monodromy groups.

\subsection{The Type III quartic}

The Type III quartic defines a unique point in the moduli space, hence the stacky monodromy is the entire group, the marked monodromy is its center, and the geometric monodromy is trivial. Unlike the previous two types, we have seen in \Cref{prop:Vtype3} that $U_{2,4}^{C_4\circledcirc A_4}$ is 1-dimensional, which affects the parameter-level monodromy. Indeed, the centralizer $C_{\PGL_3}(C_4\circledcirc A_4)$ is connected by Theorem~\ref{thm:irred-VG}, hence 
\[\Monpar_{C_4\circledcirc A_4}(\BitangentsProblem)\cong\Monmarked_{C_4\circledcirc A_4}(\BitangentsProblem)\cong Z(C_4\circledcirc A_4) \cong C_4.\]

\subsection{Type IV quartics} The automorphism group for Type IV quartics is $S_4$.
\begin{proposition}\label{prop:type4quartics-param} We have that the parameter-level monodromy for bitangents to Type IV quartics is
\begin{align*}
    \Monpar_{S_4}(\BitangentsProblem) \cong C_2 \times C_2.
\end{align*}
\end{proposition}
\begin{proof}
Since $C_2 \times C_2$ is the centralizer of $S_4$ in $W^+(E_7)$, we know we have containment, so it suffices to show this is equality. We can take a normal form for an $S_4$-symmetric quartic as:
\begin{align*}
    x^4 + y^4 + z^4 + a(x^2 y^2 + x^2 z^2 + y^2 z^2).
\end{align*}
Up to $S_4$-symmetry, any line has equation $ux+y+z=0$ or $x+vy=0$. For lines of the first kind, we can plug them into the general form for the quartic, and see when we get a square.

Up to scalars, this gives us
\begin{align*}
(u-1)^{2}(u+1)^{2}(a-2)^{2}(a+2)(u^{2}+a+1)^{2}(au^{4}+2u^{4}+8au^{2}+16).
\end{align*}
The factors with multiplicity 1 are ordinary tangencies, and a few multiplicity two factors we already know: bitangents at $u=\pm 1$, as well as a singularity at $a=2$. The only remaining factor we are left with is $u^2 + a+1$. So these bitangents are defined over $\mathbb{C}(a)(\sqrt{-a-1})$.

For lines of the other type, we get
\begin{align*}
(a-2)^{2}(v^{4}+av^{2}+1)(av^{4}+2v^{4}+2av^{2}+a+2)^{2}.
\end{align*}
Letting $t=v^2$, we get
\begin{align*}
    t = \frac{-2a\pm \sqrt{4a^2 - 4(a+2)^2}}{2(a+2)} = \frac{-a\pm 2u}{a+2}.
\end{align*}
Hence these bitangents are defined over
\begin{align*}
    \mathbb{C}(a,u)\left( \sqrt{\frac{-a\pm 2u}{a+2}}\right).
\end{align*}
Altogether this implies that the parameter-level monodromy contains a Klein-four group, so we must have equality.
\end{proof}

\begin{corollary}\label{cor:type4quartics-stacky-geom} We have that
\begin{align*}
    \Monpar_{S_4}(\BitangentsProblem) &\cong C_2\times C_2\\
    \Monstack_{S_4}(\BitangentsProblem) &\cong S_4 \times C_2 \times C_2 \\
    \Mongeom_{S_4}(\BitangentsProblem) &\cong C_2 \times C_2.
\end{align*}
\end{corollary}
\begin{proof}
The parameter-level monodromy lives in the stacky monodromy but centralizes the automorphism group $S_4$. Since the center of $S_4$ is trivial, these must meet disjointly, therefore $\Monstack_{S_4}(\mathcal{G})$ contains $C_2\times C_2 \times S_4$, and this must be equality since this is the normalizer in $W^+(E_7)$. The geometric and marked monodromy computations follow from Theorem~\ref{thm:stacky-basic-bound} \eqref{eq:isom-geometric-monodromy} and Theorem~\ref{thm: monodromy X_G vs X^G}\eqref{point 2}.
\end{proof}

\subsection{Type V quartics} The automorphism group for Type V quartics is the Pauli group.

To compute the stacky monodromy, we consider the 1-dimensional closed substack $\cB^{(P,C_4)}$ of pairs $(Q,B)$ where $Q$ is a quartic of Type V, with the order 16 automorphism group denoted by $P$, and $\Aut(Q,B)\cong C_4$; see~\cite[Theorem 1.1]{betheabrazelton_bitangentssymmetricquartics} for details. By loc.~cit., the cover $\cB^{(P,C_4)}\rightarrow\cM_{2,4}^{=P}$ has order 4, since $P$ is the only automorphism group of smooth planar quartics admitting a bitangent with isotropy group $C_4$.

We want to find $G$ for which there exists an injection $i:C_4\hookrightarrow G$ such that $\cS^{(G,C_4)}\cong\cB^{(P,C_4)}$.

\begin{lemma}\label{lem:type-of-cubics-mapping-to-P-with-C4-fixed}
    The type of pointed cubic surfaces mapping to $\cB^{(P,C_4)}$ is type IX, with a $C_4$-fixed point, corresponding to the stratum $\cS^{(C_4,C_4)}$.
\end{lemma}
\begin{proof}
Since the family is 1-dimensional, by \S~\ref{subsec: H=C4 case} the only possibilities are Type V or Type IX. Then, Lemma~\ref{lem: Type V C4-pointed cubics} excludes type V.
\end{proof}

\begin{proposition}\label{prop: Type V quartics stacky and geometric monodromy}
The stacky $P$-symmetric monodromy of the cover of bitangents is the hyperoctahedral group $W(B_4)$ (code $[384,18139]$), while the geometric monodromy is isomorphic to $S_4$.
\end{proposition}
\begin{proof}

Let $\widetilde{\cL}^{(C_4,C_4)}\rightarrow\cS^{(C_4,C_4)}$ be the Galois closure of the pullback of the cover of (all) lines to $\cS^{(C_4,C_4)}$. By \Cref{lem:type-of-cubics-mapping-to-P-with-C4-fixed} and the correspondence Theorems~\ref{thm: isom stacks cubic surfaces and plane quartics}, \ref{thm: lines bitangents correspondence}, the order of the stacky monodromy group is
    \begin{align*}
        |\Mon(\widetilde{\cB}^{=P}\rightarrow\cM_{2,4}^{=P})|=&\ |\Mon(\widetilde{\cB}^{=P}\rightarrow\cB^{=P})|\cdot\deg(\cB^{=P}\rightarrow\cM_{2,4}^{=P})\\
        =&\ |\Mon(\widetilde{\cL}^{(C_4,C_4)})\rightarrow\cS^{(C_4,C_4)}|\cdot4\\
        =&\ |\Mon(\cL^{=C_4}\rightarrow\cM^{=C_4})|\cdot4\\
        =&\ 384.
    \end{align*}
    In the third line we used Lemma~\ref{lem: Type IX C4-pointed cubics}, and in the last line we used Theorem~\ref{thm: C4 monodromy}. By Theorem~\ref{thm:stacky-basic-bound}, this implies that the stacky monodromy group is an index two subgroup of $N_{W^+(E_7)}(C_4)$, containing $\Mon(\cL^{=C_4}\rightarrow\cM^{=C_4})\cong U_2(\mathbb{F}_3)$ (Theorem~\ref{thm: C4 monodromy}). There is only one such group, as one can check computationally. The geometric monodromy follows immediately from the stacky computation.
\end{proof}

The parameter-level and marked monodromy follow from Proposition~\ref{prop: Type V quartics stacky and geometric monodromy}, Theorem~\ref{thm: monodromy X_G vs X^G}\eqref{point 2} and Corollary~\ref{cor:comparison-monodromy-UG-MG}\eqref{point:isom-parameter-marked}, as $C_{\PGL_3}(P)$ is connected.

\subsection{Type VI Quartics}

In this case, the moduli stack is 0-dimensional. Therefore, the stacky monodromy group is isomorphic to $C_9$ and the geometric monodromy is trivial, by Theorem~\ref{thm:stacky-basic-bound}. Since $C_9$ is abelian and $C_{\PGL_3}(C_9)$ is connected, the parameter and marked monodromy are also isomorphic to $C_9$.

\subsection{Type VII Quartics}

We first compute the parameter-level monodromy.

\begin{proposition}\label{prop:monpar-quartic-type7}
The parameter-level monodomy group for bitangents to Type VII quartics is $D_8 \times C_2 \times C_2$.
\end{proposition}
\begin{proof} Since $D_8 \le P$, we have inclusions $U^{P}\subset U^{D_8}$, hence $\Monpar_P(\BitangentsProblem)\le \Monpar_{D_8}(\BitangentsProblem)$ by Proposition~\ref{prop:invariance-properties-monodromy}\eqref{point:injectivity-pullback}. Therefore, $\Monpar_{D_8}(\BitangentsProblem)$ has order divisible by $|\Monpar_P(\BitangentsProblem)|=16$, and dividing $|C_{W^+(E_7)}(D_8)|=32$. We will show that we also have a subgroup $D_8 \le \Monpar_{D_8}(\BitangentsProblem)$, and since $D_8 \not\le \Monpar_P(\BitangentsProblem)$, this will show that the subgroup inclusion $\Monpar_P(\BitangentsProblem) \le \Monpar_{D_8}(\BitangentsProblem)$ is strict, and therefore $\Monpar_{D_8}(\BitangentsProblem) = C_{W^+(E_7)}(D_8) = D_8 \times C_2\times C_2$.

To find this $D_8$ in the monodromy group, we consider the subcover of the bitangent cover over $U_{2,4}^{D_8}$ consisting of those 8 bitangents whose isotropy group is conjugate to the cyclic group of order two generated by swapping $x$ and $y$ (see \cite[\S4.7]{betheabrazelton_bitangentssymmetricquartics} for a presentation of $D_8$ in $\PGL_3$). Generically no bitangents are of the form $z=0$, so we can assume the bitangent is given by the formula $x+y + uz$ for some $u$. Restricting to the $z=1$ patch and restricting the equation of our quartic to this line, we obtain
\begin{align*}
    ax^4 + a(-x-u)^4 + d + bx^2(-x-u)^2 + cx(-x-u).
\end{align*}
This will be a square if and only if $u$ is a root of the equation:
\begin{align*}
    4a(2a-b)u^4 - 8acu^2 + c^2 - 4d(2a+b)=0.
\end{align*}
For $u$ satisfying the equation above, the Galois group of $\C(a,b,c,d)(u)$ over $\C(a,b,c,d)$ is the same as the Galois group of $\Q(a,b,c,d)(u)$ over $\Q(a,b,c,d)$, which we can verify is $D_8$. Therefore this exhibits $D_8$ as a subgroup of $\Monpar_{D_8}$ and we are done.
\end{proof}

From this we can derive the stacky monodromy:

\begin{proposition}\label{prop:stacky-monodromy-D8-quartics}
We have that
\begin{align*}
    \Monmarked_{D_8}(\BitangentsProblem) &\cong D_8\times C_2\times C_2\\
    \Monstack_{D_8}(\BitangentsProblem) &\cong N_{W^+(E_7)}(D_8) \\
    \Mongeom_{D_8}(\BitangentsProblem) &\cong C_2^{\times 4}\rtimes C_2.
\end{align*}
\end{proposition}
\begin{proof}
Since $U^{=D_8}$ is connected, and the map $U^{=D_8}$ is dominant, we have that $\mathcal{X}_{=D_8}$ is connected as well. Thus by \Cref{thm: monodromy X_G vs X^G} (\Cref{eq: exact sequence M_G vs M^G}), we have a short exact sequence
\[
    1 \to \Monmarked_{D_8}(\BitangentsProblem) \to \Monstack_{D_8}(\BitangentsProblem) \to \Aut(D_8) \to 1.
\]
Moreover since $C_{\PGL_4}(D_8)$ is connected via \Cref{eq: exact sequence M_G vs M^G}, we can apply Corollary~\ref{cor:comparison-monodromy-UG-MG}\eqref{point:isom-parameter-marked} to conclude that parameter-level and marked stacky monodromy agree. This tells us that the size of the stacky monodromy is $|\Monpar_{D_8}(\BitangentsProblem)|\cdot|\Aut(D_8)| = 32\cdot8 = 256$. Thus the stacky monodromy attains the upper bound given by the normalizer, since $N_{W^+(E_7)}(D_8)$ also has order 256.
\end{proof}

\subsection{Type VIII Quartics}

Type VIII quartics have automorphism group $C_6$ and form a 1-dimensional family. By~\cite[Theorem 1.1]{betheabrazelton_bitangentssymmetricquartics}, there is a unique bitangent with stabilizer the full $C_6$; equivalently, $\cB^{(C_6,C_6)}\cong\cM_{2,4}^{=C_6}$.

\begin{lemma}\label{lem: type VIII quartics correspondence}
    The type of pointed cubic surfaces mapping to $\cB^{(C_6,C_6)}$ is type IV, with a $C_6$-fixed point, corresponding to the stratum $\cS^{(H_3(3)\rtimes C_2,C_6)}$.
\end{lemma}
\begin{proof}
    This follows from~\S\ref{subsec: H=C6 case}, and the fact that Type VIII is the only type of plane quartic admitting a bitangent with isotropy group $C_6$.
\end{proof}

\begin{proposition}\label{prop: Type VI quartics stacky and geometric monodromy}
    The marked, stacky and parameter-level $C_6$-symmetric monodromy of the cover of bitangents is isomorphic to $C_3\times\SL_2(\mathbb{F}_3)$, of index 2 in $N_{W^+(E_7)}(C_6)$, while the geometric monodromy is isomorphic to $A_4$.
\end{proposition}
\begin{proof}
By Theorem~\ref{thm:stacky-basic-bound}, the stacky monodromy group is contained in $N_{W^+(E_7)}(C_6)\cong C_3\rtimes\GL_2(\mathbb{F}_3)$, and the geometric one in $N_{W^+(E_7)}(C_6)/C_6\cong S_4$. Since $C_3\rtimes\GL_2(\mathbb{F}_3)$ has only $C_3\times\SL_2(\mathbb{F}_3)$ as an index 2 subgroup, it is enough to compute the order of the stacky monodromy.

Lemma~\ref{lem: type VIII quartics correspondence} and Theorem~\ref{thm: lines bitangents correspondence} yield the usual diagram
\begin{equation*}
\begin{tikzcd}
    & & \widetilde{\cB}^{=C_6}\arrow[d,"f"]\\
    \widetilde{\cL}^{(H_3(3)\rtimes C_2,C_6)}\arrow[r,"\phi_1"]\arrow[d]\arrow[urr,bend left=20,"\cong"] & \cS^{(H_3(3)\rtimes C_2,C_6)}\arrow[r,"\cong"]\arrow[d,"\pi"] & \cB^{(C_6,C_6)}\arrow[d,"\cong"]\\
    \widetilde{\cL}^{=H_3(3)\rtimes C_2}\arrow[r,"\phi"] & \cM_{3,3}^{=H_3(3)\rtimes C_2} & \cM_{2,4}^{=C_6}
\end{tikzcd}
\end{equation*}
where $\widetilde{\cL}^{(H_3(3)\rtimes C_2,C_6)}\rightarrow\cS^{(H_3(3)\rtimes C_2,C_6)}$ is the Galois closure of the pullback of the cover of (all) lines to $\cS^{(H_3(3)\rtimes C_2,C_6)}$.
By the above diagram and Theorem~\ref{thm:Type-IV-monodromy}, we have that
\[
\Mon(f)\cong\Mon(\phi_1)\leq\Mon(\phi)\cong W(L_3).
\]
The inclusion is a subgroup of index dividing $\deg\pi=27$, by Lemma~\ref{lem: Type IV C6-pointed cubics}, and divisible by 9, as the order of the stabilizers decreases by a factor of 9. Lemma~\ref{lem: Type IV C6-pointed cubics} also shows that $\pi$ is not a Galois cover. If $\Mon(\phi_1)$ was an index 27 subgroup of $\Mon(\phi)$, it would follow that $F$ factors through $\pi$. The same would hold for the covers between coarse moduli spaces. The resulting cover $\widetilde{L}^{=H_3(3)\rtimes C_2}\rightarrow(\cS^{(H_3(3)\rtimes C_2,C_6)})_{\mathrm{mod}}$ between coarse moduli spaces would then correspond to a subgroup of order 4 in $\Mon(L^{=H_3(3)\rtimes C_2}\rightarrow M_{3,3}^{=H_3(3)\rtimes C_2})\cong A_4$. By uniqueness, that subgroup would be the Klein group $C_2\times C_2$, hence normal, implying the normality of $\pi$, a contradiction by Lemma~\ref{lem: Type IV C6-pointed cubics}. Geometric monodromy follows immediately, while marked and parameter-level monodromy are equal to the stacky monodromy by Corollary~\ref{cor:comparison-monodromy-UG-MG}, since $C_{\PGL_3}(C_6)$ is connected and $C_3\times\SL_3(\F_3)$ centralizes $C_2$.
\end{proof}

\subsection{Type IX Quartics}
Type IX plane quartics have automorphism group isomorphic to $S_3$, and form a 2-dimensional family. Moreover, by~\cite[Theorem 1.1]{betheabrazelton_bitangentssymmetricquartics}, every such plane quartics admits a unique bitangent with isotropy group $S_3$, therefore $\cB^{(S_3,S_3)}\cong\cM_{2,4}^{=S_3}$ by consideration of dimension.

\begin{lemma}\label{lem: type IX quartics correspondence}
    The type of pointed cubic surfaces mapping to $\cB^{(S_3,S_3)}$ is Type VIII, with an $S_3$-invariant point, corresponding to the stratum $\cS^{(S_3,S_3)}$.
\end{lemma}
\begin{proof}
    From~\S\ref{subsec: H=S3 case}, we know that the only group $G$ for which $\cS^{(G,S_3)}$ has dimension 2 is $G=S_3$. The result follows.
\end{proof}

\begin{proposition}\label{prop: Type IX quartics stacky and geometric monodromy}
    The stacky $S_3$-symmetric monodromy of the cover of bitangents is isomorphic to $S_3^3$, while the geometric monodromy is isomorphic to $S_3^2$.
\end{proposition}
\begin{proof}
    It is enough to compute the geometric monodromy, as the exact sequence
    \[
        1\rightarrow S_3\rightarrow \Monstack_{S_3}(\BitangentsProblem)\rightarrow \Mongeom_{S_3}(\BitangentsProblem)\rightarrow 1
    \]
    splits by~\cite[Corollary 3.16]{landiStacksMonodromySymmetric2025}, being $S_3$ a complete group; see also~\cite[Remark 3.17]{landiStacksMonodromySymmetric2025}.
    By \Cref{lem: type IX quartics correspondence}, we have $\cS^{(S_3,S_3)}\cong\cB^{(S_3,S_3)}$. This, the fact that $\cS^{(S_3,S_3)}\rightarrow\cM^{=S_3}$ induces an isomorphism between stabilizers, and Theorem~\ref{thm: lines bitangents correspondence} yield a commutative diagram of covers of moduli spaces
    \begin{equation*}
    \begin{tikzcd}
        & & \widetilde{B}^{=S_3}\arrow[d,"f"]\\
        \widetilde{L}^{(S_3,S_3)}\arrow[r,"\phi_1"]\arrow[d]\arrow[urr,bend left=20,"\cong"] & (\cS^{(S_3,S_3)})_{\mathrm{mod}}\arrow[r,"\cong"]\arrow[d,"\pi"] & B^{(S_3,S_3)}\arrow[d,"\cong"]\\
        \widetilde{L}^{=S_3}\arrow[r,"\phi"] & M_{3,3}^{=S_3} & M_{2,4}^{=S_3}
    \end{tikzcd}
    \end{equation*}
    where $\widetilde{\cL}^{(S_3,S_3)}\rightarrow\cS^{(S_3,S_3)}$ is the Galois closure of the pullback of the cover of (all) lines to $\cS^{(S_3,S_3)}$.
    We know that $\Mon(\phi)\cong S_3^2$. By Lemma~\ref{lem: Type VIII S3-pointed cubics}, we have that $\pi$ has degree 3, hence we know that $\Mongeom_{S_3}(\BitangentsProblem)\cong\Mon(\phi_1)$ contains a copy of $S_3\times C_2$. We are left with showing that $\Mon(\phi_1)$ is divisible by 9. To show this, note that there is a morphism $U\rightarrow M_{3,3}^{=S_3}$ from an open $U\subset\A^2$, induced by the family of cubic surfaces with equations
    \begin{equation}\label{eq: S3 cubic surfaces}
        x^3+y^3+z^3+w^3+(ax+by)zw=0.
    \end{equation}
    This factors through $\pi$, see~\cite[Theorem 4.11]{landiStacksMonodromySymmetric2025} and Lemma~\ref{lem: Type VIII S3-pointed cubics}.
    Let us consider the family of lines in $\P^3$ with equations
    \begin{equation}\label{eq: general eq lines}
        x=\alpha z+\beta w, \qquad y=\gamma z+\delta w,
    \end{equation}
    and assume the determinant $\alpha\gamma-\beta\delta$ to be non-zero. Substituting~\eqref{eq: general eq lines} into~\eqref{eq: S3 cubic surfaces} and imposing the equation to be identically 0, we get the following identities:
    \begin{align*}
        \alpha^3+\gamma^3+1=0, & 3\alpha\beta+a=0,\\
        \beta^3+\delta^3+1, & 3\gamma\delta+b=0.
    \end{align*}
    This yields an irreducible polynomial equation
    \[
        27\alpha^6+(-a^3+b^3+27)\alpha^3-a^3=0,
    \]
    together with $\gamma=-\sqrt[3]{1+\alpha^3}$. The two equations yield a tower over extensions over $\C(a,b)$, each of which has order divisible by 3. By Proposition~\ref{prop:invariance-properties-monodromy}\eqref{point:injectivity-pullback}, it follows that $\Mon(\phi_1)$ is divisible by $9$. It follows that $\Mongeom_{S_3}(\BitangentsProblem)\cong\Mon(\phi_1)\cong S_3\times S_3$.
\end{proof}

\begin{proposition}\label{prop:monpar-s3-quartics} We have that $\Monmarked_{S_3}(\BitangentsProblem)\cong\Monpar_{S_3}(\BitangentsProblem)\cong S_3^2$.
\end{proposition}
\begin{proof}
Since $C_{\PGL_3}(S_3)$ is connected, we can apply Theorem~\ref{thm: monodromy X_G vs X^G}, Corollary~\ref{cor:comparison-monodromy-UG-MG-specialcase}, and Proposition~\ref{prop: Type IX quartics stacky and geometric monodromy} to conclude.
\end{proof}

\subsection{Type X Quartics}

As in \Cref{prop:Vtype10}, we let the Klein 4-group $K_4\le \PGL_3$ act via $y\mapsto -y$ and $z\mapsto -z$. With this action, we have that $U_{2,4}^{K_4}$ is given by
\[
a_0x^4 + a_1 y^4 + a_2 z^4 + b_0 x^2 y^2 + b_1 x^2 z^2 + b_2 y^2 z^2.
\]
We can compute the parameter-level monodromy by identifying two codimension two families in here whose monodromy we already understand.
\begin{proposition}
Let $W_1,W_2 \subseteq U_{2,4}^{K_4}$ be the subspaces given by
\begin{align*}
    W_1 &= U_{2,4}^{K_4} \cap \{a_0 = a_1\} \cap \{ b_1 = b_2\} \\
    W_2 &= U_{2,4}^{K_4} \cap \{a_1=a_2\}\cap \{b_0=b_1\}
\end{align*}
Then each of $W_1$ and $W_2$ are $\PGL_3$-equivalent to $U_{2,4}^{D_8}$, albeit for different $D_8$'s.
\end{proposition}
\begin{proof} We show this computation for $W_1$ and observe that the computation for $W_2$ is almost identical by symmetry. Applying the change of basis $x= X + i Y$, $y = iX+Y$, $z=Z$ to a quartic with $a_0 = a_1$ and $b_1 = b_2$, we get
\begin{align*}
    (2a_0 - b_0)X^4 + (2a_0 - b_0 )Y^4  + a_2 Z^4 + (-12a_0 - 2b_0) X^2Y^2  + 4ib_1 XYZ^2,
\end{align*}
which we can see is the same family as $U_{2,4}^{D_8}$ (see \Cref{prop:Vtype7}).
\end{proof}

We can in fact take this a step further, and write $W_1 = U_{2,4}^{D_8^1}$ and $W_2 = U_{2,4}^{D_8^2}$, where
\begin{align*}
    D_8^1 &= \left\langle K_4, \begin{pmatrix}
        0 & 1 & 0 \\
        1 & 0 & 0 \\
        0 & 0 & 1
    \end{pmatrix}\right \rangle \\
    D_8^2 &= \left\langle K_4, \begin{pmatrix}
        1 & 0 & 0 \\
        0 & 0 & 1 \\
        0 & 1 & 0
    \end{pmatrix} \right\rangle.
\end{align*}
With this in mind, we can now prove:
\begin{proposition}\label{prop:type10quartics-monpar}
We have that the marked and parameter-level monodromy for bitangents to Type X quartics is the centralizer of order 256.
\end{proposition}
\begin{proof}
Since we have $U_{2,4}^{D_8^1}, U_{2,4}^{D_8^2} \le U_{2,4}^{K_4}$, we know by Proposition~\ref{prop:monpar-quartic-type7} that $C_{W^+(E_7)}(D_8^1) \le \Monpar_{K_4}(\BitangentsProblem)$ and $C_{W^+(E_7)}(D_8^2) \le \Monpar_{K_4}(\BitangentsProblem)$. We can compute that an object in the intersection of the centralizers must lie in the centralizer of $S_4$, where $S_4$ is generated by the $K_4$-action together with the $S_3$-action obtained by permuting $x$, $y$, and $z$. Indeed, this follows from the fact that $W_1 \cap W_2$ is the family
\begin{align*}
    a_0(x^4 + y^4 + z^4) + b_0(x^2y^2 + x^2z^2 + y^2z^2),
\end{align*}
which is precisely $U_{2,4}^{S_4}$. Hence the monodromy only intersects at $C_{W^+(E_7)}(S_4)$, and we therefore have that
\begin{align*}
    |\Monpar_{K_4}(\BitangentsProblem)| \ge \frac{|\Monpar_{D_8^1}(\BitangentsProblem)|\cdot |\Monpar_{D_8^2}(\BitangentsProblem)|}{|\Monpar_{S_4}(\BitangentsProblem)|} = \frac{32\cdot 32}{4} = 256.
\end{align*}
Thus the upper bound given by the centralizer of $K_4$ in $W^+(E_7)$ is attained. The computation of the marked monodromy follows from this and Corollary~\ref{cor:comparison-monodromy-UG-MG}, as the centralizer $C_{\PGL_3}(K_4)$ is connected.
\end{proof}

\begin{proposition}\label{prop:monstack-k4-quartics}
We have that the stacky monodromy for bitangents to Type $\mathrm{X}$ quartics is the normalizer $N_{W^+(E_7)}(K_4)$ of order 1536. 
\end{proposition}
\begin{proof}
This is directly analogous to \Cref{prop:stacky-monodromy-D8-quartics}. Since $C_{\PGL_3}(K_4)$ is connected, we have by \Cref{cor:comparison-monodromy-UG-MG} the parameter-level monodromy agrees with the marked stacky monodromy. Thus by \Cref{thm: monodromy X_G vs X^G} \eqref{eq: exact sequence M_G vs M^G}, we obtain the short exact sequence
\[
    1 \to \Monpar_{K_4}(\BitangentsProblem) \to \Monstack_{K_4}(\BitangentsProblem) \to \Aut(K_4) \to 1.
\]
Thus $|\Monstack_{K_4}| = 1536$, and therefore must be equal to $N_{W^+(E_7)}(K_4)$.
\end{proof}

The geometric monodromy follows immediately by Theorem~\ref{thm:stacky-basic-bound} \eqref{eq:isom-geometric-monodromy}.

\subsection{Type XI Quartics}

Type XI quartics have automorphism group isomorphic to $C_3$, and form a 2-dimensional family. By~\cite[Theorem 1.1]{betheabrazelton_bitangentssymmetricquartics}, for every $C_3$-quartic there exists exactly one $C_3$-invariant bitangent; in particular, $\cB^{(C_3,C_3)}\cong\cM_{2,4}^{=C_3}$.

\begin{lemma}\label{lem: type XI quartics correspondence}
    The type of pointed cubic surfaces mapping to $\cB^{(C_3,C_3)}$ is Type IV, with a $C_3$-fixed point, corresponding to the stratum $\cS^{(H_3(3)\rtimes C_2,C_3)}$.
\end{lemma}
\begin{proof}
    This follows immediately from \S\ref{subsec: H=C3 case} and the fact that Type XI is the only type of quartic admitting a bitangent with isotropy group $C_3$.
\end{proof}

\begin{proposition}\label{prop: Type XI quartics stacky and geometric monodromy}
    The stacky $C_3$-symmetric monodromy of the cover of bitangents is isomorphic to $W(L_3)$, the unique index two subgroup in $N_{W^+(E_7)}(C_3)$. The geometric monodromy is isomorphic to $W(L_3)/C_3\cong \ASL_2(\F_3)$.
\end{proposition}
\begin{proof}
As usual, it is enough to compute the stacky monodromy. Lemma~\ref{lem: type XI quartics correspondence} together with Theorem~\ref{thm: lines bitangents correspondence} yield a commutative diagram of moduli stacks
\begin{equation*}
\begin{tikzcd}
    & & \widetilde{\cB}^{=C_3}\arrow[d,"f"]\\
    \widetilde{\cL}^{(H_3(3)\rtimes C_2,C_3)}\arrow[r,"\phi_1"]\arrow[d]\arrow[urr,bend left=20,"\cong"] & \cS^{(H_3(3)\rtimes C_2,C_3)}\arrow[r,"\cong"]\arrow[d,"\pi"] & \cB^{(C_3,C_3)}\arrow[d,"\cong"]\\
    \widetilde{\cL}^{=H_3(3)\rtimes C_2}\arrow[r,"\phi"] & \cM_{3,3}^{=H_3(3)\rtimes C_2} & \cM_{2,4}^{=C_3}
\end{tikzcd}
\end{equation*}
where $\widetilde{\cL}^{(H_3(3)\rtimes C_2,C_3)}\rightarrow\cS^{(H_3(3)\rtimes C_2,C_3)}$ is the Galois closure of the pullback of the cover of (all) lines to $\cS^{(H_3(3)\rtimes C_2,C_3)}$.
By Lemma~\ref{lem: Type IV unique cubic surface with C3-points}, the morphism $\pi$ is faithfully flat with irreducible fibers. Then, Proposition~\ref{prop:invariance-properties-monodromy}\eqref{point:pullback-of-monodromy-along-flat-map} implies that $\Monstack_{C_3}(\BitangentsProblem)\cong\Mon(f)\cong\Mon(\phi_1)\cong\Mon(\phi)$. By~Theorem~\ref{thm:Type-IV-monodromy}, $\Mon(\phi)$ has order $648$. It follows that $\Mon(f)$ is a subgroup of index 2 in $N_{W^+(E_7)}(C_3)\cong N_{W(E_6)}(H_3(3)\rtimes C_2)$, hence again $\Mon(f)\cong W(L_3)$. The parameter-level monodromy is equal to the marked and stacky monodromy by Theorem~\ref{thm: monodromy X_G vs X^G} and Corollary~\ref{cor:comparison-monodromy-UG-MG}, as $C_{\PGL_3}(C_3)$ is connected. The geometric monodromy follows by Theorem~\ref{thm:stacky-basic-bound}\eqref{eq:isom-geometric-monodromy}.
\end{proof}

\subsection{Type XII Quartics}

Plane quartics of Type XII have automorphism group $C_2$, and form a family of dimension 4. Any such plane quartic is isomorphic to one with equation
\[
    x^4+y^4+z^4+x^2(ay^2+byz+cz^2)+dy^2z^2,
\]
for some parameters $a$, $b$, $c$, $d$.

By~\cite[Theorem 1.1]{betheabrazelton_bitangentssymmetricquartics}, the cover $\cB^{(C_2,C_2)}\rightarrow\cM^{=C_2}$ has degree 4.

\begin{lemma}\label{lem: type XII quartics correspondence}
    The stack $\cB^{(C_2,C_2)}$ is generically isomorphic to $\cS^{(C_2,C_2)}$, and in particular irreducible.
\end{lemma}
\begin{proof}
    The isomorphism follows from the fact that the only $G$ for which $\cS^{(G,C_2)}$ has dimension 4 is $G=C_2$, in which case the only inclusion is the identity by \S\ref{subsec: H=C2 case}. Since $\cS^{(C_2,C_2)}\rightarrow\cM^{=C_2}$ is faithfully flat with irreducible fibers and base, $\cS^{(C_2,C_2)}$ is itself irreducible.
\end{proof}

\begin{proposition}\label{prop: Type XII quartics stacky and geometric monodromy}
    The marked, parameter-level, and stacky $C_2$-symmetric monodromy of the cover of bitangents are isomorphic to $N_{W^+(E_7)}(C_2)$ of order 4608, while the geometric monodromy is isomorphic to $N_{W^+(E_7)}(C_2)/C_2$, of order 2304.
\end{proposition}
\begin{proof}
    Note first that the marked, parameter-level, and stacky monodromy all coincide, hence it is enough to compute the latter. Let $\widetilde{\cL}^{(C_2,C_2)}\rightarrow\cS^{(C_2,C_2)}$ be the Galois closure of the pullback of the cover of (all) lines to $\cS^{(C_2,C_2)}$. By Lemma~\ref{lem: type XII quartics correspondence}, we have a commutative diagram of moduli stacks
    \begin{equation*}
    \begin{tikzcd}
        & & \widetilde{\cB}^{=C_2}\arrow[d,"f"]\arrow[dd,bend left=50,"g"]\\
        \widetilde{\cL}^{(C_2,C_2)}\arrow[r,"\phi_1"]\arrow[d]\arrow[urr,bend left=20,"\psi"] & \cS^{(C_2,C_2)}\arrow[r,"\cong"]\arrow[d,"\pi"] & \cB^{(C_2,C_2)}\arrow[d,"h"]\\
        \widetilde{\cL}^{=C_2}\arrow[r,"\phi"] & \cM_{3,3}^{=C_2} & \cM_{2,4}^{=C_2}
    \end{tikzcd}
    \end{equation*}
    where $\psi$ is a generic isomorphism, which we can treat as an actual isomorphism by Proposition~\ref{prop:invariance-properties-monodromy}\eqref{point:restriction-to-open}. Since $\pi$ is faithfully flat with irreducible fibers (Lemma~\ref{lem: C2 locus}), by Proposition~\ref{prop:invariance-properties-monodromy}\eqref{point:pullback-of-monodromy-along-flat-map},~\cite[Theorem 4.24]{landiStacksMonodromySymmetric2025} and the diagram above, we have $\Mon(f)\cong\Mon(\phi_1)\cong\Mon(\phi)\cong W(F_4)$, of order $1152$. It follows from this and Lemma~\ref{lem: type XII quartics correspondence} that
    \[
        |\Mon(g)|=|\Mon(f)|\cdot\deg h=4608=|N_{W^+(E_7)}(C_2)|.
    \]
    Since $\Mon(g)\leq N_{W^+(E_7)}(C_2)$, and the geometric monodromy is isomorphic to $\Mon(g)/C_2$, the result follows.
\end{proof}

\section{From monodromy to formulas in radicals}\label{sec:radicals}

Recall that solvability of the Galois group of a cover implies the existence of formulas in radicals for the solutions in terms of the defining equations. In \Cref{thm:main}, for instance, it can be easily verified that all the parameter-level monodromy groups are solvable, and therefore there are formulas in radicals over $\C$ for lines on a $G$-symmetric smooth cubic surface, defined in terms of the coordinates of $U_{3,3}^{=G}$.

While this is already valuable, it says nothing about the complexity of the coefficients which will appear in the formulas in radicals - a priori they could be arbitrary complex numbers. When possible, we would like to get a formula in radicals over the rationals. To that end, we introduce the following notation:

\begin{notation} For each $G\le \PGL_4(\C)$ in \Cref{thm:main}, presented as in \Cref{sec:U33}, we let $K_G\supseteq \Q$ denote the smallest number field for which $G \le \PGL_4(K_G)$. Therefore $U_{3,3}^{=G}$ is a well-defined $K_G$-variety which is geometrically irreducible by \Cref{thm:irred-UG}. Letting $\widetilde{V_{|U_{3,3}^{=G}}}$ denote the normal core cover taken over $K_G$, we define
\[
\Monpar_{G,K_G}(\LinesProblem) := \Gal(K_G(\widetilde{V_{|U_{3,3}^{=G}}}) / K_G(U_{3,3}^{=G})).
\]
\end{notation}
Arguing that these groups are solvable will guarantee the existence of formulas in radicals defined over $K_G$ for the lines on $G$-symmetric cubic surfaces in terms of the local coordinates of $U_{3,3}^{=G}$. An explicit example is the formula in radicals for lines on Type V cubic surfaces, which is defined over the rationals and worked out explicitly in \cite[\S7]{brazeltonMonodromySpaceSymmetric2025}.

We have the following:
\begin{corollary}\label{cor:solvability} For each $G$ in \Cref{thm:main}, we have that $\Monpar_{G,K_G}(\LinesProblem)$ is solvable.
\end{corollary}
\begin{proof}
Since each element of $g$ acts via deck transformations defined over $K_G$, it necessarily commutes with teh action of $\Monpar_{G,K_G}(\LinesProblem)$ on the geometric fiber. Therefore we have that $\Monpar_{G,K_G}(\LinesProblem) \le C_{W(E_6)}(G)$.
Since $C_{W(E_6)}(G)$ is solvable for each $G$ in \Cref{thm:main}, we conclude that $\Monpar_{G,K_G}(\LinesProblem)$ is solvable as well.
\end{proof}
Similarly for quartics, for any $G\le \PGL_3$ in \Cref{thm:main2}, we can let $K_G\supseteq \Q$ be the smallest number field for which $G\le \PGL_3(K_G)$, and define $\Monpar_{G,K_G}(\BitangentsProblem)$ similarly. An analogous proof shows it is a subgroup of $C_{W^+(E_7)}(G)$, which we may verify is solvable as well.

\begin{remark}
    Recall that, under the assumptions of Setup~\ref{setup:stacky-problem}, in~\cite[Corollary 3.27]{landiStacksMonodromySymmetric2025} the author proved that there exists a short exact sequence
    \[
    \begin{tikzcd}
        1\arrow[r] & G\arrow[r,"\omega_x"] & \Mon(\cF^G)\arrow[r] & \Mon(F^G)\arrow[r] & 1,
    \end{tikzcd}
    \]
    showing in particular that $\Mon(\cF^G)\leq N_{\Mon(\cF)}(G)$. In particular, it is assumed that the base field is algebraically closed. However, note that the proof of~\cite[Corollary 3.27]{landiStacksMonodromySymmetric2025} (as most of the results in~\cite{landiStacksMonodromySymmetric2025}), holds over an arbitrary field $k$. Indeed, the corollary is an immediate consequence of~\cite[Lemma 2.28]{landiStacksMonodromySymmetric2025} and~\cite[Proposition 3.15]{landiStacksMonodromySymmetric2025}, neither of which make use of that assumption. Indeed, \cite[Lemma 2.28]{landiStacksMonodromySymmetric2025} is based on \cite[Theorem 7.11]{noohiFundamentalGroupsAlgebraic2004}, which works in much larger generality.
\end{remark}

\appendix

\section{Proof of~\Cref{thm:irred-UG}}\label{sec:U33}

For a given $G\le \GL_4$, what would it mean for a cubic form in $H^0(\P^3, \mathcal{O}(3))$ to be $G$-invariant? It does not mean that it is literally fixed by $G$, instead it means that the form $F\in H^0(\P^3,\mathcal{O}(3))$ has the property that, for every $g\in G$ there is some nonzero scalar $\lambda_g \in \C^\times$ for which $g\cdot F = \lambda_g F$. This assignment $g\mapsto \lambda_g$ necessarily defines a 1-dimensional character of $G$. Thus we need to understand for which characters there exist smooth $G$-equivariant cubic surfaces in its character eigenspace. Up to conjugating the group in $\PGL_4$, we will see that this character can always be taken to be the identity.

\begin{proposition}\label{prop:Utype11} We have that $U_{3,3}^{C_2}$ and $U_{3,3}^{(C_2)}$ are closed irreducible subvarieties of $U_{3,3}$ of dimensions 12 and 18, respectively.
\end{proposition}
\begin{proof} The cyclic group of order 2 acts via the diagonal matrix with entries $[1,1,1,-1]$. The cubic forms in $x,y,z,w$ decompose into two character eigenspaces, depending on how many $w$'s there are. The nontrivial character eigenspace has an odd number of $w$'s in every monomial, hence every cubic surface given by such a cubic form is singular, as it contains the hyperplane $w=0$. Therefore the smooth cubic surfaces lie in the trivial character eigenspace, and the closure of $U_{3,3}^{C_2}$ is the projectivization of
\[
    W_{\chi_1} = \text{span}\{x^3, x^2y, x^2z, xy^2, xyz, xz^2, xw^2, y^3, y^2z, yz^2, yw^2, z^3, zw^2\}
\]
Hence $U_{3,3}^{C_2}$ is 12-dimensional. We can compute that the normalizer of $C_2$ in $\PGL_4$ is equal to its centralizer, and consists of block matrices of the following form:
\begin{align*}
    N_{\PGL_4}(C_2) &= \left\{ 
    \begin{pmatrix}
        \ast & \ast & \ast & 0 \\
        \ast & \ast & \ast & 0 \\
        \ast & \ast & \ast & 0 \\
        0 & 0 & 0 & \ast
    \end{pmatrix}
    \right\} \le \PGL_4.
\end{align*}
This is isomorphic to $\GL_3$. We conclude that
\[
\dim U_{3,3}^{(C_2)} = \dim U_{3,3}^{C_2} + \dim \PGL_4 - \dim N_{\PGL_4}(C_2) = 12 + 15 - 9 = 18.
\]
\end{proof}

\begin{proposition}\label{prop:Utype10} We have that $U_{3,3}^{C_2\times C_2}$ and $U_{3,3}^{(C_2\times C_2)}$ are closed irreducible subvarieties of $U_{3,3}$ of dimensions 7 and 17, respectively.
\end{proposition}
\begin{proof} The action of the Klein 4-group is via the diagonal matrices $[1,1,-1,1]$ and $[1,1,1,-1]$. Since this contains the action of $C_2$ in \Cref{prop:Utype11}, any smooth cubic surfaces must lie in the trivial character eigenspace for $C_2$ (and by symmetry, for the other order two generator). Therefore $U_{3,3}^{C_2 \times C_2}$ is given in the projectivization of the vector space
\[
W_{\chi_1} = \text{span}\{x^3, x^2y, xy^2, xz^2, xw^2, y^3, yz^2, yw^2\}.
\]
Thus $U^{C_2\times C_2}$ is 7-dimensional. The normalizer is given by all matrices of the form
\[
\begin{pmatrix}
    \ast & \ast & 0 & 0 \\
    \ast & \ast & 0 & 0 \\
    0 & 0 & \ast & 0 \\
    0 & 0 & 0 & \ast
\end{pmatrix}
\text{ or }
\begin{pmatrix}
    \ast & \ast & 0 & 0 \\
    \ast & \ast & 0 & 0 \\
    0 & 0 & 0 & \ast \\
    0 & 0 & \ast & 0
\end{pmatrix}.
\]
We see then that
\[
N_{\PGL_4}(C_2\times C_2) \cong \frac{\GL_2 \times \C^\times \times \C^\times}{\C^\times} \rtimes C_2 \cong (\GL_2 \times \C^\times) \rtimes C_2.
\]
This is 5-dimensional, so $U_{3,3}^{(C_2\times C_2)}$ is 17-dimensional. Only the matrices of the left shape centralize $C_2\times C_2$, so the centralizer is connected.
\end{proof}

\begin{proposition}\label{prop:Utype9} We have that $U_{3,3}^{C_4}$ and $U_{3,3}^{(C_4)}$ are irreducible closed subvarieties of $U_{3,3}$ of dimensions 6 and 16, respectively.
\end{proposition}
\begin{proof} The cyclic group of order 4 is generated by the diagonal matrix $[1,1,-1,i]$ in $\PGL_4$. As this group contains $C_2$, we can restrict our attention to the trivial character eigenspace for $C_2$. This leaves us with two character eigenspaces, depending on whether the generator for $C_4$ rescales the monomial by $+1$ or $-1$. These give
\begin{align*}
    W_{\chi_1} &= \text{span}\{x^3, x^2y, xy^2,  xz^2,y^3, yz^2, zw^2\} \\
    W_{\chi_{-1}} &= \text{span}\{x^2z, xyz,xw^2,y^2z,yw^2,z^3\}.
\end{align*}
We can observe that any cubic surface in $W_{\chi_{-1}}$ is singular, hence the nonsingular cubic surfaces lie in the trivial character eigenspace. Thus $U^{C_4}$ is 6-dimensional.

To compute the normalizer, we first note that the generator of $C_4$ and its inverse have different eigenvalues, hence are not conjugate in $\PGL_4$. Therefore the normalizer of $C_4$ in $\PGL_4$ is equal to its centralizer, and this can easily be computed to be those block matrices of the form
\[
N_{\PGL_4}(C_4) = \left\{
\begin{pmatrix}
    \ast & \ast & 0 & 0 \\
    \ast & \ast & 0 & 0 \\
    0 & 0 & \ast & 0 \\
    0 & 0 & 0 & \ast
\end{pmatrix}
\right\} \le \PGL_4.
\]
This is isomorphic to $\GL_2\times \mathbb{C}^\times$, which is 5-dimensional, hence $U_{3,3}^{(C_4)}$ is 16-dimensional.
\end{proof}

\begin{proposition}\label{prop:Utype8} We have that $U_{3,3}^{S_3}$ and $U_{3,3}^{(S_3)}$ are closed irreducible subvarieties of $U_{3,3}$ of dimensions 6 and 17, respectively.
\end{proposition}
\begin{proof} The action of $S_3$ is via the tranposition interchanging $z$ and $w$ and the diagonal matrix with entries $[1,1,\zeta_3,\zeta_3^2]$. It is easy to see that any matrix normalizing this $S_3$ must be block diagonal, with the upper $2\times 2$ block arbitrary, and the lower $2\times 2$ block agreeing (up to a scalar) with the lower $2\times 2$ block of some element in $S_3$. From this we clearly have that
\begin{align*}
    N_{\PGL_4}(S_3) \cong \GL_2 \times S_3.
\end{align*}
Since $S_3^\ab = C_2$, there are two character eigenspaces to consider. There is a one-dimensional space of cubic forms which are invariant under the sign character, and this is spanned by $z^3-w^3$, hence singular. Therefore the space of smooth invariant cubics is given by
\[
W_{\chi_1} = \text{span} \{x^3, x^2y, xy^2, y^3, xzw, yzw, z^3+w^3\}.
\]
This is 6-dimensional, hence $U_{3,3}^{(S_3)}$ is 17-dimensional.
\end{proof}

\begin{proposition}\label{prop:Utype7}
We have that $U_{3,3}^{C_8}$ and $U_{3,3}^{(C_8)}$ are closed irreducible subvarieties of $U_{3,3}$ of dimensions 3 and 15, respectively.
\end{proposition}
\begin{proof} The cyclic group $C_8$ acts via the diagonal matrix with entries $[1,-1,\zeta_8^2,\zeta_8^3]$. Squaring this and invoking \Cref{prop:Utype9}, we have that smooth cubic surfaces lie in the trivial character eigenspace for $C_4$, hence there are only two character eigenspaces to $C_8$ to consider, namely if it sends the generator to $+1$ or $-1$. These eigenspaces are exactly
\begin{align*}
    W_{\chi_1} &= \text{span} \{ x^3, xy^2, yz^2,w^2z\} \\
    W_{\chi_{-1}} &= \text{span} \{ x^2y,xz^2,y^3, \}
\end{align*}
The only smooth cubic forms are in the trivial character eigenspace, hence it is 3-dimensional. 

The normalizer is equal to the central torus in $\PGL_4$, which is $(\C^\times)^3$, therefore $U_{3,3}^{(C_8)}$ is 15-dimensional.
\end{proof}

\begin{proposition}\label{prop:Utype6}
We have that $U_{3,3}^{S_3\times C_2}$ and $U_{3,3}^{(S_3\times C_2)}$ are closed irreducible subvarieties of $U_{3,3}$ of dimensions 3 and 16, respectively.
\end{proposition}
\begin{proof}
As $S_3$ is subconjugate to $S_3\times C_2$ in $\PGL_4$, we can give a slightly different normal form for these symmetric cubic surfaces than what is found in \cite{Dolgachev}, as this will streamline our computations. In particular we may present $S_3\times C_2$ in $\PGL_4$ as the $S_3$ above, together with the involution which interchanges $x$ and $y$. Smooth cubic surfaces must be in the trivial character eigenspace for $S_3$, so we have two eigenspaces to consider depending on the character attached to this extra involution. We obtain two character eigenspaces
\begin{align*}
    W_{\chi_1} &= \text{span}\{x^3 + y^3, x^2y + xy^2, xzw + yzw, z^3 + w^3\} \\
    W_{\chi_{-1}} &= \text{span}\{x^3-y^3,x^2y - xy^2, xzw - yzw \}.
\end{align*}
Everything in $W_{\chi_{-1}}$ splits off the hyperplane $x=y$, so the smooth cubics are only in the first character eigenspace. Thus $U_{3,3}^{S_3\times C_2}$ is irreducible and 3-dimensional.

To compute the normalizer, we remark that this additional involution centralizes $S_3$ in $\PGL_4$, therefore we can compute that
\[
N_{\PGL_4}(S_3\times C_2) = C_{\PGL_4}(S_3\times C_2) \cdot (S_3\times C_2).
\]
This can be computed as the torus
\begin{align*}
    (\C^\times)^2 = \left\{ 
    \begin{pmatrix}
    a & b & 0 & 0\\
    b & a & 0 & 0 \\
    0 & 0 & 1 & 0 \\
    0 & 0 & 0 & 1
    \end{pmatrix} : a^2-b^2\ne 0
    \right\},
\end{align*}
multiplied with $S_3$, and this multiplication is external. Therefore $U_{3,3}^{(S_3\times C_2)}$ is 16-dimensional.
\end{proof}

\begin{proposition}\label{prop:Utype5}
We have that $U_{3,3}^{S_4}$ and $U_{3,3}^{(S_4)}$ are closed irreducible subvarieties of $U_{3,3}$ of dimensions 2 and 16, respectively.
\end{proposition}
\begin{proof} Here $S_4$ is most easily presented by the permutation matrices in $\PGL_4$. Its abelianization is $C_2$, so there are two character eigenspaces to consider. After conjugating $C_2^{\times 2}$ in $\PGL_4$, we see that it is a non-normal subgroup of $S_4$ -  in particular it is not contained in the alternating group $A_4$, and therefore since smooth cubic surfaces must lie in the trivial character eigenspace for $C_2^{\times 2}$ by \Cref{prop:Utype10}, they must lie in the trivial character eigenspace for $S_4$ as well. It is immediate to see that this is given by the elementary homogeneous symmetric polynomials, hence of projective dimension two in $U_{3,3}$. The normalizer was computed in \cite[2.10]{brazeltonMonodromySpaceSymmetric2025} hence we can see that $U_{3,3}^{(S_4)}$ is 16-dimensional. 
\end{proof}

\begin{proposition}\label{prop:Utype4}
We have that $U_{3,3}^{H_3(3)\rtimes C_2}$ and $U_{3,3}^{(H_3(3)\rtimes C_2)}$ are closed irreducible subvarieties of $U_{3,3}$ of dimensions 2 and 16, respectively.
\end{proposition}
\begin{proof}
We observe that $S_3 \le H_3(3)\rtimes C_2$ is not killed by the abelianization of $H_3(3)\rtimes C_2$, hence by \Cref{prop:Utype8} all smooth cubic forms must lie in the trivial character eigenspace for $H_3(3)\rtimes C_2$. We can see that the trivial character eigenspace is
\[
W_{\chi_1} = \text{span} \{ x^3 + y^3 + z^3, w^3, xyz\}.
\]
Thus $U_{3,3}^{H_3(3)\rtimes C_2}$ is irreducible and 2-dimensional.

To compute the normalizer, we observe that $H_3(3)\rtimes C_2$ as presented fixes the point $[0:0:0:1]$ and preserves the plane $w=0$. We can therefore study how $G$ acts on this hyperplane $\P^3$. The center $Z(H_3(3)) \cong C_3$ acts trivially, and since $H_3(3)/C_3 \cong C_3^{\times 2}$, we have that the image in $\PGL_3$ is given by $C_3^{\times 2}\rtimes C_2$. We now claim that
\[
N_{\PGL_3}(C_3^{\times 2} \rtimes C_2) = H_{216},
\]
where $H_{216} = C_3^{\times 2}\rtimes \SL_2(\mathbb{F}_3)$ is the order 216 Hessian group (c.f. \cite{artebaniHessePencilPlane2009}). Since $C_3^{\times 2}$ is the 3-Sylow subgroup of $C_3^{\times 2}\rtimes C_2$, it is characteristic, and therefore
\begin{align*}
    N_{\PGL_3}(C_3^{\times 2}\rtimes C_2) \le N_{\PGL_3}(C_3^{\times 2}) = C_3^{\times 2}\rtimes \SL_2(\mathbb{F}_3).
\end{align*}
Conversely, we can see that $C_3^{\times 2}\rtimes C_2 \le C_3^{\times 2}\rtimes \SL_2(\mathbb{F}_3)$ is normal, and hence we get the other inclusion $N_{\PGL_3}(C_3^{\times 2})\le N_{\PGL_3}(C_3^{\times 2}\rtimes C_2)$. Putting this all together, we get a short exact sequence
\begin{equation}\label{eqn:normalizer-type-4-cubics}
\begin{aligned}
    1 \to \C^\times \to N_{\PGL_4}(H_3(3)\rtimes C_2) \to H_{216} \to 1.
\end{aligned}
\end{equation}
Thus the normalizer is 1-dimensional, and hence $U_{3,3}^{(H_3(3)\rtimes C_2)}$ is 16-dimensional.
\end{proof}

\begin{proposition}\label{prop:Utype3} We have that $U_{3,3}^{H_3(3)\rtimes C_4}$ and $U_{3,3}^{(H_3(3)\rtimes C_4)}$ are closed irreducible subvarieties of $U_{3,3}$ of dimensions 1 and 15, respectively.
\end{proposition}
\begin{proof} We have that Type III cubic surfaces can be presented with the same subgroup has Type IV, but where we add on one more additional automorphism, this one is given by:
\[
    \frac{\zeta_{12} + \zeta_{12}^{-1}}{3}
    \begin{pmatrix} 1 & 1 & 1 & 0 \\
    1 & \zeta_3^2 & \zeta_3 & 0 \\
    1 & \zeta_3 & \zeta_3^2 & 0 \\
    0 & 0 & 0 & \frac{3}{\zeta_{12} + \zeta_{12}^{-1}}
    \end{pmatrix}.
\]
Since $U^{H_3(3)\rtimes C_4} \subseteq U^{H_3(3)\rtimes C_2}$, we can find that a cubic surface in the latter space admits this additional symmetry if and only if it is of the form
\[
a(x_0^3 + x_1^3 + x_2^3) + b x_3^3 + 3 (\sqrt{3}-1)a x_0 x_1 x_2.
\]
In particular we get that $U^{H_3(3)\rtimes C_4}$ is irreducible of projective dimension 1.

To compute the normalizer, a similar line of reasoning to the proof of \Cref{prop:Utype4}, taking into account this new generator, gives us that the image in $\PGL_3$ by restricting to the plane $w=0$ is $C_3^{\times2}\rtimes C_4$, where the new element has order four. This new element has normalizer $Q_8$ in $\PGL_3$, so we obtain that $N_{\PGL_3}(C_3^{\times 2}\rtimes C_4) = C_3^{\times 2}\rtimes Q_8$, of order 72. Thus we get a short exact sequence
\begin{equation}\label{eqn:normalizer-type-3-cubics}
\begin{aligned}
1\to \C^\times \to N_{\PGL_4}(H_3(3)\rtimes C_4) \to C_3^{\times 2}\rtimes Q_8 \to 1.
\end{aligned}
\end{equation}
Thus $U_{3,3}^{(H_3(3)\rtimes C_4)}$ is 15-dimensional, and $C_{\PGL_4}(H_3(3)\rtimes C_4 = \mathbb{G}_m$.
\end{proof}

\begin{proposition}\label{prop:Utype2}
We have that $U_{3,3}^{S_5}$ and $U_{3,3}^{(S_5)}$ are closed irreducible subvarieties of $U_{3,3}$ of dimensions 0 and 15, respectively.
\end{proposition}
\begin{proof} Since $S_4\le S_5$ is not killed by the abelianization, every smooth $S_5$-symmetric cubic form lies in the trivial character eigenspace. There is a unique such form (up to a scalar), given by the Clebsch cubic surface
\begin{align*}
    x^3 + y^3 + z^3 + w^3 - (x+y+z+w)^3.
\end{align*}
The normalizer is quite easy to compute -- it is straightforward to see that $C_{\PGL_4}(S_5)$ is trivial, hence since $S_5$ is complete, we obtain that the normalizer must be $S_5$ itself. Thus $U_{3,3}^{(S_5)}$ is 15-dimensional. The centralizer is trivial.
\end{proof}

\begin{proposition}\label{prop:Utype1} We have that $U_{3,3}^{(C_3^{\times 3}\rtimes S_4)}$ and $U_{3,3}^{(C_3^{\times 3}\rtimes S_4)}$ are closed irreducible subvarieties of $U_{3,3}$ of dimensions 0 and 15, respectively.
\end{proposition}
\begin{proof} It is clear that there is a unique cubic surface fixed by this automorphism group, namely the Fermat cubic
\begin{align*}
    x^3 + y^3 + z^3 + w^3.
\end{align*}
Any matrix in $\PGL_4(\C)$ which normalizes the automorphism group of the Fermat cubic must be a generalized permutation matrix. By considering its conjugation action on the permutation matrices in $C_3^{\times 3}\rtimes S_4$, it is clear that the entries of this matrix must be third roots of unity. But all such matrices are in the group, hence
\[
N_{\PGL_4(\C)}(C_3^{\times 3}\rtimes S_4) = C_3^{\times 3}\rtimes S_4.
\]
Thus $U_{3,3}^{(C_3^{\times 3}\rtimes S_4)}$ is irreducible of dimension 15. The centralizer of this group must therefore be equal to its center, which we can verify is trivial.
\end{proof}

\section{Proof of~\Cref{thm:irred-VG}}\label{sec:U24}

We prove this throughout this section, working upwards through the poset of subconjugacy classes of these groups in $\PGL_3$. As in the case of cubic surfaces, when the group is presented by diagonal matrices, the monomials form an eigenbasis for the action of the group on the space of ternary quartic forms, which makes computations easier.

\begin{proposition}\label{prop:Vtype12} Let $G=C_2$. Then $U_{2,4}^{C_2}$ is irreducible of dimension 8.
\end{proposition}
\begin{proof} Here the group $C_2$ is defined by
\begin{align*}
    C_2 = \left\langle \begin{pmatrix}
        1 & 0 & 0 \\
        0 & 1 & 0 \\
        0 & 0 & -1
    \end{pmatrix}\right \rangle \le \PGL_3.
\end{align*}
There are two character eigenspaces to consider, which are $\chi_1$ and $\chi_{-1}$, and are spanned by:
\begin{align}
    W_{\chi_1} &= \text{span}\{x^4,x^3y, x^2 y^2, x^2 z^2, xy^3, xyz^2, y^4, y^2 z^2, z^4\} \\
    W_{\chi_{-1}} &= \text{span} \{x^3z, x^2yz, xy^2z, xz^3, y^3z, yz^3\}.
\end{align}
Nothing in $\mathbb{P}(W_{\chi_{-1}})$ defines a smooth planar quartic, since they all split off the line $z=0$, hence since $\dim W_{\chi_1} = 9$, we have that $U_{2,4}^{G} = \mathbb{P}(W_{\chi_1})$ is 8-dimensional.

We can easily compute that the normalizer is equal to its centralizer, which is
\begin{align*}
    N_{\PGL_3}(C_2) \cong \left\{ \begin{pmatrix}
        \ast & \ast & 0\\
        \ast & \ast & 0\\
        0 & 0 & 1
    \end{pmatrix}\right\} \cong \GL_2.
\end{align*}
Therefore
\[
\dim U_{2,4}^{(C_2)} = \dim U_{2,4}^{C_2} + \dim \PGL_3 - \dim N_{\PGL_3}(C_2) = 8 + 8 - 4 = 12.
\]
\end{proof}

\begin{proposition}\label{prop:Vtype11} Let $G=C_3$. Then $U_{2,4}^{C_3}$ is irreducible of dimension 6.
\end{proposition}
\begin{proof} We have that
\[
C_3 =\left\langle \begin{pmatrix}
        1 & 0 & 0 \\
        0 & 1 & 0 \\
        0 & 0 & \zeta_3
    \end{pmatrix}\right \rangle \le \PGL_3.
\]
This gives three character eigenspaces, which are
\begin{align*}
    W_{\chi_1} &= \text{span} \{ x^4, x^3y, x^2 y^2, xy^3, xz^3, y^4, yz^3\} \\
    W_{\chi_{\zeta_3}} &= \text{span}\{x^3z, x^2yz, xy^2z, y^3z, z^4\} \\
    W_{\chi_{\zeta_3^2}} &= \text{span}\{x^2z^2, xyz^2, y^2 z^2 \}.
\end{align*}
The latter two consist of singular quartics, hence we have that $U_{2,4}^{C_3} = \mathbb{P}(W_{\chi_1})$ is 6-dimensional. We again compute the normalizer as
\begin{align*}
    N_{\PGL_3}(C_3) \cong \left\{ \begin{pmatrix}
        \ast & \ast & 0\\
        \ast & \ast & 0\\
        0 & 0 & 1
    \end{pmatrix}\right\} \cong \GL_2,
\end{align*}
hence $U_{2,4}^{(G)}$ is 10-dimensional.
\end{proof}

\begin{proposition}\label{prop:Vtype10} We have that $U_{2,4}^{K_4}$ is irreducible of dimension 5.
\end{proposition}
\begin{proof} Here we have that the Klein-4 group is given as
\begin{align*}
    K_4 =\left\langle \begin{pmatrix}
        1 & 0 & 0 \\
        0 & 1 & 0 \\
        0 & 0 & -1
    \end{pmatrix}, \begin{pmatrix}
        1 & 0 & 0 \\
        0 & -1 & 0 \\
        0 & 0 & 1
    \end{pmatrix}\right \rangle \le \PGL_3.
\end{align*}
Since $C_2 \le K_4 \le \PGL_3$, any smooth $K_4$-symmetric planar quartics must have the property that the first generator acts trivially. By conjugating $C_2$ inside $\PGL_3$, an identical argument to \Cref{prop:Vtype12} shows that smooth planar quartics must lie in the trivial eigenspace for the other generator of $K_4$. Hence altogether, smooth planar $K_4$-symmetric quartic forms must lie in the trivial eigenspace for $K_4$. This is given as
\begin{align*}
    W_{\chi_{\text{triv}}} &= \text{span}\{x^4, x^2 y^2, x^2 z^2, y^4, y^2 z^2, z^4\}.
\end{align*}
Hence $U_{2,4}^{K_4}$ is 5-dimensional.

We remark that all three non-identity elements in $K_4$ are conjugate inside $\PGL_3$, hence $N_{\PGL_3}(K_4)$ surjects onto $S_3$ via its permutation action on these elements. It is straightforward to see that the kernel of this action (the centralizer of all three elements) is just the diagonal torus in $\PGL_3$, hence we get a short exact sequence
\[
0 \to (\mathbb{C}^{\times})^2 \to N_{\PGL_3}(K_4) \to S_3 \to 0.
\]
This is split by the inclusion of the permutation matrices. Thus we get that $N_{\PGL_3}(K_4) \cong (\mathbb{C}^{\times})^2\rtimes S_3$, and thus $U_{2,4}^{(K_4)}$ is 11-dimensional.
\end{proof}

\begin{proposition}\label{prop:Vtype9} $U_{2,4}^{S_3}$ is irreducible of dimension 3.
\end{proposition}
\begin{proof}
We have that
\begin{align*}
    S_3 =\left\langle \begin{pmatrix}
        \zeta_3 & 0 & 0 \\
        0 & \zeta_3^2 & 0 \\
        0 & 0 & 1
    \end{pmatrix}, \begin{pmatrix}
        0 & 1 & 0 \\
        1 & 0 & 0 \\
        0 & 0 & 1
    \end{pmatrix}\right \rangle \le \PGL_3.
\end{align*}
Since $S_3^\ab = C_2$, the order 3 elements must act trivially. Considering the monomial basis for $W$, this restricts our attention to the subspace
\[
\text{span} \{x^3z, x^2y^2, xyz^2, y^3 z, z^4\}.
\]
We now ask how the element of order two acts, and we obtain two eigenspaces, given by
\begin{align*}
    W_{\chi_1} &= \text{span}\{ (x^3 + y^3)z, x^2 y^2, xyz^2, z^4\} \\
    W_{\chi_{-1}} &= \text{span} \{ (x^3 - y^3)z\}.
\end{align*}
Hence the moduli of $S_3$-symmetric smooth planar quartics lies in the trivial character eigenspace for $S_3$, which is 3-dimensional.

Since $S_3$ is complete, we have that $N_{\PGL_3}(S_3) \cong S_3 \times C_{\PGL_3}(S_3)$. The centralizer can easily be computed as
\[
C_{\PGL_3}(S_3) = \left\{
\begin{pmatrix}
    1 & 0 & 0\\
    0 & 1 & 0\\
    0 & 0 & \ast
\end{pmatrix}\right\},
\]
hence $N_{\PGL_3}(S_3) \cong S_3 \times \mathbb{C}^\times$. We therefore obtain that $U_{2,4}^{(S_3)}$ is 10-dimensional.
\end{proof}

\begin{proposition}\label{prop:Vtype8} $U_{2,4}^{C_6}$ is irreducible of dimension 3.
\end{proposition}
\begin{proof} We have that
\begin{align*}
    C_6 =\left\langle \begin{pmatrix}
        -1 & 0 & 0 \\
        0 & 1 & 0 \\
        0 & 0 & \zeta_3
    \end{pmatrix}\right \rangle \le \PGL_3.
\end{align*}
We get six character eigenspaces, depending upon how the generator acts. These are spanned by
\begin{align*}
    W_{\chi_1} &= \text{span} \{ x^4, x^2 y^2, y^4, yz^3 \} \\
    W_{\chi_{\zeta_3}} &= \text{span} \{ x^2yz, y^3z, z^4\} \\
    W_{\chi_{\zeta_3^2}} &= \text{span} \{ x^2 z^2, y^2 z^2\} \\
    W_{\chi_{-1}} &= \text{span} \{ x^3y, xy^3, xz^3 \} \\
    W_{\chi_{-\zeta_3}} &= \text{span} \{ x^3 z, xy^2 z \} \\
    W_{\chi_{-\zeta_3^2}} &= \text{span} \{ xyz^2 \}.
\end{align*}
The only locus containing any smooth quartic forms is the trivial character eigenspace, hence $U_{2,4}^{C_6}$ is 3-dimensional.

We can compute directly that no element in $\PGL_3$ can conjugate the generator of $C_6$ to its inverse, hence since these are the only two order 6 elements in the group, they must be fixed by the normalizer, and therefore $N_{\PGL_3}(C_6) = C_{\PGL_3}(C_6)$. It is clear that the centralizer is simply the diagonal matrices, which are isomorphic to $(\mathbb{C}^\times)^2$. Altogether we get that $U_{2,4}^{(C_6)}$ is 9-dimensional.
\end{proof}

\begin{proposition}\label{prop:Vtype7} $U_{2,4}^{D_8}$ is irreducible and 3-dimensional.
\end{proposition}
\begin{proof} Here we have
\begin{align*}
    D_8 = \left\langle \begin{pmatrix}
        i & 0 & 0 \\
        0 & -i & 0 \\
        0 & 0 & 1
    \end{pmatrix}, \begin{pmatrix}
        0 & 1 & 0 \\
        1 & 0 & 0 \\
        0 & 0 & 1
    \end{pmatrix}\right \rangle \le \PGL_3.
\end{align*}

Since the copy of $K_4$ in \Cref{prop:Vtype10} is subconjugate to $D_8$ in $\PGL_3$, any smooth $D_8$-symmetric forms must have the property that they lie in the trivial character eigenspace for the corresponding subgroup $K_4 \le D_8$. The relevant copy of $K_4$ here is
\[
\left\langle \begin{pmatrix}
        1 & 0 & 0 \\
        0 & 1 & 0 \\
        0 & 0 & -1
    \end{pmatrix}, \begin{pmatrix}
        0 & 1 & 0 \\
        1 & 0 & 0 \\
        0 & 0 & 1
    \end{pmatrix}\right\rangle.
\]
The trivial character eigenspace for this copy of $K_4$ is given by symmetric polynomials in $x$ and $y$ times even powers of $z$. This is 6-dimensional, spanned by:
\begin{align*}
    \text{span} \{x^4 + y^4, x^3y + xy^3, x^2y^2, (x^2 + y^2)z^2, xyz^2, z^4\}.
\end{align*}
This vector space decomposes to our two relevant character eigenspaces for $D_8$, depending upon whether the element $\begin{pmatrix}
        i & 0 & 0 \\
        0 & -i & 0 \\
        0 & 0 & 1
    \end{pmatrix}$
acts by $+1$ or $-1$. We get
\begin{align*}
    W_{\chi_1} &= \text{span}\{ x^4 + y^4, x^2 y^2, xyz^2, z^4\} \\
    W_{\chi_{-1}} &= \text{span}\{ x^3y + xy^3,  (x^2 + y^2)z^2\}.
\end{align*}
Anything in $W_{\chi_{-1}}$ has a singularity at $[0:0:1]$, hence all the smooth quartics lie in the trivial eigenspace, and hence $U_{2,4}^{D_8}$ is 3-dimensional.

To compute the normalizer, we fix some notation for the generators. Let
\begin{align*}
    r &= \begin{pmatrix}
        i & 0 & 0 \\
        0 & -i & 0 \\
        0 & 0 & 1
    \end{pmatrix} \\
    s &= \begin{pmatrix}
        0 & 1 & 0 \\
        1 & 0 & 0 \\
        0 & 0 & 1
    \end{pmatrix}.
\end{align*}
We note that $Z(D_8) = \langle r^2\rangle \cong C_2$, hence any element $g\in N_{\PGL_4}(D_8)$ must centralize $r^2$. This tells us that
\begin{equation}\label{eqn:block-diagonal}
\begin{aligned}
    N_{\PGL_4}(D_8) \le \left\{ \begin{pmatrix}
    \ast & \ast & 0 \\
    \ast & \ast & 0 \\
    0 & 0 & 1
\end{pmatrix}\right\}.
\end{aligned}
\end{equation}
That upper left $2\times 2$ block, considered in $\GL_2$, must either fix or interchange $r$ and $r^3$ under conjugation (as they are the unique things of order 4). Hence this upper block must be diagonal or anti-diagonal. If it is diagonal, then by conjugating with the transposition, we get
\begin{align*}
    \begin{pmatrix}
        a & 0 & 0\\
        0 & b & 0 \\
        0 & 0 & 1
    \end{pmatrix}s
    \begin{pmatrix}
        a^{-1}& 0 \\
        0 & b^{-1} \\
        0 & 0 & 1
    \end{pmatrix} &= \begin{pmatrix}
        0 & ab^{-1} & 0 \\
        a^{-1}b & 0 & 0 \\
        0 & 0 & 1.
    \end{pmatrix}
\end{align*}
In order for this to be in $D_8$, we must have that $a/b \in \{i,-1,-i,1\}$. Hence every such matrix is of the form
\[\begin{pmatrix}
    i^kb & 0 & 0 \\
    0 & b & 0 \\ 
    0 & 0 & 1
\end{pmatrix} = \begin{pmatrix}
    i^k & 0 & 0 \\
    0 & 1 & 0 \\ 
    0 & 0 & b
\end{pmatrix}
\]
for some power of $k$. If the block were anti-diagonal, a similar analysis shows that $a/b = i^k$ for some $k$, hence it differs from a diagonal matrix as before by a transposition in the group. Altogether, we get that the normalizer is given by
\begin{align*}
    N_{\PGL_3}(D_8) &= D_8 \cdot \left\langle \begin{pmatrix}
        1 & 0 & 0 \\
        0 & 1 & 0 \\
        0 & 0 & \ast
    \end{pmatrix}, 
    \begin{pmatrix}
        i & 0 & 0\\
        0 & 1 & 0\\
        0 & 0 & \ast
    \end{pmatrix}\right\rangle.
\end{align*}
We therefore obtain a short exact sequence
\begin{equation}\label{eqn:normalizer-quartictype7}
0 \to D_8 \to N_{\PGL_3}(D_8) \to \C^\times \to 0.
\end{equation}
Putting everything together, we get that $U_{2,4}^{(D_8)}$ has dimension 10.
\end{proof}

\begin{proposition}\label{prop:Vtype6} $U_{2,4}^{C_9}$ is irreducible and 2-dimensional.
\end{proposition}
\begin{proof} We have that
\[
C_9 = \left\langle \begin{pmatrix}
    \zeta_9^3 & 0 & 0 \\
    0 & 1 & 0 \\
    0 & 0 & \zeta_9
\end{pmatrix}\right\rangle \le \PGL_3.
\]
Since this generator cubed is the generator of $C_3$ in \Cref{prop:Vtype11}, we can only obtain smooth symmetric quartics in the trivial eigenspace for this character. Hence if powers of $z$ exist, they must be $z^3$'s. Among the remaining monomials, they break into character eigenspaces:
\begin{align*}
    W_{\chi_1} &= \text{span}\{x^3y, y^4\} \\
    W_{\chi_{\zeta_9^3}} &= \text{span}\{x^4, xy^3, yz^3\} \\
    W_{\chi_{\zeta_9^6}} &= \text{span}\{x^2 y^2, xz^3\} \\
\end{align*}
Interestingly, the only smooth quartics are not in the identity character eigenspace, but in the character eigenspace where the action of the generator scales the form by $\zeta_9^3$. We therefore obtain that $U_{2,4}^{C_9}$ is 2-dimensional.

The normalizer necessarily coincides with the centralizer, which we can see is the diagonal torus. Thus
\[
N_{\PGL_3}(C_9) = \left\{ \begin{pmatrix}
    \ast & 0 & 0 \\
    0 & \ast & 0 \\
    0 & 0 & \ast
\end{pmatrix}\right\} \cong (\mathbb{C}^\times)^2.
\]
Therefore $U_{2,4}^{(C_9)}$ is $8$-dimensional.
\end{proof}

\begin{proposition}\label{prop:Vtype5} $U_{2,4}^P$ is irreducible of dimension 2.
\end{proposition}
\begin{proof} The Type V Pauli group is given by
\begin{align*}
    P = \left\langle \begin{pmatrix}
        -1 & 0 & 0 \\
        0 & 1 & 0 \\
        0 & 0 & 1
    \end{pmatrix},
    \begin{pmatrix}
        i & 0 & 0 \\
        0 & -i & 0 \\
        0 & 0 & 1
    \end{pmatrix},
    \begin{pmatrix}
        0 & 1 & 0 \\
        1 & 0 & 0 \\
        0 & 0 & 1
    \end{pmatrix}
    \right\rangle \le \PGL_3.
\end{align*}
Observe that the dihedral group in \Cref{prop:Vtype7} is a subgroup of this group $P$, hence any smooth symmetric quartics lie in the trivial eigenspace for this group, which was a 4-dimensional vector space spanned by $x^4 + y^4$, $x^2 y^2$, $xyz^2$, and $z^4$. There are two relevant character eigenspaces for the Pauli group, corresponding to whether the order two element $x\mapsto -x$ scales the generator by $+1$ or $-1$. We obtain
\begin{align*}
    W_{\chi_1} &= \text{span}\{ x^4 + y^4, x^2 y^2, z^4 \} \\
    W_{\chi_{-1}} &= \text{span}\{ xyz^2 \}.
\end{align*}
Hence we again land in the trivial character eigenspace for the group, which in this case has dimension $\dim U_{2,4}^{P} = 2$.

To compute the normalizer, we observe the center of the Pauli group is cyclic of order 4, generated by $\begin{pmatrix}
    1 & 0 & 0\\
    0 & 1 & 0 \\
    0 & 0 & i
\end{pmatrix}$. In particular the normalizer has the same block diagonal structure as in \Cref{eqn:block-diagonal}. There is therefore a map
\begin{align*}
    N_{\PGL_3}(P) &\to \PGL_2 \\
    \begin{pmatrix}
        A & 0 \\ 0 & 1
    \end{pmatrix} &\mapsto A,
\end{align*}
and its image consists of every $2\times 2$ matrix which normalizes the $2\times 2$ upper left blocks of the elements of the Pauli group when considered in $\PGL_2$. It is clear though that the first two generators of $P$ become identified under this map to $\PGL_2$, hence the image is the normalizer of the image of $P$, which is a Klein 4-group. We therefore obtain a short exact sequence
\[
0 \to \mathbb{C}^\times \to N_{\PGL_3}(P) \to N_{\PGL_2}(K_4) \to 0.
\]
We can compute that the normalizer of $K_4$ is given by the rigid symmetries of the fixed points of the involutions in $K_4$, which is precisely the octahedron on $\mathbb{P}^1$ with vertices $\{0,\pm 1, \pm i, \infty\}$. Thus $N_{\PGL_2}(K_4) \cong S_4$ is the octahedral group, and we obtain a non-split short exact sequence
\begin{equation}\label{eqn:normalizer-P}
\begin{aligned}
    0 \to \mathbb{C}^\times \to N_{\PGL_3}(P) \to S_4 \to 0.
\end{aligned}
\end{equation}
With this we obtain that $U_{2,4}^{(P)}$ has dimension 9.
\end{proof}

\begin{proposition}\label{prop:Vtype4} $U_{2,4}^{S_4}$ is irreducible and 1-dimensional.
\end{proposition}
\begin{proof} Here our $S_4$ is given by
\begin{align*}
    S_4 = \left\langle 
    \begin{pmatrix}
        0 & 0 & 1 \\
        1 & 0 & 0 \\
        0 & 1 & 0
    \end{pmatrix},
    \begin{pmatrix}
        0 & -1 & 0 \\
        1 & 0 & 0 \\
        0 & 0 & 1
    \end{pmatrix}
    \right\rangle\le \PGL_3.
\end{align*}
Since $S_4^\ab = C_2$ is detected on transpositions, which lie in conjugates of the $S_3$ from \Cref{prop:Vtype9}, we conclude that the smooth $S_4$-symmetric quartics must all lie in the trivial character eigenspace for $S_4$. Such forms are symmetric in cycling $x,y,z$, and are invariant with swapping $x$ with $-y$. It is clear that the vector space of such forms is
\[
\text{span}\{ x^4 + y^4 + z^4, x^2 y^2 + x^2 z^2 + y^2 z^2\}.
\]
Hence $U_{2,4}^{S_4}$ is 1-dimensional.

Since the projective representation of $S_4$ given here is irreducible, Schur's lemma says that the endomorphisms of this representation (after lifting to $\GL_3$) are all scalar multiples of the identity matrix. That is, every linear matrix in $\GL_3$ commuting with this copy of $S_4$ is a linear multiple of the identity, and thus the centralizer of $S_4$ in $\PGL_3$ is trivial. Since $S_4$ is complete, we therefore obtain that
\[
N_{\PGL_3}(S_4) \cong S_4.
\]
Thus $U_{2,4}^{(S_4)}$ is 9-dimensional.
\end{proof}

\begin{proposition}\label{prop:Vtype3} We have that $U_{2,4}^{C_4\circledcirc A_4}$ is irreducible and 1-dimensional.
\end{proposition}
\begin{proof}
The Type III group is given by
\begin{align*}
    C_4 \circledcirc A_4 &= \left\langle
    \begin{pmatrix}
        \frac{1+i}{2} & \frac{-1+i}{2} & 0 \\
        \frac{1+i}{2} & \frac{1-i}{2} & 0 \\
        0 & 0 & \zeta_3
    \end{pmatrix},
    \begin{pmatrix}
        \frac{1+i}{2} & \frac{-1-i}{2} & 0 \\
        \frac{-1-i}{2} & \frac{-1+i}{2} & 0 \\
        0 & 0 & \zeta_3^2
    \end{pmatrix}
    \right\rangle \le \PGL_3
\end{align*}
In particular the Pauli group in \Cref{prop:Vtype5} is a subgroup of this group, hence the smooth $C_4 \circledcirc A_4$-symmetric quartics lie in the trivial character eigenspace for the Pauli group, which was generated by $x^4 + y^4$, $x^2 y^2$, and $z^4$. Let $M$ be the first generator of the group above:
\[
    M = \begin{pmatrix}
        \frac{1+i}{2} & \frac{-1+i}{2} & 0 \\
        \frac{1+i}{2} & \frac{1-i}{2} & 0 \\
        0 & 0 & \zeta_3
    \end{pmatrix},
\]
so that $C_4\circledcirc A_4$ is generated by $P$ and $M$. Then we can divide into three eigenspaces depending upon how $M$ acts on the relevant forms. Under the change of basis above, we can compute that
\begin{align*}
    M \cdot (x^4 + y^4) &= -\frac{1}{2} \left( x^4 + y^4\right) - \frac{1}{2} (x^4 + y^4) + 3x^2 + y^2 \\
    M \cdot (x^2 + y^2) &= -\frac{1}{4} (x^4 + y^4) - \frac{1}{2} x^2 y^2 \\
    M\cdot z^4 &= \zeta_3 z^4.
\end{align*}
Hence in the 3D basis for the trivial eigenspace for $P$, we have that $M$ is represented by the matrix
\[
\begin{pmatrix}
    -\frac{1}{2} & 3 & 0 \\
    -\frac{1}{4} & -\frac{1}{2} & 0 \\
    0 & 0 & \zeta_3
\end{pmatrix}.
\]
The eigenvalues are therefore $\zeta_3$ and $\zeta_3^2$, and the eigenspaces are given by
\begin{align*}
    M_{\zeta_3} &= \text{span} \{ x^4 + y^4 - 2i\sqrt{3} x^2 y^2, z^4 \} \\
    M_{\zeta_3^2} &= \text{span}\{ x^4 + y^4 + 2i\sqrt{3} x^2 y^2\}.
\end{align*}
The latter is singular, so $U_{2,4}^{C_4\circledcirc A_4}$ is 1-dimensional.

To compute the normalizer, we remark that the Pauli group is the unique Sylow 2-subgroup of $C_4\circledcirc A_4$, therefore
\[
N_{\PGL_3}(C_4 \circledcirc A_4) \subseteq N_{\PGL_3}(P).
\]
However we can observe that every normalizer of the Pauli group will also normalize $C_4\circledcirc A_4$. Thus $U_{2,4}^{(C_4\circledcirc A_4)}$ is 8-dimensional.
\end{proof}

\begin{proposition}\label{prop:Vtype2} We have that $U_{2,4}^{C_4^{\times 2}\rtimes S_3}$ is a single point, and $U_{2,4}^{(C_4^{\times 2}\rtimes S_3)}$ is a closed 8-dimensional subvariety of $U_{2,4}$.
\end{proposition}
\begin{proof} The Type II quartic's automorphism group is given by
\begin{align*}
    C_4^{\times 2}\rtimes S_3 &= \left\langle
    \begin{pmatrix}
        0 & 0 & 1 \\
        1 & 0 & 0 \\
        0 & 1 & 0
    \end{pmatrix},
    \begin{pmatrix}
        -i & 0 & 0 \\
        0 & 0 & 1 \\
        0 & i & 0
    \end{pmatrix}\right\rangle \le \PGL_3.
\end{align*}
Observe that the $S_4$ from \Cref{prop:Vtype4} is a subgroup of this group as listed above, hence we can consider the character eigenspace
\[
\text{span}\{ x^4 + y^4 + z^4, x^2 y^2 + x^2 z^2 + y^2 z^2\},
\]
and consider how the additional generator of $C_4^{\times 2}\rtimes S_4$ acts on these elements. Let
\[
N = \begin{pmatrix}
        -i & 0 & 0 \\
        0 & 0 & 1 \\
        0 & i & 0
    \end{pmatrix}.
\]
Then
\begin{align*}
    N \cdot (x^4 + y^4 + z^4) &= x^4 + y^4 + z^4 \\
    N\cdot (x^2 y^2 + x^2 z^2 + y^2 z^2) &= -x^2 y^2 + x^2 z^2 - y^2z^2. 
\end{align*}
This latter action is not even in an eigenspace for a 1-dimensional character. Thus there is a unique Type II quartic, given by $x^4 + y^4 + z^4$.

The normalizer of this group is the group itself, since if any other matrix normalized it, then the image of the Type II quartic under that matrix would again be the Type II quartic, however this cannot happen as the quartic is unique. We conclude then that $U_{2,4}^{(C_4^{\times 2}\rtimes S_3)}$ is 8-dimensional.
\end{proof}

\begin{proposition}\label{prop:Vtype1} We have that $U_{2,4}^{\PSL_2(7)}$ is a single point, and $U_{2,4}^{(\PSL_2(7))}$ is a closed 8-dimensional subvariety of $U_{2,4}$.
\end{proposition}
\begin{proof} We have that $\PSL_2(7)$ is as presented in \cite[2.4]{betheabrazelton_bitangentssymmetricquartics}. It is simple and perfect so abelianizes to the identity, hence any symmetric quartic must lie in the trivial character eigenspace. It is straightforward to see that the only quartic form fixed by this group is the Klein quartic. By an analogous argument to \Cref{prop:Vtype2}, the normalizer is then the group itself, and $U_{2,4}^{(\PSL_2(7))}$ is 8-dimensional.
\end{proof}

\bibliography{galois}
\bibliographystyle{amsalpha}
\end{document}

%% file: commands.tex
\usepackage{%
  mathtools,  % math extensions and fixes; loads amsmath
  amssymb,    % extra math symbols
  amsthm,     % enhanced theorem environments
  amsfonts,   % additional math fonts
  thmtools,   % customization tools for theorems
  graphicx,
  float,
  color,
  xcolor,
  tikz,
  tikz-cd,
  mathrsfs,   % for script letters
  etoolbox    % for csvlist
}

\definecolor{darkred}{rgb}{0.75,0,0}
\def\customcitecolor{darkred}
\def\customlinkcolor{darkred}

\usepackage[%
    colorlinks,
    citecolor=\customcitecolor,%
    linkcolor=\customlinkcolor,%
    urlcolor=\customlinkcolor%
]{hyperref}

\usepackage[capitalise,nameinlink,noabbrev]{cleveref}

\usepackage[margin=1in]{geometry}

\usepackage[final,nopatch=footnote]{microtype}

\theoremstyle{definition}

\newtheorem{introtheorem}{Theorem}

\newtheorem{theorem}{Theorem}[section]
\numberwithin{theorem}{section} % important bit
\numberwithin{equation}{section} % number equations like (1.1), (1.2), etc.

\usepackage{mfirstuc}
\newcommand{\capitalizename}[1]{\makefirstuc{#1}}

\newcommand{\defthm}[1]{%
  \newtheorem{#1}[theorem]{\capitalizename{#1}}%
}

\newcommand{\defthms}[1]{%
  \forcsvlist{\defthm}{#1}%
}

\defthms{%
  answer,assumption,claim,computation,conjecture,construction,corollary,
  counterexample,definition,digression,discussion,example,
  examples,exercise,fact,goal,idea,intuition,lemma,
  motivation,notation,note,proposition,question,remark,setup,
  slogan,strategy,terminology,upshot,warning%
}

\Crefname{assumption}{Assumption}{Assumptions}

\Crefname{construction}{Construction}{Constructions}

\Crefname{notation}{Notation}{Notations}

\Crefname{question}{Question}{Questions}
\Crefname{setup}{Setup}{Setups}

\newcommand{\deftextcommand}[1]{%
  \expandafter\providecommand\csname #1\endcsname{\mathrm{#1}}%
}
\newcommand{\deftextcommands}[1]{%
    \forcsvlist{\deftextcommand}{#1}%
}

\deftextcommands{ab,alg,an,ann,Aut,BG,BGL,Bl,BO,BP,BSL,BSO,BSp,BSU,BU,can,cd,cdh,cl,coBar,codim,codom,coeq,coev,cof,cofib,coker,colim,coim,cone,conj,const,core,coTor,cyc,diag,Desc,dg,Disc,disc,dual,eff,EKL,End,eq,ess,et,Et,EU,ev,Ex,ex,Exc,Ext,fib,Fix,Fl,fppf,fpqc,Frac,Frob,Fun,Gal,gen,GL,gp,Gr,gr,GW,Her,Ho,hocofib,hocolim,hofib,holim,Hom,id,Idem,im,incl,Ind,ind,inj,Inn,Inv,inv,iso,Jac,KGL,kgl,KH,KO,ko,KQ,kq,KR,KSp,KU,ku,Lan,Map,map,MGL,MO,Mon,Mor,mor,MSL,MSO,MSp,MSU,MU,mult,MUP,Nm,ob,obj,op,Orb,ord,Out,perf,Perm,PGL,pr,pre,Proj,proj,prom,PSL,quot,Ran,rank,Res,RO,sep,sgn,SH,sig,Sing,SL,SO,soc,Sp,Span,Spec,Spin,spn,Sq,st,Stab,SU,supp,Supp,Syl,syl,Sym,syn,SYT,TC,td,Th,THH,Tor,Tot,TP,TR,Tr,tr,triv,univ,var,veff,vol,Wel,Wr}

\deftextcommands{Ab,Aff,Alg,Ani,Bimod,CAlg,Cat,CDGA,CG,CGWH,Ch,CMon,coAlg,Coh,CommRing,ConjSub,coMod,Cor,Corr,CoSh,Cov,CRing,CW,Field,Fin,FinSet,Gpd,Grp,Grpd,Grph,Kan,Kar,LMod,Mfld,Mod,NAlg,Ouv,Perf,Poset,Pr,Pre,PSh,PShv,qCat,QCoh,Rep,Ring,RMod,sAb,Set,SH,Sh,Shv,Sm,Sp,Spc,Spectra,sPre,sSet,sShv,Stack,Sub,Top,Tors,Var,Vect}

\newcommand{\defblackboardletter}[1]{%
  \expandafter\providecommand\csname #1\endcsname{\mathbb{#1}}
}
\newcommand{\defblackboardletters}[1]{%
  \forcsvlist{\defblackboardletter}{#1}%
}

\defblackboardletters{A,C,F,P,Q,R,Z}

\providecommand{\xto}[1]{\xrightarrow{#1}}

\RequirePackage{bbm}

\RequirePackage{pict2e,picture}
\makeatletter
\DeclareRobustCommand{\DDelta}{{\mathpalette\bb@Delta\relax}}
\newcommand{\bb@Delta}[2]{%
  \begingroup
  \sbox\z@{$\m@th#1\Delta$}%
  \dimendef\Dht=6 \dimendef\Dwd=8
  \setlength{\Dwd}{\wd\z@}%
  \setlength{\Dht}{\ht\z@}%
  \begin{picture}(\Dwd,\Dht)
  \put(0,0){$\m@th#1\Delta$}
  \put(.42\Dwd,.7\Dht){\line(10,-26){.25\Dht}}
  \end{picture}%
  \endgroup
}

\let\emptyset\varnothing

\newcommand{\FEt}{\text{F\'{E}t}}

\newcommand{\gpdiagram}[2]{{\color{gray}\mathrm{#1}}\,|\,#2}

\usepackage{fancyvrb,newverbs,xcolor,comment}
\definecolor{cverbbg}{gray}{0.93}

%% file: TABLE-cubics.tex
\begin{center}
\footnotesize
\begin{tabular}{c c | c c c c}
Type &  $G$ & $\Monpar_G(\LinesProblem)$ & $\Monmarked_G(\LinesProblem)$ & $\Monstack_G(\LinesProblem)$ & $\Mongeom_G(\LinesProblem)$  \\
\hline
1 & $C_3^{\times 3} \rtimes S_4$ & $0$ & $0$ & $C_3^{\times 3} \rtimes S_4$ & $0$ \\
2 & $S_5$ & $0$ & $0$ & $S_5$ & $0$\\
3 & $H_3(3)\rtimes C_4$ & $C_3$ & $C_3$ & $H_3(3)\rtimes C_4$ & $0$\\
4 & $H_3(3)\rtimes C_2$ & $C_3$ & $C_3$ & $W(L_3)$ & $A_4$  \\
5 & $S_4$ &  $K_4$ & $K_4$ & $S_4\times K_4$ & $K_4$ \\
6 & $D_6$ & $S_3\times C_2$ & $S_3\times C_2$ & $S_3^2\times C_2$ &  $S_3$ \\
7 & $C_8$ & $C_8$ & $C_8$ & $C_8$ & $0$\\
8 & $S_3$ & $S_3^2$ & $S_3^2$ & $S_3^3$ & $S_3^2$ \\
9 & $C_4$ & $U_2(\mathbb{F}_3)$ & $U_2(\mathbb{F}_3)$ & $U_2(\mathbb{F}_3)$ & $S_4$ \\
10 & $C_2^{\times 2}$ & $C_2^{\times 2}\times S_4$ & $C_2^{\times 2}\times S_4$ & $(C_2^{\times 2}\rtimes C_2)\times S_4$ & $C_2\times S_4$ \\
11 & $C_2$ & $\GO_4^+(3)$ & $\GO_4^+(3)$ & $\GO_4^+(3)=W(F_4)$ & $\PGO_4^+(3)$
\end{tabular}
\end{center}

%% file: TABLE-quartics.tex
\begin{center}
\footnotesize
    \begin{tabular}{c c | c c c c}
    Type &  $G$ & $\Monpar_G(\mathfrak{B})$ & $\Monmarked_G(\mathfrak{B})$ & $\Monstack_G(\mathfrak{B})$ & $\Mongeom_G(\mathfrak{B})$  \\
    \hline
    1  & $\PSL_2(7)$ & 0 & 0 & $\PSL_2(7)$ & 0\\
    2 & $C_4^{\times 2} \rtimes S_3$ & 0 & 0 & $C_4^{\times 2} \rtimes S_3$ & 0 \\
    3 & $C_4\circledcirc A_4$ & $C_4$ & $C_4$ & $C_4\circledcirc A_4$ & 0 \\
    4 & $S_4$ & $C_2\times C_2$ & $C_2\times C_2$ & $S_4\times C_2\times C_2$ & $C_2\times C_2$ \\
    5 & $P$ & $C_4\times K_4$ & $C_4\times K_4$ & $W(B_4)$ & $S_4$ \\
    6 & $C_9$ & $C_9$ & $C_9$ & $C_9$ & $0$ \\
    7 & $D_8$ & $D_8\times K_4$ & $D_8\times K_4$ & [256,25876] & $C_2^{\times 4}\rtimes C_2$\\
    8 & $C_6$ & $C_3\times\SL_2(\F_3)$ & $C_3\times\SL_2(\F_3)$ & $C_3\times\SL_2(\F_3)$ & $A_4$ \\
    9 & $S_3$ & $S_3^2$ & $S_3^2$ & $S_3^3$ & $S_3^2$ \\
    10 & $K_4$ & [256,51978] & [256,51978] & $N_{W^+(E_7)}(K_4)$ & [384,20089]\\
    11 & $C_3$ & $W(L_3)$ & $W(L_3)$ & $W(L_3)$ & $\ASL(2,3)$ \\
    12 & $C_2$ & $N_{W^+(E_7)}(C_2)$ & $N_{W^+(E_7)}(C_2)$ & $N_{W^+(E_7)}(C_2)$ & $N_{W^+(E_7)}(C_2)/C_2$
    \end{tabular}
\end{center}